\documentclass[11pt,letterpaper]{amsart}

\makeatletter
\def\subsection{\@startsection{subsection}{2}
  \z@{.5\linespacing\@plus.7\linespacing}{.5\linespacing}
  {\normalfont\bfseries}}
\makeatother

\makeatletter
\def\@defaultbiblabelstyle#1{[#1]}
\makeatother

\makeatletter
\def\@setauthors{
  \begingroup
  \def\thanks{\protect\thanks@warning}
  \trivlist
  \centering\footnotesize \@topsep30\p@\relax
  \advance\@topsep by -\baselineskip
  \item\relax
  \author@andify\authors
  \def\\{\protect\linebreak}
  \authors
  \ifx\@empty\contribs
  \else
    ,\penalty-3 \space \@setcontribs
    \@closetoccontribs
  \fi
  \endtrivlist
  \endgroup
}
\def\@settitle{\begin{center}
  \baselineskip14\p@\relax
    \bfseries
  \@title
  \end{center}
}
\makeatother

\usepackage{amsxtra}
\usepackage{mathtools}
\usepackage{extarrows}
\usepackage{upgreek}
\usepackage{color}
\usepackage{xcolor}
\usepackage{enumitem}
\usepackage{caption}
\usepackage[labelformat=simple]{subcaption}

\setlist[enumerate]{label=\upshape(\arabic*)}
\allowdisplaybreaks[4]
\usepackage[colorlinks,allcolors=blue,backref=page]{hyperref}
\renewcommand*{\backrefalt}[4]{%
  \ifcase #1 %
  \else
    {↑#2}%
  \fi}
\usepackage{amsthm}
\usepackage{amssymb}
\usepackage{amsmath}
\numberwithin{equation}{section}
\usepackage{upref}
\usepackage{tikz}
\usetikzlibrary{calc}

\newcommand{\xmiddle}[1]{\;\middle#1\;}

\newtheorem{theorem}{Theorem}[section]
\newtheorem{lemma}[theorem]{Lemma}
\newtheorem{proposition}[theorem]{Proposition}
\newtheorem{corollary}[theorem]{Corollary}
\newtheorem{conjecture}[theorem]{Conjecture}
\newtheorem{question}[theorem]{Question}

\theoremstyle{definition}
\newtheorem{definition}[theorem]{Definition}
\newtheorem{example}[theorem]{Example}

\theoremstyle{remark}
\newtheorem{remark}[theorem]{Remark}

\begin{document}

\title[Asymptotic log-concavity of dominant lower Bruhat intervals]{Asymptotic log-concavity of dominant lower Bruhat intervals via Brunn--Minkowski inequality}


\subjclass[2020]{05E16 (Primary), 05E14, 52A27, 52A38, 51F15 (Secondary)}


\author{Gaston Burrull}
\address{(Gaston Burrull) \newline \indent Beijing International Center for Mathematical Research, Peking University, No.\@~5 Yiheyuan Road, Haidian District, Beijing 100871, China}
\email{gaston(at)bicmr(dot)pku(dot)edu(dot)cn}

\author{Tao Gui}
\address{(Tao Gui) \newline \indent Beijing International Center for Mathematical Research, Peking University, No.\@~5 Yiheyuan Road, Haidian District, Beijing 100871, China}
\email{guitao18(at)mails(dot)ucas(dot)ac(dot)cn}

\author{Hongsheng Hu}
\address{(Hongsheng Hu) \newline \indent School of Mathematics, Hunan University, Changsha 410082, China}
\email{huhongsheng(at)hnu(dot)edu(dot)cn}

\begin{abstract}
Bj\"orner and Ekedahl [Ann.\@~of Math.\@~(2), 170.2(2009), pp.\@ 799--817] pioneered the study of length-counting sequences associated with (parabolic) lower Bruhat intervals in crystallographic Coxeter groups.
In this paper, we study the asymptotic behavior of these sequences in affine Weyl groups.
Let $W$ be an affine Weyl group with corresponding Weyl group~$W_f$ and $\prescript{f}{}{W}$ be the set of minimal representatives for the right cosets $W_f \backslash W$.
Let $t_{\lambda}$ be the translation by a dominant coroot lattice element $\lambda$ and $\prescript{f}{}{b}_i^{t_{\lambda}}$ be the number of elements of length $i$ below $t_\lambda$ in the Bruhat order on~$\prescript{f}{}{W}$, which is the $2i$-dimensional Betti number of a Schubert variety in a certain affine Grassmannian.
We show that the sequence $(\prescript{f}{}{b}_i^{t_{\lambda}})_i$ is ``asymptotically log-concave'' in the following sense:
The sequence of discrete measures $(\mathfrak{m}_k)_k$ constructed from the $k$-fold dilated sequence $(\prescript{f}{}{b}_i^{t_{k\lambda}})_i$, as $k$ tends to infinity, converges weakly to a continuous measure obtained from a polytope $P^\lambda$.
Moreover, the sequence of step functions $(S_k)_k$ of $(\prescript{f}{}{b}_i^{t_{k\lambda}})_i$ converges uniformly to the density function of this continuous measure.
By Brunn--Minkowski inequality, this density is log-concave.
\end{abstract}

\maketitle

\setcounter{tocdepth}{2}
\tableofcontents

\section{Introduction}

\subsection{Background}
Schubert varieties are certain subvarieties of flag varieties.
They are projective algebraic varieties that are usually singular.
Studying classes of Schubert varieties in the cohomology ring of the flag variety leads to important results in enumerative geometry (the classical ``Schubert calculus''), while the study of their intersection cohomology plays a fundamental role in the representation theory of Lie-theoretic objects (the ``Kazhdan--Lusztig theory'').
Following Bj\"orner and Ekedahl~\cite{bjorner2009shape}, we are interested in the behavior of the Betti numbers of Schubert varieties.

More precisely, consider a complex Kac--Moody group~$G$ (a gentle introduction of Kac--Moody groups and their flag varieties can be found in~\cite{kumar2002kacmoody}) with Borel subgroup~$B$ and maximal torus $T$.
The corresponding Weyl group~$W$ has the structure of a crystallographic Coxeter system $(W, S)$, where $S$ is the generating set of simple reflections and we denote by $\ell\colon W \rightarrow \mathbb{N}$ the length function.
For any $J \subset S$, there is a parabolic subgroup $W_J:=\langle s \in J\rangle$ of $W$ with corresponding subgroup $P_J:= B W_J B$ of $G$.

The quotient~$P_J \backslash G$ is a projective (ind-)variety called the \emph{generalized (partial) flag variety}.
We have Bruhat decomposition 
\begin{equation*}
    P_J \backslash G=\bigsqcup_{w  \in \prescript{J}{}{W}} P_J \backslash P_J w B,
\end{equation*}
where $\prescript{J}{}{W}$ is the set of minimal representatives for the right cosets $W_J \backslash W$.
The component $C_w:=P_J \backslash P_J w B$ is called the \emph{Schubert cell} associated with $w \in \prescript{J}{}{W}$.
Topologically, $C_w$ is an $\ell(w)$-dimensional affine space $\mathbb{A}^{\ell(w)}$.
Its closure  $X_w : = \overline{C_w}$, called the \emph{Schubert variety} associated with $w$, is a finite-dimensional projective irreducible subvariety of $P_J \backslash G$.
There is a partial order $\leq$ on~$\prescript{J}{}{W}$ called the \emph{Bruhat--Chevalley order}, defined by $v \leq w$ if $C_v \subseteq X_w$.
Furthermore, we have the decomposition
\begin{equation} \label{eq-Bruhat-dec}
    X_w=\bigsqcup_{ v \in \prescript{J}{}{W}, v \leq w} P_J \backslash P_J v B.
\end{equation}

An interesting question is: 
\begin{question} \label{question how many cells}
How many complex $i$-dimensional cells occur in the decomposition \eqref{eq-Bruhat-dec} of $X_w$?
\end{question} 
Let us denote this number by $\prescript{J}{}{b}_i^w$.
Equation \eqref{eq-Bruhat-dec} gives the equality
\begin{equation}\label{eq-definition-betti-number}
\prescript{J}{}{b}_i^w=\operatorname{Card}\left\{v \in \prescript{J}{}{W} \xmiddle| v \leq w \text { and } \ell(v)=i\right\}, 
\end{equation}
which also equals the $2i$-dimensional Betti number of $X_w$ (the odd dimensional Betti numbers of $X_w$ are 0).

Question~\ref{question how many cells} is difficult to answer in general.
If $X_w$ is smooth, the Poincar\'e duality for ordinary cohomology implies that $\prescript{J}{}{b}_i^w=\prescript{J}{}{b}_{\ell(w)-i}^{w}$.
While the hard Lefschetz theorem for ordinary cohomology of smooth projective varieties implies that the sequence $(\prescript{J}{}{b}_i^w)_i$ is \emph{unimodal}.
This means that for some $0 \leq k \leq \ell(w)$ (here  $k=\left\lceil\ell(w)/2\right\rceil$), we have
\begin{equation*}
\prescript{J}{}{b}_0^w \leq \prescript{J}{}{b}_1^w \leq \cdots \leq \prescript{J}{}{b}_{k-1}^w \leq \prescript{J}{}{b}_{k}^w \geq \prescript{J}{}{b}_{k+1}^w \geq \cdots \geq \prescript{J}{}{b}_{\ell(w)}^w.
\end{equation*}
But $X_w$ is singular in general, hence Poincar\'e duality and hard Lefschetz theorem for ordinary cohomology usually fail.
By means of deep results in Hodge theory such as the purity theorem and hard Lefschetz theorem for the intersection cohomology~\cite{beilinson1982faisceaux}, Bj{\"o}rner and Ekedahl~\cite{bjorner2009shape} showed that, for every $w\in \prescript{J}{}{W}$, the sequence $(\prescript{J}{}{b}_i^w)_i$ satisfies the following two sets of inequalities: 
\begin{gather}
    \prescript{J}{}{b}_i^w \leq \prescript{J}{}{b}_{\ell(w)-i}^{w} \text{ for } i \leq \frac{\ell(w)}{2}, \text{ and } 
    \prescript{J}{}{b}_0^w \leq \prescript{J}{}{b}_1^w \leq \cdots \leq \prescript{J}{}{b}_{\left\lceil\ell(w)/2\right\rceil}^{w}.\label{eq-1-3}
\end{gather}
The first set of inequalities is rephrased as the sequence being \emph{top-heavy}, while the second is the fact that the sequence is weakly increasing in the ``lower half part''.

Some variants of Question~\ref{question how many cells} have been studied.
By Equation  \eqref{eq-definition-betti-number}, one can formulate an analog of Question~\ref{question how many cells} for general Coxeter groups.
Using Soergel bimodules and their Hodge theory established by Elias and Williamson in their foundational paper~\cite{elias2014hodge}, it is proven that the analog of the inequalities \eqref{eq-1-3} holds for general Coxeter groups in the non-parabolic case (that is, $J=\emptyset$, see~\cite{Pa25}).
For the parabolic case, we believe that a proof of these inequalities should follow from the Hodge theory of singular Soergel bimodules, see~\cite{patimo2022singular} and~\cite[Remark 3.18]{Pa25}.
On the other hand, in the context of Schubert varieties of hyperplane arrangements, 
Huh and Wang~\cite{huh2017enumeration} proved Dowling and Wilson's ``Top-Heavy conjecture'' for matroids realizable over some field using Hodge-theoretic ideas analogous to the ones in Bj\"orner and Ekedahl~\cite{bjorner2009shape}.
Later, Braden, Huh, Matherne, Proudfoot, and Wang~\cite{braden2020singular} achieved the remarkable task of generalizing this result to every matroid by establishing Hodge theory for the intersection cohomology of matroids.

Despite these great achievements, the unimodality of $(\prescript{J}{}{b}_i^w)_i$  for the ``upper half part'' remains an interesting open problem.
To the best of our knowledge, there is no partial result yet.
However, conjectures related to this problem have been made.
Before we get into these, let us recall that a sequence $(a_0, a_1, \ldots, a_n)$ of positive real numbers is said to be \emph{log-concave} if 
\begin{equation*}
a_{i-1} a_{i+1} \leq a_i^2 \text{ for all $0<i<n$}.
\end{equation*}
This notion is stronger: log-concave sequences are always unimodal.
Regarding log-concavity of Bruhat intervals, Brenti conjectured the following:

\begin{conjecture}[{\cite[Conjecture 2.11]{brentisome}}] \label{conj-Brenti-log-concavity}
Let $W$ be a Weyl group, and $u, v \in W$.
The sequence $(b_i^{[u, v]})_i$  is log-concave, where \[b_i^{[u, v]}=\operatorname{Card}\left\{w \in W \, | \, u \leq w \leq v, \ell(w)=i\right\}.\]
\end{conjecture}

The parallel conjecture in the setting of matroids and Schubert varieties of hyperplane arrangements is known as ``Rota's unimodality conjecture"~\cite{rota1970combinatorial}.
Its simplest non-trivial case for rank-$3$ realizable matroids is referred to as the ``points-lines-planes conjecture'' and it remains open.

The parabolic analog of Conjecture~\ref{conj-Brenti-log-concavity} does not hold.
For example, consider the symmetric group $W = \mathfrak{S}_{12}$ with generating set $S = \{s_1, \dots, s_{11}\}$ where $s_i=(i,i+1)$ is the corresponding adjacent transposition and $J := S \setminus \{s_4\}$.
According to Stanton~\cite{stanton1990unimodality}, Young's lattice for the partition $(8,8,4,4)$ defines a non-unimodal sequence 
\begin{equation*}
(1,1,2,3,5,6,9,11,15,17,21,23,27,28,\mathbf{31,30,31},27,24,18,14,8,5,2,1).
\end{equation*}
This is the sequence $(\prescript{J}{}{b}_i^w)_i$ corresponding to a parabolic lower Bruhat interval of $\prescript{J}{}{W}$, which gives the Betti numbers of the Schubert variety $X_{(8,8,4,4)}$ inside the Grassmannian $\operatorname{Gr}(4,12)$.

\subsection{Our setting}

Let $W = \mathbb{Z} \Phi\spcheck \rtimes W_f$ be the affine Weyl group with Weyl group~$W_f$ and root system $\Phi$ of rank $r$.
Let $(E, (-|-) )$ be the $r$-dimensional Euclidean space where $\Phi$ lives.
Let $\prescript{f}{}{W}$ (resp.\@~$W^f$) be the set of minimal representatives for the right cosets $W_f \backslash W$ (resp.\@~left cosets $W / W_f$).
We denote by $C_+$ the dominant Weyl chamber.
Let $\lambda \in \mathbb{Z} \Phi \spcheck \cap \overline{C_+}$ be a dominant coroot lattice element, and $t_\lambda \in W$ be the translation by $\lambda$.
Let 
\begin{equation*}
\prescript{f}{}{[e, t_\lambda]}:= \{ w \in \prescript{f}{}{W} \mid w \le t_\lambda \}
\end{equation*}
be the \emph{dominant lower Bruhat interval} (see Section~\ref{sec-pre} for details.)
The inverse map $w \mapsto w^{-1}$ preserves the length and the Bruhat--Chevalley order.
It restricts to a bijection from $\prescript{f}{}{[e, t_\lambda]}$ to $[e, t_{-\lambda}]^f:= \{ w \in W^f \mid w \le t_{-\lambda} \}$, where the second Bruhat interval corresponds to a (spherical) Schubert variety in the corresponding affine Grassmannian as we will explain below.

\subsubsection{Spherical Schubert varieties in the affine Grassmannian} \label{subsec-spherical}
We recommend~\cite[Chapter 9]{achar2021perverse}, \cite{kumar2002kacmoody}, and~\cite{zhu2016introduction} for a more in-depth introduction to the affine Grassmannians.

Let $G$ be the semisimple and simply connected complex algebraic group with coroot lattice $\mathbb{Z} \Phi\spcheck$.
Let $B$ be a Borel subgroup of $G$ and $T \subset B$ be a maximal torus.
Let $F=\mathbb{C}((t))$ with ring of integers $\mathcal{O}=\mathbb{C}[[t]]$.
The algebraic loop group~$G(F)$ has a maximal compact group $K=G(\mathcal{O})$.
Consider the evaluation at $t=0$ map
\begin{equation*}
\mathrm{ev}\colon K=G(\mathcal{O}) \xrightarrow{t=0} G(\mathbb{C})=G,
\end{equation*}
and consider the Iwahori subgroup
\begin{equation*}
I:=\mathrm{ev}^{-1}\left(B\right).
\end{equation*}
The \emph{affine Grassmannian} associated to $G$ is
\begin{equation*}
\mathcal{G} r:=G(F)/K,
\end{equation*}
which is an ind-projective variety.

The affine Grassmannian $
\mathcal{G} r$ has the following decomposition into $K$-orbits: 
\begin{equation*}
\mathcal{G} r=\bigsqcup_{\lambda \in \mathbb{Z} \Phi \spcheck \cap \overline{C_+}} \mathcal{G} r_\lambda  \text { where }  \mathcal{G} r_\lambda:=K t^\lambda K / K.
\end{equation*}
We regard $t^\lambda$ as a point in~$G(\mathbb{C}\left[t^{\pm 1}\right]) \subset G(F)$.
Let $N_{G(F)}(T)$ be the normalizer of $T$ in~$G(F)$.
The isomorphism $N_{G(F)}(T)/T \rightarrow W$ sends $t^\lambda$ to the anti-dominant translation $t_{-\lambda}\in W^f$~\cite[Section 4.1]{LR16}.
The corresponding \emph{spherical Schubert variety} is 
\begin{equation*}
\overline{\mathcal{G} r_\lambda}=\bigsqcup_{\mu \in \mathbb{Z} \Phi \spcheck \cap \overline{C_+},\ \mu \preceq \lambda}\mathcal{G} r_\mu,
\end{equation*}
where $\mu \preceq \lambda$ if and only if $\lambda-\mu$ is a sum of positive coroots.
It plays an important role in geometric representation theory since its intersection cohomology carries the structure of the simple $G\spcheck$-module of the Langlands dual group~$G\spcheck$ with highest weight $\lambda$~\cite{lusztig1983singularities} via the celebrated geometric Satake equivalence~\cite{mirkovic2007geometric}, a cornerstone of the geometric Langlands program.

On the other hand, the $I$-orbits on~$\mathcal{G} r$ give the Bruhat decomposition
\begin{equation*}
\mathcal{G} r=\bigsqcup_{x \in W^f} I x K / K.
\end{equation*}
We have the Schubert variety 
\begin{equation*}
X_y:=\overline{IyK/K}=\bigsqcup_{x \in W^f,\ x \leq y}IxK/K.
\end{equation*}
Each $I$-orbit $IxK/K$ is a Schubert cell which is isomorphic to the affine space of dimension $\ell(x)$.

By considering the $I$-orbits, we have~\cite[Section 4.3]{LR16}
\begin{equation*}
\overline{\mathcal{G} r_\lambda}=X_{t_{-\lambda}},
\end{equation*}
that is, the spherical Schubert variety $\overline{\mathcal{G} r_\lambda}$ is a special Schubert variety $X_{t_{-\lambda}}$ indexed by the translation $t_{-\lambda} \in W^f$, where $-\lambda$ is anti-dominant.
After taking inverses, we can see that the topological information of $\overline{\mathcal{G} r_\lambda}$ is encoded in the dominant lower Bruhat interval $\prescript{f}{}{[e, t_{\lambda}]}$.

\subsection{Main results}
In this paper, we prove that the length-counting sequence 
\begin{equation*}
\prescript{f}{}{b}_i^{t_\lambda} : = \operatorname{Card} \bigl\{w \in \prescript{f}{}{[e, t_\lambda]} \bigm| \ell(w) = i \bigr\}, \quad i = 0, 1, \dots, \ell(t_\lambda),
\end{equation*}
of $\prescript{f}{}{[e, t_\lambda]}$, which is the sequence of Betti numbers of the spherical Schubert variety $\overline{\mathcal{G} r_{-w_0 \lambda}}=X_{t_{-\lambda}}$ is \emph{asymptotically log-concave}.
This statement consists of two parts: 
\begin{itemize}
    \item Theorem~\ref{thm-main-in-the-introduction}, which informally speaking, tells us that the ``shape'' of the length-counting sequences of the $k$-fold dilated intervals $\prescript{f}{}{[e, t_{k\lambda}]}$ converges to a continuous function as $k$ tends to infinity (see Figure~\ref{fig: betti convergence} for an illustration).
    \item Corollary~\ref{cor-den}, which states that this continuous function is log-concave.
\end{itemize}

\subsubsection{Asymptotic convergence} 
Let $\lambda \in \mathbb{Z} \Phi\spcheck \cap \overline{C_+}$.
We define the convex polytope
\[P^\lambda := \operatorname{Conv} \{w \lambda \mid w \in W_f\} \cap \overline{C_+} \subset E,\]
where $\operatorname{Conv}$ denotes the convex hull of a set (see Section~\ref{subsec-def-Plambda} for details.)

Let $\operatorname{ht}\colon P^\lambda \to \mathbb{R}$ be the height function $\operatorname{ht}(x) := (2\rho|x)$, where $\rho$ is the half-sum of positive roots.
In particular, $\operatorname{ht} (\lambda) = \ell(t_\lambda)$.
We denote by $\operatorname{Vol}_r$ the Lebesgue measure on~$E$ and by $\mathrm{ht}_* \mathrm{Vol}_r$ the corresponding push-forward measure on~$\mathbb{R}$.
We also denote by $\operatorname{Vol}_{r-1}$ the Lebesgue measure on affine hyperplanes of $E$.
Then, the density function of $\mathrm{ht}_* \mathrm{Vol}_r$, which is 
\[g(z) = \frac{1}{\lVert 2 \rho \rVert} \operatorname{Vol}_{r-1} (\operatorname{ht}^{-1} (z)), \quad z \in \mathbb{R},\]
evaluates volumes of hyperplane sections of the polytope $P^\lambda$ up to a scalar (see Section~\ref{sec-main} for details.)

Let $\delta_{z}$ denote the Dirac measure (that is, point mass) at the point $z \in \mathbb{R}$.
For any positive integer $k$, we define the discrete measure $\mathfrak{m}_k$ supported on~$[0, \ell(t_\lambda)]$ by
\begin{equation} \label{measure}
    \mathfrak{m}_k := \frac{1}{k^r} \sum_{0 \le i \le k\ell(t_{\lambda})} \prescript{f}{}{b}_i^{t_{k\lambda}} \delta_{\frac{i}{k}}.
\end{equation}
Intuitively, $\mathfrak{m}_k$ distributes the sequence $(\prescript{f}{}{b}_i^{t_{k\lambda}})_i$ evenly on the interval $[0,\ell(t_\lambda)]$.

We also define the step function $S_k\colon [0, \ell(t_\lambda)] \to \mathbb{R}$ by
\[S_k(z) := \frac{1}{k^{r-1}} \prescript{f}{}{b}_{i}^{t_{k\lambda}}, \text{ whenever } z \in \left[\frac{i}{k}, \frac{i+1}{k} \right).\]
The function $S_k$ records the numbers $(\prescript{f}{}{b}_i^{t_{k\lambda}})_i$ and behaves like the ``density function'' of $\mathfrak{m}_k$ (see Section~\ref{sec-main-unifm-convg} for details.)
The following is our main theorem.

\begin{theorem} \label{thm-main-in-the-introduction} 
Let $\operatorname{Vol}_r(A_+)$ be the volume of the fundamental alcove $A_+$.
\begin{enumerate} 
    \item \label{thm-main-1-in-the-introduction} {\em (The weak convergence of $(\mathfrak{m}_k)_k$).} The sequence of measures $(\mathfrak{m}_k)_k$, as $k$ tends to infinity, converges weakly to the measure
    \begin{equation*}
        \frac{1}{\operatorname{Vol}_r(A_+)} \mathrm{ht}_* \mathrm{Vol}_r.
    \end{equation*}
    \item \label{thm-main-2-in-the-introduction} {\em (The uniform convergence of $(S_k)_k$).} The sequence of functions $(S_k)_k$, as $k$ tends to infinity, converges uniformly to the function $\frac{1}{\operatorname{Vol}_r(A_+)}g$.
\end{enumerate}
\end{theorem}

This convergence result also holds for general ``dominant'' elements, that is, elements in~$\prescript{f}{}{W}$.
See Section~\ref{sec-general}.

\subsubsection{The dominant lattice formula} \label{subsec-intro-dominant}

We define the \emph{Poincar\'e polynomial} $\pi^{t_\lambda} (q)$ of the sequence $(\prescript{f}{}{b}_i^{t_\lambda})_i$ by
\[\pi^{t_\lambda}(q)  := \sum_{0 \le i \le \ell(t_\lambda)} \prescript{f}{}{b}_i^{t_\lambda} q^i,\]
which is the Poincar\'e polynomial of the singular cohomology of the spherical Schubert variety $\overline{\mathcal{G} r_{\lambda}}=X_{t_{-\lambda}}$ as we explained in the previous section.
Moreover, for $\mu \in \mathbb{Z} \Phi\spcheck$, we write
\[\prescript{\mu}{}{\pi}_f (q) := \sum_{w \in \prescript{\mu}{}{W_f}} q^{\ell(w)},\]
where $\prescript{\mu}{}{W_f}$ is the set of minimal representatives for the right cosets $W_{f, \mu} \backslash W_f$, and $W_{f, \mu}$ is the stabilizer of $\mu$ in~$W_f$.

One of the key ingredients to prove Theorem~\ref{thm-main-in-the-introduction} is the following formula, which is one of our most important results (see Figure~\ref{fig: similarity} for an illustration).

\begin{theorem}[Dominant lattice formula]\label{thm-lattice-formula-in-the-introduction} Let $\lambda \in \mathbb{Z} \Phi\spcheck \cap \overline{C_+}$.
We have
    \begin{equation*} 
      \pi^{t_\lambda} (q) = \sum_{\mu \in P^\lambda \cap \mathbb{Z} \Phi\spcheck} q^{(2\rho | \mu)} \cdot \prescript{\mu}{}{\pi_f}(q^{-1}).
    \end{equation*}
\end{theorem}

This formula is a powerful way---in terms of computer efficiency---to obtain the generating function of the sequence $(\prescript{f}{}{b}_i^{t_{\lambda}})_i$ from the set of lattice points in~$P^\lambda$.
The evaluation at $q=1$ gives a formula for $\operatorname{Card}(\prescript{f}{}{[e,t_{\lambda}]})$, while the coefficient of $q^i$ gives a formula for $\prescript{f}{}{b}^{t_{\lambda}}_i$.

\subsubsection{Log-concavity and a conjecture on unimodality} 
Recall that a real function $f$ defined on a convex subset $U$ of a vector space $V$ is called \emph{concave} if
$f\left(\frac{a+b}{2}\right) \geq \frac{f(a)+f(b)}{2}$ for any $a,b \in U$.
The following result follows from the classical Brunn--Minkowski inequality.

\begin{theorem}[Brunn--Minkowski, see {\cite[p.\@~270]{okounkov1997log}}]
Let $L_1$ be a real vector space and let $M \subset L_1$ be a convex body.
Let $p\colon L_1 \rightarrow L_2$ be a linear transformation.
Then
\[z \mapsto \Bigl(\operatorname{Vol}\bigl(p^{-1}(z) \cap M\bigr)\Bigr)^{1 /(\dim M - \dim p(M))}\]
is a concave function on~$p(M)$.
\end{theorem}
Applying the above theorem to the map $\mathrm{ht}\colon P^\lambda \rightarrow \mathbb{R}$ and composing with the logarithm function (which is concave), we immediately have the following corollary.
\begin{corollary} \label{cor-den} 
    The density function $g$ of the measure $\mathrm{ht}_*\mathrm{Vol}_r$ is log-concave.
    That is, $\log g$ is a concave function.
\end{corollary}

Note that even though $\log g(z)=-\infty$ if $z \notin [0, \ell(t_\lambda)]$, the notion of concavity still makes sense.

\begin{remark} \label{rmk-log}
Asymptotic log-concavity turns out to be an interesting fact since the sequence $(\prescript{f}{}{b}_i^{t_{\lambda}})_i$ can fail to be log-concave itself.
For example, from the step function in Figure~\ref{fig:poly1}, we observe that the sequence contains the consecutive terms $(4, 4, 5)$.
\end{remark}

In spite of Remark~\ref{rmk-log}, the unimodality of the sequence $(\prescript{f}{}{b}_i^{t_{\lambda}})_i$ might hold without the need to take limits, and it has been tested in many small rank cases (see below).
We state it as a conjecture.
\begin{conjecture} \label{conj-main}
The sequence $(\prescript{f}{}{b}_i^{t_{\lambda}})_i$ is unimodal.
\end{conjecture}

Before listing the computational evidence of the conjecture above, let us define $\lambda$ to be \emph{below} a vector $v$ if, under the basis of fundamental coweights, each coefficient of $\lambda$ is less than or equal to the corresponding coefficient of $v$.

We verified Conjecture~\ref{conj-main} with the software SageMath and the help of our dominant lattice formula (which considerably reduced the time of computations) for several $\lambda\in \mathbb{Z}\Phi^{\spcheck}$ in cases where the rank of $\Phi$ is not too big: 
\begin{itemize}
    \item $\operatorname{rank}(\Phi)\leq 3$: for $\Phi$ of each type, we tested hundreds of $\lambda$, and we omit the details.
    \item $\operatorname{rank}(\Phi)=4$:
    \begin{itemize}[label=$\circ$]
        \item \textit{$\Phi$ of type $A$.} Tested for the $30$ intervals with $\lambda$ below $[2,3,3,2]$.
        \item \textit{$\Phi$ of type $B$.} Tested for the $54$ intervals with $\lambda$ below $[2,2,3,2]$.
        \item \textit{$\Phi$ of type $C$.} Tested for the $48$ intervals with $\lambda$ below $[2,2,3,2]$.
        \item \textit{$\Phi$ of type $D$.} Tested for the $30$ intervals with $\lambda$ below $[2,2,3,2]$.
        \item \textit{$\Phi$ of type $F$.} Tested for the $36$ intervals with $\lambda$ below $[1,2,2,1]$.
    \end{itemize}
    \item $\operatorname{rank}(\Phi)=5$: \textit{$\Phi$ of type A.} Tested for the $24$ intervals with $\lambda$ below $[1,2,3,2,1]$.
\end{itemize}
We did not test other $\Phi$'s due to the limitation of our computer resources.

\subsection{An example} \label{subsec-example} 

\begin{example}[Type $C_3$]\label{Example1}
    Let $W$ be the affine Weyl group associated with the root system $\Phi$ of type $C_3$ and simple roots $\Delta =\{\alpha_1, \alpha_2, \alpha_3\}$.
    Then, $r=3$.
    Following~\cite[Plate III]{bourbaki1968}, we write $\alpha_1=\epsilon_1-\epsilon_2$, $\alpha_2=\epsilon_2-\epsilon_3$, and $\alpha_3=2\epsilon_3$.
    Let 
    \begin{equation*}
        \lambda=3\alpha_1\spcheck+6\alpha_2\spcheck+7\alpha_3\spcheck.
    \end{equation*}
    We have that $\operatorname{ht}(\lambda)=32$.
    For convenience, let $(a,b,c)_\Phi := a\alpha_1\spcheck+b\alpha_2\spcheck+c\alpha_3\spcheck$.
    The polytope $P^\lambda$ is the convex polyhedron with six vertices given by
    \begin{equation*}
        \{(0,0,0)_\Phi, (3,3,3)_\Phi, (3,5,7)_\Phi, (3,6,6)_\Phi, (7/3,14/3,7)_\Phi, (3,6,7)_\Phi\},
    \end{equation*}
    which is an example of a non-lattice polytope.
    Since $\rho=(3,5,3)_\Phi$, we get $\lVert\rho\rVert=\sqrt{14}$.
    By direct computations, we have $\operatorname{Vol}_3(A_+)=1/48$.
    In view of Theorem~\ref{thm-main-in-the-introduction}\ref{thm-main-2-in-the-introduction}, the only missing ingredient to compute the limit function is to determine $\operatorname{Vol}_{2} (\operatorname{ht}^{-1}(z))$.
    By the theory of convex polytopes, this function is a piecewise quadratic polynomial, and its exact form can be obtained by Lagrange interpolation.
    We omit details and give a graph of the function $g/\operatorname{Vol}_3(A_+)$ in Figure~\ref{fig:poly367}.
    
    \begin{figure}[ht] 
        \begin{subfigure}[t]{.48\linewidth}
            \centering
            \includegraphics[scale=0.4]{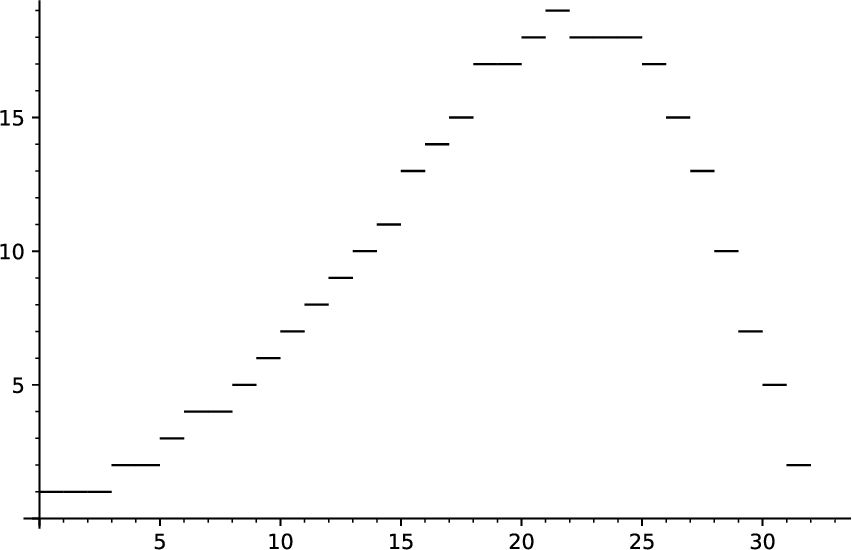}
            \caption{Graph of $S_1$.}
            \label{fig:poly1}
        \end{subfigure}
        \begin{subfigure}[t]{.48\linewidth}
            \centering
            \includegraphics[scale=0.4]{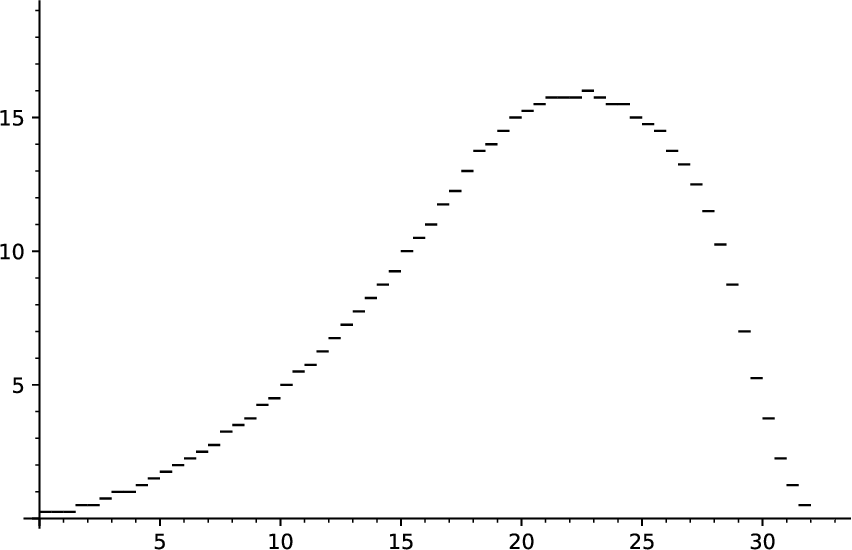}
            \caption{Graph of $S_2$.}
            \label{fig:poly2}
        \end{subfigure}
        \smallbreak
        \begin{subfigure}[t]{.48\linewidth}
            \centering
            \includegraphics[scale=0.4]{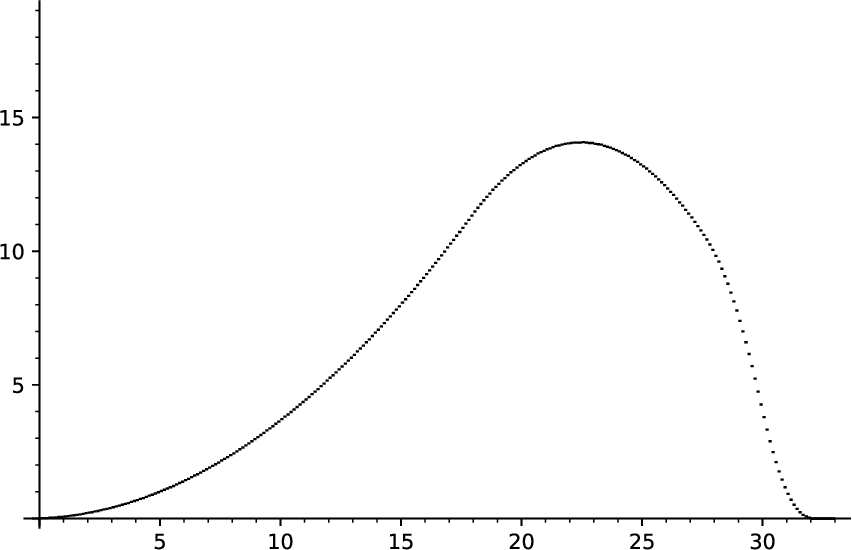}
            \caption{Graph of $S_8$.}
            \label{fig:poly8}
        \end{subfigure}
        \begin{subfigure}[t]{.48\linewidth}
            \centering
            \includegraphics[scale=0.4]{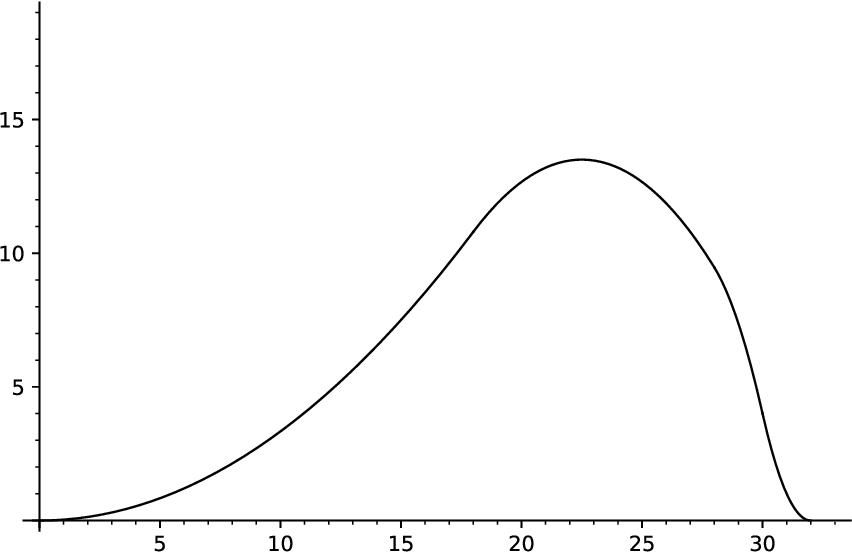}
            \caption{Graph of $g/\operatorname{Vol}_r(A_+)$.}
            \label{fig:poly367}
        \end{subfigure}
        \caption{In the affine Weyl group~$W$ of affine type $C_3$, we consider $\lambda=3\alpha_1\spcheck+6\alpha_2\spcheck+7\alpha_3\spcheck$. These pictures illustrate how the sequence of step functions $S_k\colon [0,\ell(t_\lambda)]\to \mathbb{R}$ converges uniformly to a continuous function.}
        \label{fig: betti convergence}
    \end{figure}
    
    We can use Theorem~\ref{thm-main-in-the-introduction}\ref{thm-main-2-in-the-introduction} to give quick estimates of the terms in the sequence $(\prescript{f}{}{b}_{i}^{t_{k\lambda}})_{i}$ when $k$ is big.
    For instance, when $k=8$, the value of $\prescript{f}{}{b}_{196}^{t_{8\lambda}}$ is virtually impossible to obtain in a computer directly from definitions.
    Let us take $z=24.5(=196/8)$.
    By our theorem, we have
    \begin{equation*}
        S_8(24.5)=\frac{1}{8^2} \prescript{f}{}{b}_{196}^{t_{8\lambda}}\sim 48g(24.5) = \frac{389}{30},
    \end{equation*}
    which gives $\prescript{f}{}{b}_{196}^{t_{8\lambda}}\sim 829.87$.
    
    On the other hand, our dominant lattice formula gives the exact values of the terms in the sequence $(\prescript{f}{}{b}_{i}^{t_{k\lambda}})_{i}$ (which takes a considerable time to get in a computer.) 
    In particular, we have $\prescript{f}{}{b}_{196}^{t_{8\lambda}}= 863$.
    Our quick estimate of $829.87$ was off by $3.84\%$.
    See Figure~\ref{fig: betti convergence} for the graphs of the step functions $S_1$, $S_2$, and $S_8$.
\end{example}

\subsection{Connections with other works} \label{subsec-connection}

\subsubsection{Asymptotic behavior of weight multiplicities} \label{subsec-symplectic}

Let $K$ be a compact connected Lie group with a maximal torus $T$, let $W_f$ be the corresponding Weyl group, and let $\lambda$ be a dominant weight of $T$.
In~\cite{heckman1982projections}, Heckman used the weight multiplicities $\operatorname{dim} V(k \lambda)_\mu$ of the irreducible representation of $K$ with highest weight $k\lambda$ to construct a discrete measure
\begin{equation*}
    \frac{\sum_\mu \operatorname{dim} V(k \lambda)_\mu \delta_{\frac{\mu}{k}}}{\sum_\mu \operatorname{dim} V(k \lambda)_\mu},
\end{equation*}
supported on the weight polytope $\operatorname{Conv}\{w \lambda \mid w \in W_f\}$.
He proved that this sequence of discrete measures, as $k$ tends to infinity, converges weakly to the push-forward of the Liouville measure of the coadjoint orbit of $\lambda$ under the moment map.
The limit measure is the \emph{Duistermaat--Heckman measure} associated with $\lambda$.
Its density function---the \emph{Duistermaat--Heckman function}---is a piecewise polynomial~\cite{duistermaat1982variation}, and Graham proved that it is log-concave via Hodge--Riemann inequalities~\cite{graham1996logarithmic}.

Later, Okounkov~\cite{okounkov1996brunn} introduced the Newton--Okounkov bodies to prove that, in a similar weak limit sense, for any representation $V$ of a reductive group~$G$, the multiplicities of irreducible $G$-modules in the homogeneous coordinate ring of a $G$-stable irreducible subvariety of $\mathbb{P}(V)$ are log-concave.
In a sequel paper~\cite{okounkov2003would}, Okounkov pointed out the importance of log-concavity and related it to the properties of entropy in statistical physics.

Our construction in Equation \eqref{measure} and the formulation of Theorem~\ref{thm-main-in-the-introduction}\ref{thm-main-1-in-the-introduction} draw inspiration from the works of Heckman~\cite{heckman1982projections} and Okounkov~\cite{okounkov1996brunn}.
Actually, Theorem~\ref{thm-main-in-the-introduction}\ref{thm-main-1-in-the-introduction} has the following representation-theoretic interpretation:
Let $G$ be the semisimple and simply connected complex algebraic group and $B$ be a Borel subgroup of $G$ as in Section~\ref{subsec-spherical}.
Let $\check{\mathfrak{g}}$ be the complex semisimple Lie algebra of the Langlands dual group~$G\spcheck$.
Consider a standard principal nilpotent element $e \in \check{\mathfrak{g}}$, and let $X=\left(G\spcheck\right)_e$ denote the stabilizer of $e$ in~$G\spcheck$ under the adjoint action.
The subgroup~$X$ is abelian and has 
dimension equal to the rank of $G$.
For instance, in type $A$, $X$ is the group of upper-triangular unipotent Toeplitz matrices~\cite{rietsch2003totally}.
Let $\check{\mathfrak{g}}^e$ be the abelian Lie algebra of $X$.
By Ginzburg's theorem~\cite[Theorem 1.5]{ginzburg1998loopgrassmanniancohomologyprincipal}, the canonical map 
\begin{equation*}
    H^{\bullet}(\overline{\mathcal{G} r_\lambda}, \mathbb{C}) \rightarrow \mathit{I H}^{\bullet}(\overline{\mathcal{G} r_\lambda}, \mathbb{C})
\end{equation*}
is injective and corresponds, under the geometric Satake equivalence, to the inclusion 
\begin{equation*}
    U \rightarrow V_{\lambda},
\end{equation*}
where $V_{\lambda}$ is the simple $G\spcheck$-module with highest weight $\lambda$ and $U$ is the cyclic submodule of the enveloping algebra $U(\check{\mathfrak{g}}^e)$ generated by a lowest weight vector $v_{\lambda}$ of $V_{\lambda}$.
Note that $\check{\mathfrak{g}}^e$ is stable under the adjoint action of a regular semisimple element and the weight decomposition gives a grading on~$U(\check{\mathfrak{g}}^e)$ and hence on~$U=U(\check{\mathfrak{g}}^e)/\mathrm{Ann}_{U(\check{\mathfrak{g}}^e)}v_{\lambda}$, which corresponds to the cohomological grading on~$H^{\bullet}(\overline{\mathcal{G} r_\lambda}, \mathbb{C})$.
Therefore, Theorem~\ref{thm-main-in-the-introduction}\ref{thm-main-1-in-the-introduction} explicitly computes the Duistermaat--Heckman measure of the one-dimensional torus $\mathbb{C}^\times$ generated by the regular semisimple element acting on~$U$.
The corresponding Duistermaat--Heckman function is then given by the 
density function $g$ of $\mathrm{ht}_* \mathrm{Vol}_r$ up to a constant.

Nevertheless, our framework and proof techniques differ significantly from Heckman's work since there is no symplectic manifold related to the module $U$.
Hence, it is not immediately clear that our original cell-counting problem has such a fundamental connection to the geometry of a convex polytope. 
Furthermore, Theorem~\ref{thm-main-in-the-introduction}\ref{thm-main-2-in-the-introduction} is novel compared to the results of Heckman and Okounkov, and it is stronger than Theorem~\ref{thm-main-in-the-introduction}\ref{thm-main-1-in-the-introduction} (see Remark~\ref{rmk-deduce}).

\subsubsection{Ehrhart's theory} \label{subsec-Ehrhart}
For an $r$-dimensional lattice polytope $P$ (that is, all vertices of $P$ are points of a given lattice $L$), the \emph{Ehrhart polynomial}~\cite{ehrhart1962polyedres} $E(P, k)$ is a rational polynomial in~$k$ that counts the number of lattice points in the $k$-fold dilation $kP$.
That is, there exist rational numbers $E_0(P), \ldots, E_r(P)$, such that
\begin{equation*}
    E(P, k) := \lvert L \cap k P \rvert = E_r(P) k^r+E_{r-1}(P) k^{r-1}+\cdots+E_0(P)
\end{equation*}
for all positive integers $k$.
The leading coefficient, $E_r(P)$, is equal to the $r$-dimensional volume $\operatorname{Vol}_r(P)$ of $P$, divided by the volume $d(L)$ of the fundamental region of the lattice $L$.
This implies that 
\begin{equation} \label{eq-asym}
    \operatorname{Vol}_r(P)=\lim _{k \rightarrow \infty} \frac{d(L)  \cdot \lvert L \cap k P \rvert}{k^r}.
\end{equation}
If $X$ is the toric variety associated with $P$, then $P$ defines an ample line bundle on~$X$. The Ehrhart polynomial of $P$ coincides with the Hilbert polynomial of this line bundle, and the above asymptotic result \eqref{eq-asym} is a consequence of the asymptotic Riemann--Roch theorem~\cite[\href{https://stacks.math.columbia.edu/tag/0BJ8}{Tag 0BJ8}]{stacks}.

The problem in our work is analogous to the one in Ehrhart's theory.
However, we count alcoves in each length rather than all lattice points in the polytope $P^\lambda$.
When $P^\lambda$ is not a lattice polytope, it has no Ehrhart polynomial, and the proof of Theorem~\ref{thm-main-in-the-introduction}\ref{thm-main-2-in-the-introduction} is technical (see Section~\ref{subsec-outline} for an outline). We hope the proof techniques will be helpful in future works concerning general convex polytopes. 

Let $r=\operatorname{rank}(\Phi).$ We want to raise the following question: 
\begin{question}\label{question-total-betti-number}
    For $k$ sufficiently large, is the total Betti number 
    \begin{equation*}
        \operatorname{Card}\left(\prescript{f}{}{[e, t_{k\lambda}]}\right) =  \sum_i \prescript{f}{}{b}_{i}^{t_{k\lambda}}
    \end{equation*}
    a quasi-polynomial\footnote{A function $f\colon \mathbb{N} \rightarrow \mathbb{N}$ is a \emph{quasi-polynomial} if there exist polynomials $p_0, \ldots, p_{s-1}$ such that $f(n)=p_i(n)$ when $i \equiv n \bmod s$. 
    } in~$k$ of degree $r$, with $\frac{\operatorname{Vol}_r(P^\lambda)}{\operatorname{Vol}_r(A_+)}$
     as the leading coefficient?
\end{question}

Our main object of study, the sequence $(\prescript{f}{}{b}_i^{t_{\lambda}})_i$, is related to a refinement of this question, which deals with some slices of $P^\lambda$ instead of $P^\lambda$ as a whole. These slices are hyperplane sections $\operatorname{ht}^{-1}(i)$ of $P^\lambda$, for different values of $i$. Theorem~\ref{thm-lattice-formula-in-the-introduction} relates $(\prescript{f}{}{b}_i^{t_{\lambda}})_i$ with the numbers of lattice points in these different slices. The refinement of Question~\ref{question-total-betti-number} is the following:

\begin{question}\label{question-cardinality-of-section-of-polytope}
    For $k$ sufficiently large, is $\prescript{f}{}{b}_{ki}^{t_{k\lambda}}$ a quasi-polynomial in~$k$ of degree $(r-1)$  with 
    \begin{equation*}
        \frac{\operatorname{Vol}_{r-1} (\operatorname{ht}^{-1} (i))}{\operatorname{Vol}_r(A_+) \cdot \lVert 2 \rho \rVert}
    \end{equation*}
    as the leading coefficient?
\end{question}
If the answer to Question~\ref{question-cardinality-of-section-of-polytope} is yes, this immediately implies Corollary~\ref{cor-layer}.

\begin{figure}[ht] 
    \begin{subfigure}[t]{.45\linewidth}
        \centering
        \includegraphics[scale=0.5]{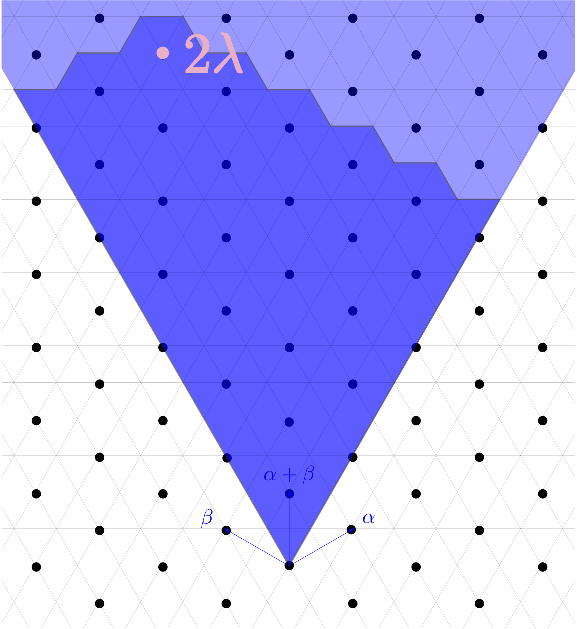}
        \caption{The Bruhat interval $\prescript{f}{}{[e,t_{2\lambda}]}$.}
        \label{fig:2lambda}
    \end{subfigure}
    \begin{subfigure}[t]{.45\linewidth}
        \centering
        \includegraphics[scale=0.5]{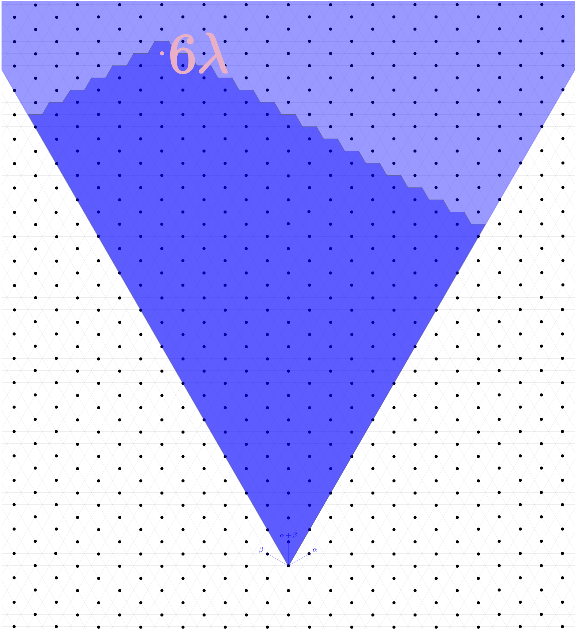}
        \caption{The Bruhat interval  $\prescript{f}{}{[e,t_{6\lambda}]}$.}
        \label{fig:6lambda}
    \end{subfigure}
    \bigbreak
    \begin{subfigure}[c]{\linewidth}
        \centering
        \includegraphics[scale=0.5]{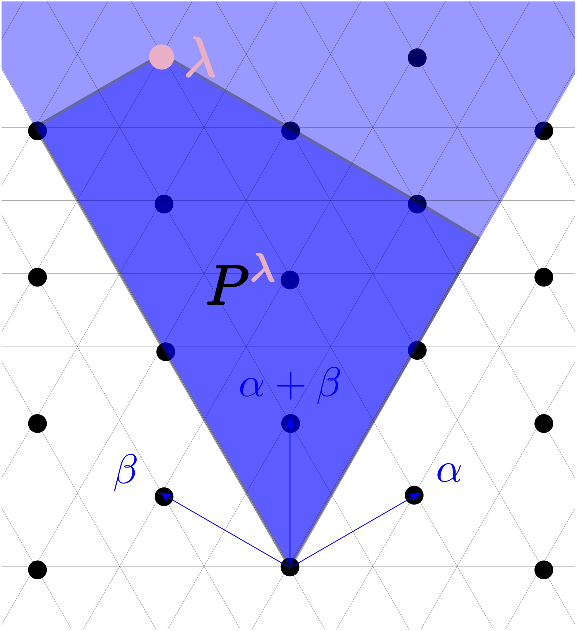}
        \caption{Polytope $P^{\lambda}$.}
        \label{fig:llambda}
    \end{subfigure}
    \caption{Behavior of the intervals $\prescript{f}{}{[e, t_{k\lambda}]}$  when $W$ is of affine type $A_2$ and $\lambda=3\alpha+4\beta$, where $\alpha:=\alpha_1\spcheck$ and $\beta:=\alpha_2\spcheck$. 
    In each picture, the set of small triangles corresponds to the set of alcoves.
    The coroot lattice is indicated by black bullets, and the dominant Weyl chamber is colored blue.
    In the first two pictures, the alcoves corresponding to the elements in the intervals are filled with darker blue.
    So is the polytope $P^\lambda$ in the third picture.}
    \label{fig: limit}
\end{figure}

\subsubsection{Euclidean geometry of alcoves} \label{subsec-alcovic}
In his expository paper~\cite{libedinsky2022introsurvey}, Libedinsky conjectured an intrinsic connection between Bruhat intervals for affine Weyl groups and Euclidean geometry.
The results in the present paper give some evidence in this regard:
From the perspective of Euclidean geometry, a Bruhat interval can be realized as a region---a union of finitely many alcoves---inside $E$. 
As the lattice element $\lambda$ in the interval $\prescript{f}{}{[e, t_{\lambda}]}$ is dilated and the region is rescaled accordingly, the alcoves within it become progressively smaller.
In the limit, the region converges to the polytope $P^\lambda$ (Figure~\ref{fig: limit}). 
This observation is the main geometric intuition behind Theorem~\ref{thm-main-in-the-introduction}. 

Although connections between convex polytopes and the affine Bruhat order are well known, their precise formulations can be unexpectedly subtle.
For example, even in the simplest nontrivial case of type~$\widetilde{A}_2$, making the connection between an affine Bruhat interval and a convex polygon explicit proved highly nontrivial, as shown in~\cite[Theorem C]{BLV25}.

For elements associated with dominant coweights $\mu$, Castillo, de la Fuente, Libedinsky, and Plaza proved a ``Geometric Formula''~\cite[Theorem B]{castillo2023size} relating the cardinality of the corresponding lower Bruhat interval and volumes of faces of the weight polytope $\operatorname{Conv}\{w \mu \mid w \in W_f\}$ using their ``Lattice formula''~\cite[Theorem A]{castillo2023size}.
Our dominant lattice formula is similar to theirs but with some differences: their lattice formula considers the presence of dominant and non-dominant weights, while our formula takes into account the length of each lattice point together with the singular phenomena when lattice points belong to walls of the dominant Weyl chamber.
In view of their work, we would like to ask the following:
\begin{question}\label{question-convex-geometric-formula}
For $\lambda \in \mathbb{Z} \Phi\spcheck \cap \overline{C_+}$, is there a convex geometry formula giving the cardinality of  $\prescript{f}{}{[e, t_\lambda]}$ (that is, the total Betti number $\sum_i \prescript{f}{}{b}_i^{t_{\lambda}}$ of the spherical Schubert variety $\overline{\mathcal{G} r_{-w_0 \lambda}}=X_{t_{-\lambda}}$) in terms of volumes of faces of the polytope $P^\lambda$?
\end{question}

\subsubsection{Measuring top-heaviness}
Recall that Bj{\"o}rner and Ekedahl~\cite{bjorner2009shape} showed that the sequence $(\prescript{J}{}{b}_i^w)_i$ is \emph{top-heavy}, that is, $\prescript{J}{}{b}_i^w \leq \prescript{J}{}{b}_{\ell(w)-i}^{w}$, for $i \leq \frac{\ell(w)}{2}$. It is natural to ask ``how top-heavy'' these sequences are. In this direction---and inspired by the question at the end of~\cite{bjorner2009shape}---we ask the following:
\begin{question}
    For each $\Phi$ and $\lambda \in \mathbb{Z} \Phi\spcheck \cap \overline{C_{+}}$, let $z_{0}(\lambda)$ be the point where the density function of $\mathrm{ht}_* \mathrm{Vol}_{r}(P^\lambda)$ reaches its maximum. 
    Is it true that $\sup\left\{\frac{z_{0}(\lambda)}{\ell\left(t_\lambda\right)}\right\}=1$? Here, the supremum is taken over all root systems and all $\lambda \in \mathbb{Z} \Phi\spcheck \cap \overline{C_{+}}$.
\end{question}

If the answer to the above question is yes, Theorem~\ref{thm-main-in-the-introduction} would reveal the existence of an interesting family of non-trivial parabolic Kazhdan--Lusztig polynomials or, in other words, non-rationally-smooth affine parabolic Schubert varieties. 

\subsection{Structure of the paper}
The remainder of the paper is organized as follows. In Section~\ref{sec-pre}, we introduce our notations and some preliminary results. In Section~\ref{sec-order}, we compare two partial orders related to our study: the partial order $\preceq$ on~$\mathbb{Z} \Phi\spcheck$ and the Bruhat--Chevalley order $\le$ on~$\prescript{f}{}{W}$. 
In Section~\ref{sec-dom-lattice-formula}, we introduce our relevant sequences of study and the polytope $P^\lambda$, and prove the dominant lattice formula. In Section~\ref{sec-main}, we prove the weak convergence of the discrete measure sequence $(\mathfrak{m}_k)_k$. 
In Section~\ref{sec-main-unifm-convg}, we prove the uniform convergence of the step function sequence $(S_k)_k$.
As a byproduct, we obtain a combinatorial identity (Equation  \eqref{eq-main-eq}) related to the root system.

\subsection*{Acknowledgments}
The first author would like to express his gratitude to Anthony Henderson and Geordie Williamson for great discussions regarding the affine Weyl group during his Ph.D.\@~program.
The authors would like to thank Allen Knutson, David Plaza, Alex Weekes, Ge Xiong, and Xinwen Zhu for their valuable communications. 
The third author is supported by the Fundamental Research Funds for the Central Universities (No.~531118010972).

\section{Preliminaries} \label{sec-pre}
In this section, we recollect our notations and some preliminary results for later use. 
Throughout the paper, we use either $\operatorname{Card}(A)$ or $\lvert A \rvert$ to denote the cardinality of a set $A$.

\subsection{The affine Weyl group and alcoves}\label{subsec-affine}
Standard references for Sections~\ref{subsec-affine} and~\ref{subsec-dom} are~\cite{IM1965, bourbaki1968}. 
Let $\Phi$ be the (irreducible, reduced, and crystallographic) root system of a finite-dimensional simple Lie algebra over $\mathbb{C}$, and $E$ be the Euclidean space where $\Phi$ lives.
We denote by $(-|-)\colon E \times E \to \mathbb{R}$ the inner product on~$E$.
For any root $\alpha \in \Phi$, we denote by $\alpha\spcheck$ the corresponding coroot $\frac{2 \alpha}{(\alpha|\alpha)}$. 
Then $\Phi\spcheck := \left\{\alpha\spcheck \in E \mid \alpha \in \Phi \right\}$ is the dual root system.
We fix a set $\Delta = \left\{\alpha_i \mid i = 1,\dots, r\right\}$ of simple roots of $\Phi$, where $r = \dim E$ is the rank of $\Phi$.
Let 
\begin{equation*}
    s_i\colon E \to E, \quad x \mapsto x - (x|\alpha_i) \alpha_i \spcheck, \quad i = 1, \dots, r,
\end{equation*}
be the simple reflections. Let $W_f:= \langle s_1 ,\dots, s_r \rangle$ be the (finite) Weyl group associated with $\Phi$.

For a nonzero vector $\beta \in E$ and a real number $z \in \mathbb{R}$, we have the hyperplane
\begin{equation*}
    H_{\beta,z} := \left\{x \in E \mid (x | \beta) = z \right\}.
\end{equation*}
The \emph{affine Weyl group} $W$ associated with $\Phi$ is defined to be the group of affine transformations generated by affine reflections $\left\{s_{\alpha,k} \mid \alpha \in \Phi^+, k \in \mathbb{Z} \right\}$, where $\Phi^+$ denotes the set of positive roots and $s_{\alpha,k}$ is the affine reflection along the affine hyperplane $H_{\alpha,k}$.
Explicitly, the affine reflection $s_{\alpha,k}$ ($\alpha \in \Phi^+$, $k \in \mathbb{Z}$) is given by
\begin{equation} \label{eq-def-refl}
    s_{\alpha,k}(v) = v + \bigl(k - (v | \alpha) \bigr) \alpha\spcheck = s_{\alpha,0} (v) + k \alpha\spcheck, \text{ for all } v \in E.
\end{equation}
In particular, $s_i = s_{\alpha_i, 0} \in W$ for all $i = 1, \dots, r$.
Let $\alpha_0$ be the highest root in~$\Phi$, and $s_0 := s_{\alpha_0,1}$.
Then, $(W,S)$ is a Coxeter system with the set $S := \{s_0,s_1, \dots, s_r\}$ of simple reflections. 
We denote by $\ell$ the length function on~$W$ and by $\le$ the Bruhat--Chevalley order on~$W$.
For $x, y \in W$, we write $x < y$ if $x \le y$ and $x \ne y$.
The restriction of $\ell$ to the finite Coxeter system $(W_f, S_f)$ is also denoted by $\ell$, where $S_f = S \setminus \{s_0\}$.

The affine Weyl group~$W$ is isomorphic to the semidirect product 
\begin{equation*}
    \mathbb{Z} \Phi\spcheck \rtimes W_f,
\end{equation*}
where $\mathbb{Z} \Phi \spcheck$ is the coroot lattice.
For any element $\lambda \in \mathbb{Z}\Phi\spcheck$, we denote by $t_\lambda \in W$ the translation $x \mapsto x + \lambda$ on~$E$. 
In the semidirect product decomposition, the element $t_\lambda$ corresponds to the pair $(\lambda, \operatorname{id}_{W_f})$.

A connected component of $E \setminus \{H_{\alpha,k} \mid \alpha\in \Phi^+, k \in \mathbb{Z}\}$ is called an \emph{alcove}.
The closure $\overline{A}$ of any alcove $A$ is a fundamental domain for the action of $W$ on~$E$.
Let
\begin{equation*}
    A_+ := \{x \in E \mid 0 < (x|\alpha) < 1 \text{ for all } \alpha \in \Phi^+\}
\end{equation*}
be the \emph{fundamental alcove}. 
The group~$W$ acts on~$E$ and preserves the set of hyperplanes 
\begin{equation*}
    \{H_{\alpha, k} \mid \alpha \in \Phi^+, k \in \mathbb{Z}\}.
\end{equation*}
Therefore, $W$ acts on the set of alcoves. This action is simply transitive, hence induces a natural bijection between $W$ and the set of alcoves
\begin{equation} \label{eq-bij}
    w \mapsto A_w := wA_+, \text{ for any } w \in W.
\end{equation}

The following lemma is classical.

\begin{lemma} \label{lem-alcove-bruhat-order}
    Suppose  $w \in W$, and $s$ is a reflection along some hyperplane $H_{\alpha, k}$ for $\alpha \in \Phi^+$ and $k \in \mathbb{Z}$. Then:
    \begin{enumerate} 
        \item \label{lem-alcove-bruhat-order-1} $sw > w$ if and only if the two alcoves $A_+$ and $A_w$ are in the same side of  $H_{\alpha, k}$, that is, 
        \begin{equation*}
            \bigl((u | \alpha)- k\bigr) \bigl((v | \alpha)- k\bigr) > 0 
        \end{equation*}
        for some (and hence for all) $u \in A_+$ and $v \in A_w$.
        \item \label{lem-alcove-bruhat-order-2} The length $\ell(w)$ of $w$ equals the number of hyperplanes separating $A_+$ and $A_w$, that is, 
        \begin{equation*}
            \ell(w) = \operatorname{Card} \left\{H_{\alpha,k} \xmiddle| 
            \begin{gathered}
                \text{$\alpha \in \Phi^+$, $k \in \mathbb{Z}$, such that $A_+$ and} \\ 
                \text{$A_w$ lie in opposite sides of $H_{\alpha,k}$} 
            \end{gathered}
            \right\}.
        \end{equation*}
        \item \label{lem-alcove-bruhat-order-3} If $s \in S_f$, that is, $\alpha \in \Delta$ and $k=0$, then $ws > w$ if and only if $A_+$ and $A_w$ are in the same side of  $w H_{\alpha, 0}$.
    \end{enumerate}
\end{lemma}

\begin{remark}
    The statement ``$A_+$ and $A_w$ are in the same side of  $w H_{\alpha, 0}$'' in Lemma~\ref{lem-alcove-bruhat-order}\ref{lem-alcove-bruhat-order-3} is equivalent to ``$A_{w^{-1}}$ and $A_+$ are in the same side of $H_{\alpha, 0}$'', that is, 
    \begin{equation*}
        (u | \alpha) (v | \alpha) > 0 \text{ for some (and hence for all) } u \in A_{w^{-1}} \text{ and } v \in A_+.
    \end{equation*}
    Since taking inverse preserves the Bruhat--Chevalley order, Lemma~\ref{lem-alcove-bruhat-order}\ref{lem-alcove-bruhat-order-3} follows from Lemma~\ref{lem-alcove-bruhat-order}\ref{lem-alcove-bruhat-order-1}.
\end{remark}

By the orbit-stabilizer theorem, we have the following lemma.

\begin{lemma} \label{lem-left-coset-1-1-lattice}
    We have a natural bijection $\mathbb{Z} \Phi\spcheck \to W / W_f$ given by $\lambda \mapsto t_\lambda W_f$.
\end{lemma}

\begin{proof}
    The action of $W$ on~$E$ induces a transitive action of $W$ on the lattice $\mathbb{Z} \Phi\spcheck$, and the stabilizer of the origin~$0 \in E$ is exactly $W_f$.
\end{proof}

\begin{remark} \label{rmk-closure}
    The lattice point $\lambda \in \mathbb{Z} \Phi\spcheck$ belongs to the closure of $A_w$ $(w \in W)$ if and only if $w \in t_\lambda W_f$.
    Geometrically, the alcoves $\{A_w \mid w \in t_\lambda W_f\}$ are the translation by $\lambda$ of the alcoves corresponding to $W_f$, and they are ``arranged'' around $\lambda$.
\end{remark}

\subsection{Dominant alcoves and \texorpdfstring{$\prescript{f}{}{W}$}{fW}} \label{subsec-dom}

The \emph{dominant Weyl chamber} $C_+$ is defined by
\begin{equation*}
    C_+ := \{x \in E \mid (x|\alpha_i) > 0, \text{ for all } i = 1, \dots, r\},
\end{equation*}
which is an open simplicial cone in~$E$. The hyperplanes $\{H_{\alpha,0} \mid \alpha \in \Delta\}$ are called the \emph{walls} of $C_+$. 
A point $x \in E$ is called \emph{strongly dominant} if $x \in C_+$, and \emph{dominant} if $x \in \overline{C_+}$, the closure of $C_+$.
The closed cone $\overline{C_+}$ is a fundamental domain for the action of $W_f$ on~$E$. Each alcove $A$ is contained in~$C_+$ or has an empty intersection with $C_+$; in the first case, we say that $A$ is \emph{dominant}.

The Weyl group~$(W_f, S_f)$ is a standard parabolic subgroup of the affine Weyl group~$(W,S)$.
As mentioned in the introduction, $\prescript{f}{}{W}\subset W$ is defined to be the set of minimal representatives for the right cosets $W_f \backslash W$. 
By Lemma~\ref{lem-alcove-bruhat-order}\ref{lem-alcove-bruhat-order-1}, we immediately have the following proposition. It states that under the bijection \eqref{eq-bij} the set $\prescript{f}{}{W}$ corresponds to the set of dominant alcoves. 

\begin{proposition} \label{prop-dom-alcove}
    An element $w \in W$ is the minimal length representative in its right coset $W_f w$ (that is, $w \in \prescript{f}{}{W}$) if and only if $A_w$ is dominant, that is, $A_w \subset C_+$.
    In particular, for any $\lambda \in \mathbb{Z} \Phi \spcheck \cap \overline{C_+}$, we have $t_\lambda \in \prescript{f}{}{W}$.
\end{proposition}

Therefore, we call the set $\prescript{f}{}{[e, w]} := \{v \in \prescript{f}{}{W} \mid v\le w\}$ a \emph{dominant lower Bruhat interval}.

Let $\rho := \frac{1}{2} \sum_{\alpha \in \Phi^+} \alpha$ be the half-sum of all positive roots\footnote{The vector $\rho$ is often called the \emph{Weyl vector} in the literature.}.
Moreover, let $\rho \spcheck := \frac{1}{2} \sum_{\alpha \in \Phi^+} \alpha \spcheck$ be the half-sum of all positive coroots (in general, $\rho\spcheck$ and $\frac{2 \rho}{(\rho|\rho)}$ are not the same). The next lemma and its following corollary are standard.

\begin{lemma} [{\cite[Ch.\@~VI, Sect.\@~1, No.\@~10]{bourbaki1968}}] \label{lem-rho}
    We have $(\rho| \alpha_i\spcheck) = 1$ and $(\rho \spcheck|\alpha_i) =1$ for all $i = 1, \dots, r$.
\end{lemma}

\begin{corollary} \label{cor-rho}
    There is $\delta \in \mathbb{R}_{>0}$ such that $\varepsilon \rho \spcheck \in A_+$ for any $0 < \varepsilon < \delta$.
\end{corollary}

The following lemma computes the length of $t_\lambda$ for $\lambda \in \mathbb{Z} \Phi\spcheck \cap \overline{C_+}$.
This is a particular case of~\cite[Proposition 1.23]{IM1965}.

\begin{lemma} \label{lem-length-trans}
     For any $\lambda \in \mathbb{Z} \Phi\spcheck \cap \overline{C_+}$, we have $\ell(t_\lambda) = (2\rho | \lambda)$.
\end{lemma}

\begin{corollary} \label{cor-length-k-lambda}
    For any $\lambda \in \mathbb{Z} \Phi\spcheck \cap \overline{C_+}$ and $k \in \mathbb{N}$, we have $\ell(t_{k\lambda}) = k \ell(t_{\lambda})$.
\end{corollary}

The following lemma is useful when dealing with $\lambda \in \mathbb{Z} \Phi\spcheck$ on some walls of the dominant Weyl chamber $C_+$.

\begin{lemma} \label{lem-dom-long}
    Suppose $\lambda \in \mathbb{Z} \Phi\spcheck \cap \overline{C_+}$.
    Let 
    \begin{equation*}
        I := \{i \mid 1 \le i \le r, (\lambda| \alpha_i) = 0 \},
    \end{equation*}
    that is, $\lambda$ belongs to the wall $H_{\alpha_i,0}$ if $i \in I$.
    Then, we have:
    \[\text{$t_\lambda s_i > t_\lambda$ if $i \in I$;} \qquad \text{$t_\lambda s_i < t_\lambda$ if $i \in \{1, \dots, r\} \setminus I$.}\]
\end{lemma}
\begin{proof} 
    By Lemma~\ref{lem-alcove-bruhat-order}\ref{lem-alcove-bruhat-order-3}, it suffices to show that $A_+$ and $A_{t_\lambda}$ lie in the same side of $t_\lambda H_{\alpha_i,0}$ if and only if $i \in I$.
    This statement is equivalent to the following:
    \begin{equation} \label{eq-210}
        \text{$A_{t_{- \lambda}}$ and $A_+$ lie in the same side of $H_{\alpha_i,0}$ if and only if $i \in I$.}
    \end{equation}
    Let $\varepsilon \in \mathbb{R}_{> 0}$ be small enough such that $\varepsilon \rho \spcheck \in A_+$ (see Corollary~\ref{cor-rho}).
    For all $i = 1, \dots, r$,  we have $(\rho \spcheck| \alpha_i) = 1$.
    Moreover, $(\lambda|\alpha_i) \ge 1$ if $i \notin I$.
    Thus, $(\varepsilon\rho \spcheck | \alpha_i) = \varepsilon > 0$, and
    \begin{equation*}
        (\varepsilon\rho \spcheck - \lambda| \alpha_i) = \varepsilon - (\lambda| \alpha_i) 
        \begin{cases}
            > 0, \text{ if } i \in I, \\
            < 0, \text{ if } i \notin I.
        \end{cases}
    \end{equation*}
    Since $\varepsilon \rho \spcheck \in A_+$ and $\varepsilon \rho \spcheck - \lambda \in A_{t_{-\lambda}}$, this verifies the statement \eqref{eq-210}. 
\end{proof}

\begin{corollary} \label{cor-strong-dom-long}
    If a coroot lattice point $\lambda$ is strongly dominant, that is, $\lambda \in \mathbb{Z} \Phi\spcheck \cap C_+$, then $t_\lambda$ is the maximal element in the left coset $t_\lambda W_f$.
\end{corollary}

\begin{proof}
    If $\lambda$ is strongly dominant, then the set $I$ in Lemma~\ref{lem-dom-long} is empty.
    Thus $t_\lambda s_i < t_\lambda$ for all $i = 1, \dots, r$.
\end{proof}

\begin{remark}
   Let $\lambda \in \mathbb{Z} \Phi\spcheck \cap C_+$, then $t_\lambda$ is the minimal representative for the \emph{right} coset $W_f t_\lambda$ by Proposition~\ref{prop-dom-alcove} and the maximal representative for the \emph{left} cost $t_\lambda W_f$ by Corollary~\ref{cor-strong-dom-long}.
\end{remark}

\subsection{Measures on \texorpdfstring{$\mathbb{R}$}{R} and \texorpdfstring{$E$}{E}} \label{subsec-measure}
In this subsection, we recollect some terminology and basic results about weak convergence of measures.
For a more complete exposition, we refer to~\cite{billingsley1995measure, billingsley1999measures}.

Let $\mathcal{B}$ be the \emph{Borel $\sigma$-field} in~$\mathbb{R}$ generated by the set of open subsets, that is, the smallest collection of subsets of $\mathbb{R}$ containing all open subsets, which is closed under taking complements, countable unions, and countable intersections.
A set in~$\mathcal{B}$ is called a \emph{Borel set}. In particular, any countable set is a Borel set.
A \emph{measure} $\mathfrak{m}$ on~$(\mathbb{R}, \mathcal{B})$ (or simply, on~$\mathbb{R}$) is a set function 
$\mathfrak{m} \colon \mathcal{B} \to \mathbb{R}_{\ge 0} \cup \{\infty\}$ such that $\mathfrak{m}(\emptyset) = 0$ and $\mathfrak{m} \left( \bigcup_{i = 1}^\infty U_i\right) = \sum_{i=1}^\infty \mathfrak{m} \left(U_i\right)$ for any sequence of disjoint sets $(U_i)_i$.

A measure $\mathfrak{m}$ on~$\mathbb{R}$ is said to be \emph{bounded} if $\mathfrak{m} (\mathbb{R}) < \infty$.
If a closed set $F \subset \mathbb{R}$ is such that $\mathfrak{m}(\mathbb{R} \setminus F) = 0$, we say that $\mathfrak m$ \emph{is supported on} $F$.
\begin{definition}
A sequence of bounded measures $(\mathfrak{m}_k)_k$ is said to \emph{converge weakly} to a measure $\mathfrak{m}$ if
\begin{equation*}
\lim_{k \to \infty} \int f \, \mathrm{d}\mathfrak{m}_k = \int f \, \mathrm{d}\mathfrak{m}
\end{equation*}
for any bounded continuous real function $f$. 
\end{definition}

\begin{lemma} [{\cite[Sect.\@~25]{billingsley1995measure}, \cite[Example 2.3]{billingsley1999measures}}] \label{lem-weak-convergence}
    A sequence of bounded measures $(\mathfrak{m}_k)_k$ on~$\mathbb{R}$ converges weakly to a measure $\mathfrak{m}$ if and only if  
    \begin{equation*}
        \lim_{k \to \infty} \mathfrak{m}_k \bigl((-\infty, z]\bigr) = \mathfrak{m} \bigl((-\infty, z]\bigr) \text{ for any $z \in \mathbb{R}$ such that $\mathfrak{m} \left(\{z\}\right) = 0$.} 
    \end{equation*}
\end{lemma}

\begin{remark} \label{rmk-strong-convergence}
    There is another notion of convergence of measures.
    A sequence of measures $(\mathfrak{m}_k)_k$ is said to \emph{converge strongly} (or \emph{setwise}) to a measure $\mathfrak{m}$ if $\lim_{k \to \infty} \mathfrak{m}_k (U) = \mathfrak{m} (U)$ for any Borel set $U$.
\end{remark}

On the $r$-dimensional Euclidean space $E$ where the root system $\Phi$ lives, we have the Lebesgue measure induced by the inner product $(-|-)$.
This measure is denoted by $\operatorname{Vol}_r$. 
Moreover, for an $i$-dimensional affine subspace $M$ in~$E$ ($i = 0, 1, \dots, r-1$), we can talk about the $i$-dimensional Lebesgue measure $\operatorname{Vol}_i$ on~$M$ which is also induced by $(-|-)$.
For example, $\operatorname{Vol}_1(v) = \sqrt{(v|v)}$ is the length of a vector $v$, also denoted by $\lVert v \rVert$.

Let $D$ be the open parallelotope spanned by the simple coroots, that is, 
\begin{equation*}
D := \left\{\sum_{i=1}^r a_i \alpha_i\spcheck \in E \xmiddle| 0 < a_i < 1 \text{ for all } i \right\}.
\end{equation*}
The following lemma relates the volumes of $D$ and $A_+$.

\begin{lemma} \label{lem-vol-D}
    $\operatorname{Vol}_r (D) = \lvert W_f \rvert \cdot \operatorname{Vol}_r (A_+)$.
\end{lemma}

\begin{proof}
    We have the following facts:
    \begin{enumerate}
        \item $W$ acts continuously and properly on~$E$, and $\operatorname{Vol}_r$ is a $W$-invariant measure;
        \item $\mathbb{Z} \Phi\spcheck$ is a subgroup of $W$; 
        \item $A_+$ and $D$ are both open subsets of $E$ with finite non-zero measure;
        \item the unions $E^\prime := \bigsqcup_{w \in W} w A_+$ and $E'' := \bigsqcup_{\lambda \in \mathbb{Z} \Phi\spcheck} t_\lambda D$ are both disjoint;
        \item $\operatorname{Vol}_r (E \setminus E^\prime) = \operatorname{Vol}_r (E \setminus E'') = 0$.
    \end{enumerate}
    Applying~\cite[Ch.\@~VI, Sect.\@~2, No.\@~4, Lemma 1]{bourbaki1968} to the above facts, we obtain 
    \[\lvert W_f \rvert = (W:\mathbb{Z} \Phi\spcheck) = \frac{\operatorname{Vol}_r(D)}{\operatorname{Vol}_r(A_+)}\]
    as desired.
\end{proof}

\subsection{Polytopes}

For references on polytopes, one may refer to~\cite{grunbaum03,Ziegler95}.
A (\emph{convex}) \emph{polytope} $P$ in the Euclidean space $E$ is the convex hull of a finite set of points in~$E$.
Note that a polytope is a bounded closed set. Equivalently, a polytope in~$E$ is a bounded subset of $E$, which can be written in the form
\[P = \bigcap_{i = 1, \dots, k} \{x \in E \mid f_i(x) \ge a_i\},\]
where each $f_i$ is a linear function on~$E$ and $a_i \in \mathbb{R}$. We define an \emph{open face} of $P$ to be an equivalence class of the equivalence relation $\sim$ on~$P$ defined by $x \sim y$ if for each $i = 1, \dots, k$, we have $f_i(x) = a_i$ if and only if $f_i(y) = a_i$.
The closure of an open face is called a \emph{face}.
Clearly, $P$ is a disjoint union of finitely many open faces.

For a subset $M$ of $E$, we denote by $\langle M \rangle_\text{aff}$ the affine subspace of $E$ spanned by $M$, that is, the smallest affine subspace containing $M$. 
For an open face $F^\circ$ of $P$ and the corresponding face $F = \overline{F^\circ}$, $F^\circ$ is the relative interior of $F$ and it is an open subset of $\langle F^\circ \rangle_\text{aff}$ (hence we call it an ``open face''). 
The \emph{dimension} of $F^\circ$ and $F$ is defined to be the dimension of $\langle F^\circ \rangle_\text{aff}$. 
In particular, the $0$-dimensional open faces of $P$ are the same as the $0$-dimensional faces of $P$, that is, the vertices of $P$.
The \emph{dimension} of $P$ is the maximal dimension among its faces.
We list some elementary facts without proof.

\begin{lemma} \label{lem-intersect-face}
    Let $P$ be a polytope in~$E$, and $H \subseteq E$ be an affine hyperplane.
    \begin{enumerate}
        \item \label{lem-intersect-face-1} If $P \cap H$ is not empty, then it is a polytope.
        \item If $F^\circ$ is an open face of $P$ and $F^\circ \cap H \ne \emptyset$, then $F^\circ \cap H$ is an open face of $P \cap H$.
        \item Any open face of $P \cap H$ is of the form $F^\circ \cap H$ for some open face $F^\circ$ of $P$.
    \end{enumerate}
\end{lemma}

\section{The partial order \texorpdfstring{$\preceq$}{prec} and the Bruhat--Chevalley order} \label{sec-order}

Besides the Bruhat--Chevalley order $\le$ on~$W$, there are two other related partial orders. One is defined on~$\mathbb{Z} \Phi\spcheck$ and the other on~$W$. 

\begin{definition} \label{def-order} \leavevmode
    \begin{enumerate}
        \item \label{def-order-1} For $\lambda, \mu \in E$, we write $\mu \preceq \lambda$ if $\lambda - \mu$ is a non-negative  linear combination of simple coroots $\{\alpha_1\spcheck, \dots, \alpha_r\spcheck\}$ (as well as simple roots).
        This gives a partial order on~$E$.
        For simplicity, we write $\mu \prec \lambda$ if $\mu \preceq \lambda$ and $\mu \ne \lambda$.
        (Note that if $\lambda, \mu \in \mathbb{Z} \Phi\spcheck$ and $\mu \preceq \lambda$, then $\lambda - \mu$ is a non-negative \emph{integral} linear combination of simple coroots.)
        \item \label{def-order-2} For $x, y \in W$, we write $y \le_a x$ if there exists a sequence of elements $(y = y_0, y_1, \dots, y_{n-1}, y_n = x)$ in~$W$ and a sequence of affine reflections 
        \begin{equation*}
            \{s_{\beta_i, k_i}\mid  \beta_i \in \Phi^+, k_i \in \mathbb{Z}, i = 1, \dots, n \},
        \end{equation*}
        such that for each $i = 1,\dots, n$, we have $s_{\beta_i, k_i} y_i = y_{i-1}$ and $(u_i|\beta_i) > k_i$ for some (and hence for all) $u_i \in A_{y_i}$.
        Following Verma~\cite{verma1975} and Wang~\cite{wang1987order}, we call the partial order $\le_a$ on~$W$ the \emph{affine order}.
    \end{enumerate}
\end{definition}

In Definition~\ref{def-order}\ref{def-order-2}, we have $u_i \in A_{y_i}$, $s_{\beta_i, k_i} (u_i) \in A_{y_{i-1}}$, and
\begin{equation*}
    u_i - s_{\beta_i, k_i} (u_i) = u_i - \Bigl( u_i + \bigl( k_i - (u_i | \beta_i) \bigr)  \beta_i \spcheck \Bigr) = \bigl( (u_i | \beta_i) - k_i  \bigr)  \beta_i \spcheck
\end{equation*}
which is a positive multiple of $\beta_i\spcheck$.
Inductively, we have:

\begin{corollary} \label{cor-affine-order}
    If $x, y \in W$ and $y \le _a x$, then $u - yx^{-1} (u)$ is an $\mathbb{R}_{\ge 0}$-linear combination of simple coroots for any $u \in A_x$ (and hence $yx^{-1} (u) \in A_y$).
\end{corollary}

We have the following relation between the affine order $\le_a$ on~$W$ and the partial order $\preceq$ on~$\mathbb{Z} \Phi\spcheck$, which says that $\le_a$ is an extension of $\preceq$.
\begin{lemma} \label{lem-lattice-affine-dominance-order}
    For two lattice points $\mu, \lambda \in \mathbb{Z} \Phi\spcheck$, $t_\mu \le _a t_\lambda$ if and only if $\mu \preceq \lambda$.
\end{lemma}

\begin{proof}
  Suppose $t_\mu \le_a t_\lambda$.
  For $u \in A_{t_\lambda}$, we have $t_\mu t_\lambda^{-1} (u) = u - \lambda + \mu \in A_{t_\mu}$.
  By Corollary~\ref{cor-affine-order}, $u - t_\mu t_\lambda^{-1} (u) = \lambda - \mu$ is an $\mathbb{R}_{\ge 0}$-linear combination of simple coroots.
  Therefore, we have $\mu \preceq \lambda$.
  
  Conversely, suppose $\mu \preceq \lambda$. By induction, 
  we may assume $\lambda = \mu + \alpha\spcheck$ for some simple coroot $\alpha\spcheck$.
  Let $k = (\lambda | \alpha) \in \mathbb{Z}$.
  From the formula of affine reflections in Equation  \eqref{eq-def-refl}, one may easily verify that
  $s_{\alpha, k-1} s_{\alpha,k} = t_{-\alpha\spcheck}$.
  Therefore, \[s_{\alpha, k-1} s_{\alpha,k} t_\lambda = t_\mu.\]
  We choose $\varepsilon \in \mathbb{R}_{>0}$ small enough so that $\varepsilon \rho\spcheck \in  A_+$ as in Corollary~\ref{cor-rho}.
  Then $\lambda +\varepsilon\rho\spcheck \in A_{t_\lambda}$.
  Moreover, 
  \[(\lambda + \varepsilon \rho\spcheck | \alpha)  = k + \varepsilon > k.\]
  This implies $s_{\alpha,k} t_\lambda \le_a t_\lambda$ by Definition~\ref{def-order}\ref{def-order-2}.
  Similarly, $s_{\alpha,k} (\lambda + \varepsilon \rho\spcheck) \in A_{s_{\alpha,k} t_\lambda}$, and 
  \begin{align*}
      \bigl( s_{\alpha,k} (\lambda + \varepsilon \rho\spcheck) \big| \alpha \bigr) & = \bigl(\lambda + \varepsilon\rho\spcheck + \left(k-(\lambda +\varepsilon\rho\spcheck| \alpha)\right)\alpha\spcheck \big| \alpha \bigr) \\
      & = (\lambda + \varepsilon \rho\spcheck | \alpha) + 2 (k-(\lambda +\varepsilon\rho\spcheck| \alpha)) \\
      & =  k - \varepsilon \\
      & > k-1.
  \end{align*}
  This gives $t_\mu = s_{\alpha, k-1} s_{\alpha,k} t_\lambda \le_a s_{\alpha,k} t_\lambda \le_a t_\lambda$. 
\end{proof}

The following lemma states that the Bruhat--Chevalley order and the affine order coincide for elements in~$\prescript{f}{}{W}$. 

\begin{lemma} [{\cite[Sect.\@~1.6]{verma1975}}] \label{lem-dom-order-eq-affine}
    For $x, y \in \prescript{f}{}{W}$, $y \le x$ if and only if $y \le_a x$.
\end{lemma}

\begin{remark} \label{rmk-bruhat-affine-order-dom}
    The assertion in Lemma~\ref{lem-dom-order-eq-affine} was firstly stated and partially proved by Verma~\cite[Sect.\@~1.6]{verma1975}. 
    However, Verma skipped the key point (whose proof is somewhat involved) in his incomplete proof and referred it to his unpublished preprint~\cite{verma1972preprint}.
    A complete proof can be found in Wang's paper~\cite{wang1987order}, whose Chinese original appeared 12 years later than Verma's paper.
    According to Wang (see the introductory section and the ``Notes by the author'' in~\cite{wang1987order}), and to the best of the authors' knowledge, there was no complete proof of Lemma~\ref{lem-dom-order-eq-affine} available in the literature before his paper.
\end{remark}

We need the following lemma.

\begin{lemma} [{\cite[Proposition 8.44]{hall2015gtm222}}] \label{lem-dom-convhull}
    Suppose $\lambda, \mu \in \overline{C_+}$.
    Then the following are equivalent:
    \begin{enumerate}
        \item \label{lem-dom-convhull-1} $\mu \preceq \lambda$.
        \item \label{lem-dom-convhull-2} $\mu \in \operatorname{Conv}\{w \lambda \mid w \in W_f\}$, the convex hull of the finite set $W_f \lambda$ in~$E$. 
    \end{enumerate}
\end{lemma}

By combining the lemmas above, we have the following important result comparing the partial orders $\le$ on~$\prescript{f}{}{W}$ and $\preceq$ on~$\mathbb{Z} \Phi\spcheck$.
This will allow us to translate questions involving the Bruhat--Chevalley order on~$\prescript{f}{}{W}$ into questions involving convex geometry.

\begin{theorem} \label{lem-lattice-bruhat}
    Suppose $\lambda, \mu \in \mathbb{Z} \Phi\spcheck \cap \overline{C_+}$.
        The following are equivalent:
        \begin{enumerate} 
            \item \label{lem-lattice-bruhat-1} $t_\mu \le t_\lambda$ in the Bruhat--Chevalley order.
            \item \label{lem-lattice-bruhat-a} $t_\mu \le_a t_\lambda$ in the affine order.
            \item \label{lem-lattice-bruhat-2} $\mu \preceq \lambda$.
            \item \label{lem-lattice-bruhat-3} $\mu \in \operatorname{Conv}\{w \lambda \mid w \in W_f\}$. 
        \end{enumerate}
\end{theorem}

\begin{proof} 
    \ref{lem-lattice-bruhat-1} $\Leftrightarrow$ \ref{lem-lattice-bruhat-a}:
    Since $\lambda, \mu$ are dominant, we have $t_\lambda, t_\mu \in \prescript{f}{}{W}$ by Proposition~\ref{prop-dom-alcove}.
    Then the equivalence between~\ref{lem-lattice-bruhat-1} and~\ref{lem-lattice-bruhat-a} follows from Lemma~\ref{lem-dom-order-eq-affine}.

    \ref{lem-lattice-bruhat-a} $\Leftrightarrow$ \ref{lem-lattice-bruhat-2}: This is a particular case of Lemma~\ref{lem-lattice-affine-dominance-order}.

    \ref{lem-lattice-bruhat-2} $\Leftrightarrow$ \ref{lem-lattice-bruhat-3}:
    This is a particular case of Lemma~\ref{lem-dom-convhull}.    
\end{proof}

\begin{remark}
    The equivalences in Theorem~\ref{lem-lattice-bruhat} are not new.
    Similar statements appear in~\cite[Sect.\@~2]{lusztig1983singularities} and~\cite[Sect.\@~0 and Theorem~4.10]{stembridge2006}.
\end{remark}

The implication ``\ref{lem-lattice-bruhat-1} $\Rightarrow$~\ref{lem-lattice-bruhat-2}'' in Theorem~\ref{lem-lattice-bruhat} is also a particular case of the following proposition, whose proof also uses Lemma~\ref{lem-dom-order-eq-affine}.

\begin{proposition} \label{prop-element-to-lattice} 
    Suppose $x \in \prescript{f}{}{W} \cap \left( t_\lambda W_f \right)$ and $y \in \prescript{f}{}{W} \cap \left( t_\mu W_f \right)$.
    If $y \le x$, then $\mu \preceq \lambda$.
\end{proposition}

\begin{proof}
    Suppose $x = t_\lambda w$ and $y = t_\mu v$ where $w,v \in W_f$.
    If $y \le x$, then by Lemma~\ref{lem-dom-order-eq-affine}, we have $y \le_a x$. 
    By Corollary~\ref{cor-affine-order}, for any $u \in A_x$, the vector $u - yx^{-1} (u)$ is an $\mathbb{R}_{\ge 0}$-linear combination of simple coroots.
    Note that $\lambda \in \overline{A_x} = t_\lambda w \overline{A_+}$ and $\mu \in \overline{A_y} = t_\mu v \overline{A_+}$.
    We have 
    \begin{equation*}
      \lim_{u \in A_x, u \to \lambda} \bigl ( u - yx^{-1} (u) \bigr ) = \lambda - yx^{-1} (\lambda) = \lambda - t_\mu v w^{-1} t_{-\lambda} (\lambda) = \lambda - \mu,
    \end{equation*}
    and it is an $\mathbb{R}_{\ge 0}$-linear combination of simple coroots.
    Therefore, $\mu \preceq \lambda$.
\end{proof}

See Remark~\ref{rmk-closure} for the geometric interpretation of the coset $t_\lambda W_f$.

\section{Dominant elements} \label{sec-dom-lattice-formula}

For the rest of the paper, $\lambda \in \mathbb{Z} \Phi\spcheck \cap \overline{C_+}$ is a fixed dominant coroot lattice point. 

\subsection{The polytope \texorpdfstring{$P^\lambda$}{P-lambda} and the dominant lower interval \texorpdfstring{$\prescript{f}{}{[e,t_{\lambda}]}$}{f-e-t-lambda}} \label{subsec-def-Plambda}

As stated in the introduction, we are interested in the sequence $(\prescript{f}{}{b}_i^{t_{\lambda}})_i$. Recall that
\begin{align*}
    \prescript{f}{}{b}_i^{t_{\lambda}} & = \operatorname{Card} \left\{x \in \prescript{f}{}{W} \xmiddle| x \le t_\lambda, \ell(x) = i \right\} \\
    & = \operatorname{Card} \left\{x \in W \xmiddle| x \le t_\lambda, \ell(x) = i, A_x \subset C_+ \right\},
\end{align*}
where $i\in\{0, \ldots, \ell(t_{\lambda})\}$.
The second equality is due to Proposition~\ref{prop-dom-alcove}.
Instead of studying these numbers directly, we first study the partner sequence $(b_i^{\lambda})_i$, given by
\begin{align*}
    b_i^\lambda  :=& \operatorname{Card} \left\{\mu \in \mathbb{Z} \Phi\spcheck \cap \overline{C_+} \xmiddle| \mu \preceq \lambda, (2\rho | \mu) = i \right\} \\
     =& \operatorname{Card} \left\{\mu \in \mathbb{Z} \Phi\spcheck \cap \overline{C_+} \xmiddle| t_\mu \le t_\lambda, \ell(t_\mu) = i \right\}.
\end{align*}
The second equality is due to Lemma~\ref{lem-length-trans} and Theorem~\ref{lem-lattice-bruhat}.
For later use, we introduce $(b_{i,+}^\lambda)_i$, the ``strongly dominant'' version of $(b_i^\lambda)_i$, given by 
\begin{equation*}
    b^\lambda_{i,+} := \operatorname{Card} \left\{\mu \in \mathbb{Z} \Phi\spcheck \cap C_+ \xmiddle| \mu \preceq \lambda, (2\rho | \mu) = i \right\}.
\end{equation*}

We define the convex polytope $P^\lambda$ and its ``strongly dominant'' counterpart $P_+^\lambda$:
\begin{align*}
    P^\lambda &:= \operatorname{Conv}\{w \lambda \mid w \in W_f\} \cap \overline{C_+} \subset E, \\
    P_+^\lambda &:= \operatorname{Conv}\{w \lambda \mid w \in W_f\} \cap C_+ \subset E,
\end{align*}
where $\operatorname{Conv}\{w \lambda \mid w \in W_f\}$ is the convex hull of the finite set $\{w \lambda \mid w \in W_f\}$ in~$E$.
Note that $P^\lambda$ is a bounded closed subset of $E$, and $\overline{P_{+}^\lambda} = P^\lambda$.

We obtain the following proposition from Lemma~\ref{lem-dom-convhull} and Theorem~\ref{lem-lattice-bruhat}, which motivates the definition of the polytope $P^\lambda$.

\begin{proposition} \label{prop-polytope}
  We have the following four equalities:
  \begin{gather*}
      P^\lambda  =\overline{C_+} \cap \left\{\lambda - \sum_{i=1}^r c_i \alpha_i\spcheck \Biggm\vert c_i \in \mathbb{R}_{\ge 0} \right\}, \\
      P_+^\lambda  = C_+ \cap \left\{\lambda - \sum_{i=1}^r c_i \alpha_i\spcheck \Biggm\vert c_i \in \mathbb{R}_{\ge 0} \right\}, \\
      P^\lambda \cap \mathbb{Z} \Phi\spcheck  =\left\{\mu \in \mathbb{Z} \Phi\spcheck \cap \overline{C_+} \xmiddle| \mu \preceq \lambda \right\}, \\
      P^\lambda_+ \cap \mathbb{Z} \Phi\spcheck  =\left\{\mu \in \mathbb{Z} \Phi\spcheck \cap C_+ \xmiddle| \mu \preceq \lambda \right\}. 
  \end{gather*}
\end{proposition}

\begin{proof}
    If $\mu \in \overline{C_+}$, then $\mu \preceq \lambda$ if and only if $\mu \in \operatorname{Conv}\{w \lambda \mid w \in W_f\}$ by Lemma~\ref{lem-dom-convhull}.
    Therefore,
    \[P^\lambda = \operatorname{Conv}\{w \lambda \mid w \in W_f\} \cap \overline{C_+} = \overline{C_+} \cap \{\mu \in E \mid \mu \preceq \lambda\}.\]
    This proves the first equality. 
    The other equalities hold by similar reasons.
\end{proof}

\begin{remark} \label{rmk-b-i-lambda}
    By definition, the number $b_i^\lambda$ is the cardinality of  $P^\lambda \cap \mathbb{Z} \Phi\spcheck \cap H_{2\rho,i}$.
    In other words, $b_i^\lambda$ counts the number of lattice points in the slice of $P^\lambda$ cut by the hyperplane $H_{2\rho,i}$.
    On the other hand, since $(2\rho|\mu)$ is an even integer for any $\mu \in \mathbb{Z} \Phi\spcheck$, each element in~$P^\lambda \cap \mathbb{Z} \Phi \spcheck$ is counted by $b_i^\lambda$ for a unique $i$, and also $b_i^\lambda = 0$ whenever $i$ is odd. 
    A similar observation also applies to $b_{i,+}^\lambda$ which counts the points in~$P_+^\lambda \cap \mathbb{Z} \Phi\spcheck \cap H_{2\rho,i}$ (that is, omitting those points on the walls of $C_+$). Obviously, $b_{i,+}^\lambda \le b_i^\lambda$.
\end{remark}

\begin{remark} \label{rem-cube} 
For $\nu \in \overline{C_+}$ (not necessarily a coroot lattice element), we can also define the polytope $P^\nu$. If $\nu$ is strongly dominant, then the face structure of the polytope $P^\nu$ has a simple description~\cite{BGHpolytope}: $P^\nu$ is combinatorially equivalent to an $r$-dimensional cube with $2^r$ vertices 
\[\left\{\nu - \sum\nolimits_{j \in J} c_j \alpha_j\spcheck \xmiddle| J \subseteq [r] \right\}, \]
where $(c_j)_{j \in J} = M_J^{-1} \cdot \bigl( (\alpha_j | \lambda) \bigr)_{j \in J}$ and $M_J$ is the square submatrix of the Cartan matrix $M$ corresponding to the index set $J$.
Thus, the combinatorial type of $P^\nu$ depends only on the rank $r$ of $\Phi$, and its vertices can be computed explicitly from the Cartan matrix.
On the other hand, if $\nu$ lies on some of the walls, the structure of the polytope $P^\nu$ is more complicated and not well explored.
However, for the present paper, the fact that $P^\lambda$ is bounded and convex is sufficient.
\end{remark}

By definition of $P^\lambda$ or by Proposition~\ref{prop-polytope}, $P^{k \lambda}$ is the $k$-fold dilation of $P^\lambda$:

\begin{corollary}
  $P^{k\lambda} = k P^\lambda := \{k x \mid x \in P^\lambda\}$, for $k \in \mathbb{N}$.
\end{corollary}

For $\mu \in \mathbb{Z} \Phi\spcheck$, we denote by $W_{f,\mu}$ the parabolic subgroup of $W_f$ generated by $\{s_i \mid 1 \le i \le r, (\mu| \alpha_i)=0\}$, and by $\prescript{\mu}{}{W_f}$ the set of minimal representatives for the right cosets $W_{f,\mu} \backslash W_f$.
Then, $W_{f,\mu}$ is the stabilizer of $\mu$ in~$W_f$.
In particular, if $\mu \in C_+$, then $W_{f,\mu}$ is trivial and $\prescript{\mu}{}{W_f} = W_f$.

The following lemma is needed to prove the dominant lattice formula (Theorem~\ref{thm-lattice-formula-in-the-introduction}).

\begin{lemma} \label{lem-lattice-onwall}
    Suppose $\mu \in P^\lambda \cap \mathbb{Z} \Phi\spcheck$.
    Let $w = t_\mu w_f \in t_\mu W_f$ be arbitrary, where $w_f \in W_f$.
    Then, $w \in  \prescript{f}{}{W}$ if and only if $w_f \in \prescript{\mu}{}{W_f}$.
    In this case, we have $w \le t_\mu \le t_\lambda$ and $\ell(w) = \ell(t_\mu) - \ell(w_f)$.

In particular, if $\mu \in P_+^\lambda \cap \mathbb{Z} \Phi\spcheck$, then we always have $w \le t_\mu$, $w \in \prescript{f}{}{W}$, and $\ell(w) = \ell(t_\mu) - \ell(w_f)$.
\end{lemma}

\begin{proof}
    We choose $\varepsilon$ small enough so that $\varepsilon \rho \spcheck \in A_+$ (see Corollary~\ref{cor-rho}).
    Then 
    \[w (\varepsilon \rho \spcheck) = \varepsilon w_f (\rho \spcheck) + \mu \in A_{w}.\]
    For any $i = 1, \dots, r$, we have 
    \begin{equation*}
        \bigl(w(\varepsilon\rho \spcheck) \big| \alpha_i\bigr)  = \varepsilon \bigl( w_f( \rho \spcheck) \big| \alpha_i\bigr) + (\mu | \alpha_i) = \varepsilon \bigl( \rho \spcheck \big| w_f^{-1} (\alpha_i) \bigr) + (\mu | \alpha_i).
    \end{equation*}
    Let $I_\mu := \{i \mid 1 \le i \le r, (\mu | \alpha_i) = 0\}$.
    If $i \notin I_\mu$, then $(\mu | \alpha_i) \ge 1$ and $(w(\varepsilon\rho \spcheck) | \alpha_i) > 0$.
    If otherwise $i \in I_\mu$, then $(\mu | \alpha_i) = 0$.
    In this case, $(w(\varepsilon\rho \spcheck) | \alpha_i) > 0$ if and only if $w_f^{-1} (\alpha_i) \in \Phi^+$.
    Therefore, we have the following equivalences:
    \begin{align*}
        w \in \prescript{f}{}{W} & \iff A_w \subset C_+ \text{ (by Proposition~\ref{prop-dom-alcove}),} \\
        & \iff \bigl(w (\varepsilon \rho \spcheck) \big| \alpha_i \bigr) > 0, \text{ for all } i = 1, \dots, r, \\
        & \iff \bigl(w (\varepsilon \rho \spcheck) \big| \alpha_i \bigr) > 0, \text{ for all } i \in I_\mu, \\
        & \iff w_f^{-1} (\alpha_i) \in \Phi^+ \text{ for all } i \in I_\mu, \\
        & \iff s_i w_f > w_f \text{ for all } i \in I_\mu, \\
        & \iff w_f \in \prescript{\mu}{}{W_f}.
    \end{align*}
    This proves the first assertion.

    Next, we show that for any $w_f \in \prescript{\mu}{}{W_f}$ we always have $t_\mu w_f \le t_\mu$.
    
    Let $w_f = s_{i_1} \cdots s_{i_l}$ be a reduced expression for $w_f$, where $1 \le i_1, \dots, i_l \le r$.
    For any $k = 1, \dots, l$, we have 
    \begin{equation}  \label{eq-lattice-onwall-2}
        s_{i_{k-1}} \cdots s_{i_1} t_\mu^{-1} (\varepsilon\rho \spcheck) = s_{i_{k-1}} \cdots s_{i_1} (\varepsilon\rho \spcheck - \mu) \in s_{i_{k-1}} \cdots s_{i_1} t_\mu^{-1} A_+,
    \end{equation}
    and
    \begin{equation} \label{eq-lattice-onwall-1}
        \bigl(s_{i_{k-1}} \cdots s_{i_1} (\varepsilon\rho \spcheck - \mu) \big| \alpha_{i_k}\bigr) = \varepsilon \bigl( \rho \spcheck \big| s_{i_1} \cdots s_{i_{k-1}} (\alpha_{i_k}) \bigr) - \bigl( \mu \big| s_{i_1} \cdots s_{i_{k-1}} (\alpha_{i_k}) \bigr).
    \end{equation}
    Note that $s_{i_1} \cdots s_{i_{k-1}} (\alpha_{i_k}) \in \Phi^+$ since $s_{i_1} \cdots s_{i_{k-1}} s_{i_k} > s_{i_1} \cdots s_{i_{k-1}}$.
    We write 
    \begin{equation*}
        \gamma_k := s_{i_1} \cdots s_{i_{k-1}} (\alpha_{i_k}) = \sum_{i = 1}^r c_{k,i} \alpha_i, \quad (c_{k,i} \in \mathbb{N}).
    \end{equation*}
    
    We claim that for any $k$, there exists $j \notin I_\mu$ such that $c_{k,j} \ne 0$.
    Otherwise, say, $\gamma_k = \sum_{i \in I_\mu} c_{k,i} \alpha_i$.
    For all $i \in I_\mu$ we have $w_f^{-1} \alpha_i \in \Phi^+$ since $w_f^{-1} s_i > w_f^{-1}$.
    Thus $w_f^{-1} (\gamma_k) \in \Phi^+$.
    However,
    \begin{equation*}
        w_f^{-1} (\gamma_k) = s_{i_l} \cdots s_{i_1} \cdot s_{i_1} \cdots s_{i_{k-1}} (\alpha_{i_k}) = s_{i_l} \cdots s_{i_{k}} (\alpha_{i_k}) \in \Phi^-,
    \end{equation*}
    which is a contradiction.
    Therefore, $c_{k,j} > 0$ for some $j \notin I_\mu$.

    Let $j$ be as above (may depend on~$k$).
    Note that $\mu \in \overline{C_+}$.
    Thus, $(\mu | \gamma_k) \ge (\mu | c_{k,j} \alpha_{j}) \ge (\mu | \alpha_{j}) \ge 1$.
    By Equation  \eqref{eq-lattice-onwall-1}, we have 
    \begin{equation*}
    \bigl(s_{i_{k-1}} \cdots s_{i_1} (\varepsilon\rho \spcheck - \mu) \big| \alpha_{i_k}\bigr) = \varepsilon \bigl( \rho \spcheck \big| s_{i_1} \cdots s_{i_{k-1}} (\alpha_{i_k}) \bigr) - (\mu | \gamma_k) < 0.
    \end{equation*}
    By \eqref{eq-lattice-onwall-2}, this implies that the alcoves $s_{i_{k-1}} \cdots s_{i_1} t_\mu^{-1} A_+$ and $A_+$ lie in different sides of $H_{\alpha_{i_k},0}$.
    Equivalently, $A_+$ and $t_\mu s_{i_1} \cdots s_{i_{k-1}} A_+$ lie in different sides of the affine hyperplane $t_\mu s_{i_1} \cdots s_{i_{k-1}} H_{\alpha_{i_k},0}$.
    By Lemma~\ref{lem-alcove-bruhat-order}\ref{lem-alcove-bruhat-order-3}, we have
    \begin{equation} \label{eq-lattice-onwall-3}
        t_\mu s_{i_1} \cdots s_{i_{k-1}} s_{i_k} < t_\mu s_{i_1} \cdots s_{i_{k-1}} \text{ (for all $k = 1, \dots, l$).} 
    \end{equation}
    Consequently, we have $t_\mu w_f \le t_\mu$ as we wanted. 
    Moreover, \eqref{eq-lattice-onwall-3} also yields the equation $\ell(t_\mu w_f) = \ell(t_\mu) - l$ where $l = \ell(w_f)$. 
    At last, $t_\mu \le t_\lambda$ by Theorem~\ref{lem-lattice-bruhat} and Proposition~\ref{prop-polytope}. 
    This completes the proof.
\end{proof}

We end this subsection with a corollary of the above lemma, which describes the interval $\prescript{f}{}{[e,t_{\lambda}]}$ in terms of the lattice points in~$P^\lambda$.

\begin{corollary} \label{cor-for-dom-lat-for}
  Let $\lambda \in \mathbb{Z} \Phi\spcheck \cap \overline{C_+}$. We have
  \begin{equation*}
    \prescript{f}{}{[e,t_{\lambda}]} = \{t_\mu w_f\in W \mid \mu \in P^\lambda\cap \mathbb{Z} \Phi \spcheck, w_f\in \prescript{\mu}{}{W_f} \}.
  \end{equation*}
\end{corollary}

\begin{proof}
    Any element $w \in \prescript{f}{}{[e, t_\lambda]}$ can be uniquely written in the form $w = t_\mu w_f$ where $\mu \in \mathbb{Z} \Phi\spcheck$ and $w_f \in W_f$.
    By Proposition~\ref{prop-element-to-lattice}, we have $\mu \preceq \lambda$.
    Moreover, $\mu \in \overline{C_+}$ since $A_w \subset C_+$ (by Proposition~\ref{prop-dom-alcove}) and $\mu \in \overline{A_w}$.
    Therefore, $\mu \in P^\lambda \cap \mathbb{Z} \Phi\spcheck$ by Proposition~\ref{prop-polytope}.
    By Lemma~\ref{lem-lattice-onwall}, we have $w_f \in \prescript{\mu}{}{W_f}$. 
    This proves ``$\subseteq$''.

    On the other hand, for any $\mu \in P^\lambda \cap \mathbb{Z} \Phi\spcheck$ and any $w_f \in \prescript{\mu}{}{W_f}$, we have $t_\mu w_f \le t_\lambda$ and $t_\mu w_f \in \prescript{f}{}{W}$ by Lemma~\ref{lem-lattice-onwall} again.
    That is, $t_\mu w_f \in \prescript{f}{}{[e,t_\lambda]}$. 
    This proves ``$\supseteq$''.
    \end{proof}
 
\subsection{The dominant lattice formula}

We define the \emph{Poincar\'e polynomials} $\pi^{t_\lambda}(q)$,  $\pi^\lambda(q)$ and $\pi^\lambda_+ (q)$ of the sequences $(\prescript{f}{}{b}_i^{t_\lambda})_i$, $(b_i^\lambda)_i$ and $( b_{i,+}^\lambda )_i$, respectively, by 
\begin{align*}
    \pi^{t_\lambda}(q)  &:= \sum_{0 \le i \le \ell(t_\lambda)} \prescript{f}{}{b}_i^{t_\lambda} q^i = \sum_{w \in \prescript{f}{}{[e, t_\lambda]}} q^{\ell(w)}, 
    \\
    \pi^\lambda(q)  &:= \sum_{0 \le i \le \ell(t_\lambda)} b_i^\lambda q^i = \sum_{\mu \in P^\lambda \cap \mathbb{Z} \Phi\spcheck } q^{(2\rho | \mu)}, \\
    \pi^\lambda_+ (q) &:= \sum_{0 \le i \le \ell(t_\lambda)} b_{i,+}^\lambda q^i = \sum_{\mu \in P^\lambda_+ \cap \mathbb{Z} \Phi\spcheck } q^{(2\rho | \mu)}.
\end{align*}
The \emph{Poincar\'e polynomial} $\pi_f(q)$ of the finite Weyl group~$W_f$ is defined to be 
\begin{equation*}
    \pi_f(q) := \sum_{w \in W_f} q^{\ell(w)} = \sum_{i} b_{i, f} q^i,
\end{equation*}
where
\[b_{i,f} := \operatorname{Card} \{w \in W_f \mid \ell(w) = i\}.\]
For $\mu \in \mathbb{Z} \Phi\spcheck$, we also define the \emph{Poincar\'e polynomial} $\prescript{\mu}{}{\pi_f}$ of the set $\prescript{\mu}{}{W_f}$ by
\[\prescript{\mu}{}{\pi_f}(q) := \sum_{w \in \prescript{\mu}{}{W_f}} q^{\ell(w)}.\]
In particular, if $\mu \in \mathbb{Z} \Phi\spcheck \cap C_+$, then $\prescript{\mu}{}{\pi_f} = \pi_f$.

Obviously $\pi^{t_\lambda}(q)$, $\pi^\lambda(q)$, $\pi_+^\lambda(q)$ and $\prescript{\mu}{}{\pi_f}(q)$ belong to $\mathbb{N}[q]$. 

\begin{remark}
The polynomials $\pi^{t_\lambda}(q)$, $\pi_f(q)$, and $\prescript{\mu}{}{\pi_f} (q)$ are indeed the Poincar\'e polynomials for the ordinary cohomology of the corresponding spherical Schubert variety, finite flag variety, and finite partial flag variety, respectively. 
However, we are not aware of any topological interpretation of the polynomials $\pi^\lambda(q)$ and $\pi_+^\lambda(q)$, and they are not palindromic or unimodal in general.
\end{remark}

We are ready to prove the dominant lattice formula (Theorem~\ref{thm-lattice-formula-in-the-introduction}):
\[\pi^{t_\lambda} (q) = \sum_{\mu \in P^\lambda \cap \mathbb{Z} \Phi\spcheck} q^{(2\rho | \mu)} \cdot \prescript{\mu}{}{\pi_f}(q^{-1}).\]

\begin{proof}[Proof of Theorem~\ref{thm-lattice-formula-in-the-introduction}]
    We have
    \begin{align*}
        \pi^{t_\lambda} (q) =  \sum_{w \in \prescript{f}{}{[e, t_\lambda]}} q^{\ell(w)} & = \sum_{\mu \in P^\lambda \cap \mathbb{Z} \Phi\spcheck, w_f \in \prescript{\mu}{}{W_f}} q^{\ell(t_\mu w_f)} \\
        & = \sum_{\mu \in P^\lambda \cap \mathbb{Z} \Phi\spcheck, w_f \in \prescript{\mu}{}{W_f}} q^{\ell(t_\mu)} \cdot q^{-\ell(w_f)} \\
        & = \sum_{\mu \in P^\lambda \cap \mathbb{Z} \Phi\spcheck} q^{\ell(t_\mu)} \cdot \sum_{w_f \in \prescript{\mu}{}{W_f}} q^{-\ell(w_f)} \\
        & = \sum_{\mu \in P^\lambda \cap \mathbb{Z} \Phi\spcheck} q^{(2\rho | \mu)} \cdot \prescript{\mu}{}{\pi_f} (q^{-1}). 
    \end{align*}
    The second equality is due to Corollary~\ref{cor-for-dom-lat-for}, and the third one is due to Lemma~\ref{lem-lattice-onwall}.
\end{proof}
 For an illustration of the formula in the case $\Phi$ of type $A_2$, see Figure~\ref{fig: similarity}.

\begin{remark}
The dominant lattice formula can be proved and understood geometrically.
Recall from Sections~\ref{subsec-spherical} and~\ref{subsec-intro-dominant} that the polynomial $\pi^{t_\lambda} (q)$ is equal to the Poincar\'e polynomial of the singular cohomology of the spherical Schubert variety $\overline{\mathcal{G} r_{\lambda}}$.
We recall the $K$-orbit decomposition,
\begin{equation*}
    \overline{\mathcal{G} r_\lambda}=\bigsqcup_{\mu \in \mathbb{Z} \Psi^{\vee} \cap \overline{C_{+}},\ \mu \preceq \lambda} \mathcal{G} r_\mu. 
\end{equation*}
Let $P_\mu\subset G$ be the parabolic subgroup generated by the root subgroups for simple roots $\alpha$ satisfying $\left(\mu | \alpha\right)=0$. 
Note that each $\mathcal{G} r_{\mu}$ has dimension $\left(2\rho | \mu\right)$ and is an affine space bundle over the finite partial flag variety $G/P_\mu$~\cite[Section 2]{zhu2016introduction}.
In particular, the Poincar\'e polynomial of $\mathcal{G} r_{\mu}$ is the same as the Poincar\'e polynomial of $G/P_\mu$, which is given by $\prescript{\mu}{}{\pi_f}\left(q\right)$. 
By Poincar\'e duality, $\prescript{\mu}{}{\pi_f}\left(q\right)$ is palindromic, so it differs from $q^{(2\rho | \mu)} \cdot \prescript{\mu}{}{\pi_f} (q^{-1})$ only by a power of $q$.
Since the odd cohomology of each $\mathcal{G} r_{\mu}$ is zero, the long exact sequence of cohomology splits, and this gives the dominance lattice formula. 
A similar argument can be found in~\cite[Section 2(iii)]{MR3952347}.
\end{remark}

\begin{figure}[ht]
    \centering
    \includegraphics[scale=0.9]{"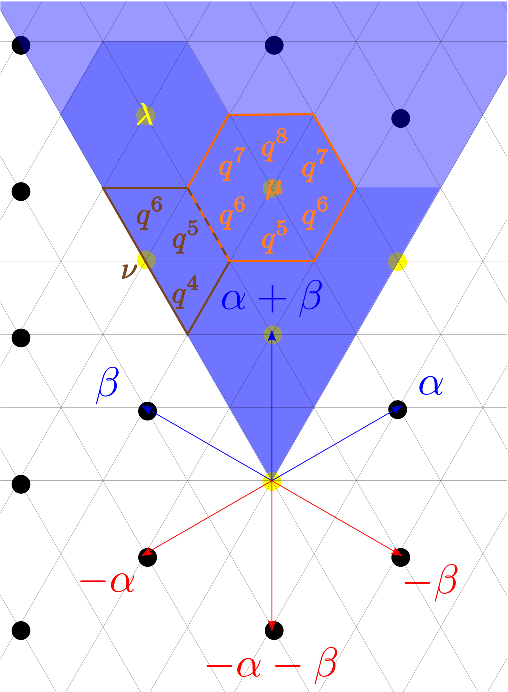"}
  \caption{Description of two of the summands of the dominant lattice formula when $W$ is of affine type $A_2$ and $\lambda=2\alpha+3\beta$, where $\alpha:=\alpha_1\spcheck$ and $\beta:=\alpha_2\spcheck$.
  The yellow points are the lattice points inside $P^\lambda$. The dominant Weyl chamber is blue-colored.
  The alcoves of the interval $\prescript{f}{}{[e,t_{\lambda}]}$ are colored with a darker blue. There are $6$ dominant alcoves arranged around the strongly dominant lattice point $\mu := 2 \alpha + 2 \beta$, and $3$ around the lattice point $\nu := \alpha + 2 \beta$ which is on the wall.  
  The summand corresponding to $\mu$ in the formula is given by 
  $q^8\cdot\prescript{\mu}{}{\pi_f}(q^{-1})=q^5+2q^6+2q^7+q^8$.
  The terms of this polynomial are colored orange and placed in the corresponding alcoves in the picture. 
  The summand corresponding to $\nu$ is given by 
  $q^6\cdot\prescript{\nu}{}{\pi_f}(q^{-1})=q^4+q^5+q^6$, 
  whose terms are colored brown.}
  \label{fig: similarity} 
\end{figure}

\subsection{Truncations} \label{subsec-truncation}
In this subsection, we introduce truncations of Laurent polynomials, which will be used in the proof of Theorem~\ref{thm-main-in-the-introduction}\ref{thm-main-1-in-the-introduction}. At the end of this subsection, we prove a useful consequence (Corollary~\ref{cor-lattice-formula}) of the dominant lattice formula.

For two Laurent polynomials $g(q), h(q) \in \mathbb{N}[q^{\pm 1}]$, 
we write $g(q) \le h(q)$ if $g_i \le h_i$ for any degree $i$ where $g_i, h_i \in \mathbb{N}$ are coefficients of $q^i$ in~$g(q)$ and $h(q)$, respectively.
Moreover, for any integer $n \in \mathbb{Z}$, we define the \emph{truncation} $T^{\le n} g(q)$ of $g(q)$ to be the Laurent polynomial given by 
\begin{equation*}
    T^{\le n} g(q) := \sum _{i \le n} g_i q^i.
\end{equation*}
For a real number $z$, we write $T^{\le z} g(q)$ to indicate the truncation $T^{\le \lfloor z \rfloor} g(q)$, where $\lfloor z \rfloor$ is the largest integer less than or equal to $z$.
By abuse of notation, we write $T^{\le z} g(1)$ to denote the evaluation of the Laurent polynomial $T^{\le z} g(q)$ at $q=1$. We need the following two lemmas.

\begin{lemma} \label{lem-laurent-truncation} Let $f(q), g(q), h(q) \in \mathbb{N}[q^{\pm 1}]$, we have:
    \begin{enumerate}
        \item \label{lem-laurent-truncation-1} If $g(q) \le h(q)$, then $T^{\le z} g(q) \le T^{\le z} h(q)$ for any $z \in \mathbb{R}$. Moreover, we have $g(1) \le h(1)$ and $T^{\le z} g(1) \le T^{\le z} h(1)$.
        \item \label{lem-laurent-truncation-2} If $g(q) \le h(q)$, then we have $f(q) \cdot g(q) \le f(q) \cdot h(q)$ and $f(q) + g(q) \le f(q) + h(q)$.
    \end{enumerate} 
\end{lemma}

\begin{proof}
    Obvious.
\end{proof}

\begin{lemma} \label{lem-laurent-truncation-3}
    Suppose $h(q) = \sum_{-l \leq i\le 0} h_i q^i \in \mathbb{N}[q^{- 1}]$, where $l \ge 1$ is an integer.
    Then, for any $g(q) \in \mathbb{N}[q^{\pm 1}]$ and $z \in \mathbb{R}$ we have 
    \begin{equation*}
        \left( T^{\le z} g(q) \right) \cdot h(q) \le T^{\le z} \bigl( g(q) \cdot h(q) \bigr) \le \left( T^{\le z + l} g(q) \right) \cdot h(q).
    \end{equation*}
\end{lemma}

\begin{proof}
    We write $g(q) = \sum_i g_i q^i$ where each $g_i \in \mathbb{N}$.
    Then 
    \begin{equation*}
        g(q) \cdot h(q) = \sum_k \Bigl( \sum_{i + j = k} g_i h_j \Bigr) q^k.
    \end{equation*}
    We denote by $(gh)_k := \sum_{i + j = k} g_i h_j$ the coefficient of $q^k$ in~$g(q) \cdot h(q)$.
    It is clear that the exponents of the $q$-powers of $\left( T^{\le z} g(q) \right) \cdot h(q)$ concentrate on the interval $(-\infty, z]$.
    For any integer $k \in (-\infty, z]$, the coefficient of $q^k$ in~$\left( T^{\le z} g(q) \right) \cdot h(q)$, say, $a_k$, satisfies
    \begin{equation*}
        a_k = \sum_{i \le z, i + j = k} g_i h_j \le \sum_{i + j = k} g_i h_j = (gh)_k.
    \end{equation*}
    This proves the first inequality $\left( T^{\le z} g(q) \right) \cdot h(q) \le T^{\le z} \bigl( g(q) \cdot h(q) \bigr)$.

    On the other hand, for any integer $k \in (-\infty, z]$, the restrictions $i+j = k$ and $-l \le j \le 0$ force $i \le z + l$.
    Therefore, we have
    \begin{equation*}
        (gh)_k = \sum_{i + j = k} g_i h_j = \sum_{i \le z + l, i + j = k} g_i h_j 
    \end{equation*}
    which equals the coefficient of $q^k$ ($k \le z$) in~$\left( T^{\le z + l} g(q) \right) \cdot h(q)$.
    While for $k > z$, the coefficient of $q^k$ in~$T^{\le z} \bigl( g(q) \cdot h(q) \bigr)$ is zero.
    Thus, 
    \[T^{\le z} \bigl( g(q) \cdot h(q) \bigr) \le \left( T^{\le z + l} g(q) \right) \cdot h(q)\]
    which is the second inequality.
\end{proof}

Recall that $\pi^{t_\lambda}(q), \pi^\lambda(q), \pi_+^\lambda(q), \prescript{\mu}{}{\pi_f}(q)$ are polynomials in~$\mathbb{N}[q]$.
We regard them as Laurent polynomials in~$\mathbb{N}[q^{\pm 1}]$. The following corollary of the dominant lattice formula is crucial in the proof of Theorem~\ref{thm-main-in-the-introduction}.

\begin{corollary} \label{cor-lattice-formula}
    $\pi_+^\lambda(q) \cdot \pi_f(q^{-1}) \le \pi^{t_\lambda} (q) \le \pi^\lambda(q) \cdot \pi_f(q^{-1})$.
\end{corollary}

\begin{proof}
    By definition, whatever $\mu$ is, we always have $\prescript{\mu}{}{\pi_f} (q^{-1}) \le \pi_f(q^{-1})$.
    Therefore, 
    \begin{align*}
       \pi^{t_\lambda} (q) & = \sum_{\mu \in P^\lambda \cap \mathbb{Z} \Phi\spcheck} q^{(2\rho | \mu)} \cdot \prescript{\mu}{}{\pi_f}(q^{-1}) \\
       & \le \sum_{\mu \in P^\lambda \cap \mathbb{Z} \Phi\spcheck} q^{(2\rho | \mu)} \cdot \pi_f(q^{-1}) = \pi^\lambda(q) \cdot \pi_f(q^{-1}),
    \end{align*}
    where the first equality is the dominant lattice formula, and the inequality follows from Lemma~\ref{lem-laurent-truncation}\ref{lem-laurent-truncation-2}.

    On the other hand, $\prescript{\mu}{}{\pi_f} (q^{-1}) = \pi_f(q^{-1})$ if $\mu \in \mathbb{Z} \Phi\spcheck \cap C_+$.
    Thus,
    \begin{align*}
        \pi_+^\lambda(q) \cdot \pi_f(q^{-1}) & = \sum_{\mu \in P^\lambda_+ \cap \mathbb{Z} \Phi\spcheck} q^{(2\rho | \mu)} \cdot \pi_f(q^{-1}) \\
        & = \sum_{\mu \in P^\lambda_+ \cap \mathbb{Z} \Phi\spcheck} q^{(2\rho | \mu)} \cdot \prescript{\mu}{}{\pi_f}(q^{-1}) \\
        & \le \sum_{\mu \in P^\lambda \cap \mathbb{Z} \Phi\spcheck} q^{(2\rho | \mu)} \cdot \prescript{\mu}{}{\pi_f}(q^{-1}) \\ 
        & = \pi^{t_\lambda} (q).
    \end{align*}
    This completes the proof.
\end{proof}

\section{The weak convergence of \texorpdfstring{$(\mathfrak{m}_1, \mathfrak{m}_2,\ldots)$}{m1, m2, ...}} \label{sec-main}

Let $\lambda \in \mathbb{Z} \Phi\spcheck \cap \overline{C_+}$ be a fixed dominant coroot lattice point as before. Recall that $\operatorname{Vol}_r$ denotes the Lebesgue measure on~$E$ induced by the inner product $(-,-)$.
Let $\operatorname{ht}\colon P^\lambda \to \mathbb{R}$ (``ht'' for ``height'')  be the linear map $\operatorname{ht} (x) := (2\rho|x)$.
The push-forward measure $\mathrm{ht}_*\mathrm{Vol}_r$ is defined as follows, for any Borel set $U \subseteq \mathbb{R}$,
\begin{equation*}
    (\mathrm{ht}_*\mathrm{Vol}_r)(U) := \operatorname{Vol}_r (\operatorname{ht}^{-1} U) = \operatorname{Vol}_r \left( \left\{x \in P^\lambda \xmiddle| (2\rho | x) \in U \right\} \right).
\end{equation*}
Note that by Proposition~\ref{prop-polytope} we have $0 \le \operatorname{ht} (x)  \le (2\rho|\lambda) = \ell(t_\lambda)$ for any $x \in P^\lambda$.
Therefore, $\mathrm{ht}_*\mathrm{Vol}_r$ is supported on~$[0, \ell(t_\lambda)]$.
Moreover, for any $z \in \mathbb{R}$ we have $\operatorname{ht}^{-1} (z) = P^\lambda \cap H_{2\rho, z}$.
Therefore, we also have 
\begin{equation*}
(\mathrm{ht}_*\mathrm{Vol}_r)(U) = \int_U \frac{1}{\lVert 2 \rho \rVert} \operatorname{Vol}_{r-1} (\operatorname{ht}^{-1} (z)) \, \mathrm{d}z.
\end{equation*}
Here the integral is the Lebesgue integral, $\lVert 2 \rho \rVert$ denotes the length of $2\rho$, and $\operatorname{Vol}_{r-1}$ is the Lebesgue measure on the hyperplanes $H_{2\rho, z}$ induced by the inner product $(-|-)$. 
The coefficient $1/\lVert 2\rho \rVert$ appears in the integral because, any line segment $\mathcal{L}$ perpendicular to $H_{2 \rho, z}$ is mapped by $\operatorname{ht}$ to an interval of length $\lVert 2\rho \rVert \cdot \lVert \mathcal{L} \rVert$.
In other words, the density function of the measure $\mathrm{ht}_*\mathrm{Vol}_r$ is
\[g(z) = \frac{1}{\lVert 2 \rho \rVert} \operatorname{Vol}_{r-1} (\operatorname{ht}^{-1} (z)).\]

For any real number $z \in \mathbb{R}$, let $\delta_z$ be the Dirac measure at $z$, that is, for any Borel set $U \subset \mathbb{R}$, $\delta_z (U) = 1$ if $z \in U$, and $\delta_z (U) = 0$ if $z \notin U$.
For each positive integer $k$, we define a measure $\mathfrak{m}_k$ on~$\mathbb{R}$ as follows,
\begin{equation*}
    \mathfrak{m}_k := \frac{1}{k^r}\sum_{i=0}^{\ell(t_{k \lambda})} \prescript{f}{}{b}_i^{t_{k\lambda}} \delta_{\frac{i}{k}}.
\end{equation*}
Note that $\ell(t_{k \lambda}) = k \ell(t_\lambda)$ by Corollary~\ref{cor-length-k-lambda}.

\subsection{Weak convergence of ``companion'' measures}

We  define two ``companions'' of $\mathfrak{m}_k$, denoted by $\mathfrak{m}_k^{\mathrm{lat}}$ and $\mathfrak{m}_{k,+}^{\mathrm{lat}}$ (``lat'' for ``lattice''),  respectively, by
\begin{equation*}
\mathfrak{m}_k^{\mathrm{lat}} := \frac{\lvert W_f \rvert}{k^r} \sum_{i=0}^{\ell(t_{k \lambda})} b_i^{k \lambda} \delta_{\frac{i}{k}}, \qquad 
\mathfrak{m}_{k,+}^{\mathrm{lat}} := \frac{\lvert W_f \rvert}{k^r} \sum_{i=0}^{\ell(t_{k \lambda})} b_{i,+}^{k \lambda} \delta_{\frac{i}{k}}.
\end{equation*}
Recall Lemma~\ref{lem-vol-D} that $\lvert W_f \rvert = \operatorname{Vol}_r (D) / \operatorname{Vol}_r(A_+)$ where $D$ is the open parallelotope spanned by the simple coroots.

\begin{remark} \label{rmk-measure-m-k-lat}
     In the definition of $\mathfrak{m}_k^\mathrm{lat}$, intuitively, we consider the $k$-fold dilation of the polytope $P^\lambda$, and count the $\mathbb{Z} \Phi\spcheck$-lattice points\footnote{This procedure is in a similar flavor as the theory of Ehrhart polynomial for lattice polytope~\cite{ehrhart1962polyedres}, although  $P^\lambda$ is not a lattice polytope in general.} in~$P^{k \lambda} \cap H_{2 \rho, i}$.
    Then we scale back and put the number $b_i^{k\lambda}$ obtained as the mass at the point $\frac{i}{k} \in \mathbb{R}$.
    This is equivalent to count the $\frac{1}{k} \mathbb{Z} \Phi\spcheck$-lattice points in the slice $P^\lambda \cap H_{2 \rho, \frac{i}{k}}$, that is, 
    \begin{equation} \label{eq-b-i-k-lambda}
        b_i^{k \lambda} = \operatorname{Card} \bigl( P^\lambda \cap (\frac{1}{k} \mathbb{Z} \Phi\spcheck) \cap H_{2 \rho, \frac{i}{k}} \bigr). 
    \end{equation} 
    In addition, as mentioned in Remark~\ref{rmk-b-i-lambda}, each element in~$P^\lambda \cap \bigl(\frac{1}{k} \mathbb{Z} \Phi\spcheck \bigr)$ contributes to the measure $\mathfrak{m}_k^\mathrm{lat}$ once and only once.
    Similar observations also apply to $\mathfrak{m}_{k,+}^\mathrm{lat}$.
    
    The scalar $\lvert W_f \rvert / {k^r}$ in the definition is for normalization, since there are $\lvert W_f \rvert$ many alcoves ``around'' a lattice point (see Remark~\ref{rmk-closure}).
\end{remark}

Clearly, all  the measures $\mathrm{ht}_*\mathrm{Vol}_r$, $\mathfrak{m}_k$, $\mathfrak{m}_k^\mathrm{lat}$ and $\mathfrak{m}_{k,+}^\mathrm{lat}$ are non-negative, bounded, and supported on~$[0,\ell(t_\lambda)]$. We have the following:
\begin{proposition} \label{prop-conv-measure-lat} \leavevmode
    \begin{enumerate}
        \item The sequence of measures $(\mathfrak{m}_k^\mathrm{lat})_k$  converges weakly to $\frac{1}{\operatorname{Vol}_r(A_+)}\mathrm{ht}_*\mathrm{Vol}_r$.
        \item The sequence of measures $(\mathfrak{m}_{k,+}^\mathrm{lat})_k$ converges  weakly to $\frac{1}{\operatorname{Vol}_r(A_+)}\mathrm{ht}_*\mathrm{Vol}_r$.
    \end{enumerate}
\end{proposition}

\begin{proof}
    To prove the first statement, by Lemma~\ref{lem-weak-convergence}, it suffices to show that for any $z \in [0,\ell(t_\lambda)]$, we have 
    \begin{equation} \label{eq-lem-conv-measure-lat-1}
        \lim_{k \to \infty} \mathfrak{m}_k^\mathrm{lat} ([0,z]) = \frac{1}{\operatorname{Vol}_r(A_+)} \mathrm{ht}_*\mathrm{Vol}_r  ([0,z]).
    \end{equation}
    We denote by $P^\lambda_{[0,z]} := \left\{x \in P^\lambda \xmiddle| 0 \le (2\rho| x) \le z \right\}$ the truncation of $P^\lambda$.
    Then 
    \begin{equation} \label{eq-lem-conv-measure-lat-2}
      \left(\mathrm{ht}_*\mathrm{Vol}_r \right) ([0,z]) = \operatorname{Vol}_r \bigl(P^\lambda_{[0,z]}\bigr).
    \end{equation}
    On the other hand, we have (see Remark~\ref{rmk-measure-m-k-lat})
    \begin{equation*}
        \sum_{i=0}^{k \ell(t_\lambda)} b_i^{k \lambda} \delta_{\frac{i}{k}} ([0,z]) = \sum_{i \in \mathbb{N}, 0 \le i \le kz} b_i^{k \lambda} =\operatorname{Card} \left( P^\lambda_{[0,z]} \cap \Bigl( \frac{1}{k} \mathbb{Z} \Phi\spcheck \Bigr)  \right),
    \end{equation*}
    namely, the number of lattice points of $\frac{1}{k} \mathbb{Z} \Phi\spcheck$ in the truncation $P^\lambda_{[0,z]}$.
    Therefore,
    \begin{align*}
       \mathfrak{m}_k^\mathrm{lat} ([0,z]) & = \frac{\lvert W_f \rvert}{k^r} \operatorname{Card} \left( P^\lambda_{[0,z]} \cap \Bigl( \frac{1}{k} \mathbb{Z} \Phi\spcheck \Bigr)  \right) \\ 
       & = \frac{1}{\operatorname{Vol}_r(A_+)} \operatorname{Vol}_r \left(\frac{D}{k}\right) \operatorname{Card} \left( P^\lambda_{[0,z]} \cap \Bigl( \frac{1}{k} \mathbb{Z} \Phi\spcheck \Bigr)  \right)
    \end{align*}
    is a Riemann sum, and the limit $\lim_k \mathfrak{m}_k^\mathrm{lat} ([0,z])$ is the Riemann integral 
    \[\int_{P^\lambda_{[0,z]}} \frac{1}{\operatorname{Vol}_r(A_+)} = \frac{1}{\operatorname{Vol}_r(A_+)} \operatorname{Vol}_r \bigl(P^\lambda_{[0,z]}\bigr).\]
    By Equation  \eqref{eq-lem-conv-measure-lat-2}, we see that Equation  \eqref{eq-lem-conv-measure-lat-1} is valid.

    The proof of the second statement is similar.
    We have the Riemann sum
    \begin{equation*}
       \mathfrak{m}_{k,+}^\mathrm{lat} ([0,z]) = \frac{\lvert W_f \rvert}{k^r} \operatorname{Card} \left( P^\lambda_{+, [0,z]} \cap \Bigl( \frac{1}{k} \mathbb{Z} \Phi\spcheck \Bigr)  \right),
    \end{equation*}
    where 
    \begin{equation*}
        P^\lambda_{+, [0,z]} := \left\{x \in P_+^\lambda \xmiddle| 0 \le (2\rho| x) \le z \right\}.
    \end{equation*}
    But $\operatorname{Vol}_r (P^\lambda_{[0,z]}) = \operatorname{Vol}_r (P^\lambda_{+,[0,z]})$.
    So the analog of Equation  \eqref{eq-lem-conv-measure-lat-1} also holds for $\mathfrak{m}_{k,+}^\mathrm{lat}$.
\end{proof}

\begin{remark}
    The sequences $(\mathfrak{m}_{k}^\mathrm{lat})_k$ and $(\mathfrak{m}_{k,+}^\mathrm{lat})_k$ do \emph{not} converge strongly (see Remark~\ref{rmk-strong-convergence}) to the limit measure. 
    For example, $Q = \mathbb{Q} \cap [0, \ell(t_\lambda)]$ is a Borel set and its inverse image $\operatorname{ht}^{-1} (Q)$ consists of a countable union of $(r-1)$-dimensional slices, so $(\mathrm{ht}_*\mathrm{Vol}_r) (Q) = 0$.
    However, $\lim_k \mathfrak{m}_k^\mathrm{lat} (Q) = \operatorname{Vol}_r(P^\lambda) / \operatorname{Vol}_r(A_+) > 0$, this is because any point mass of any $\mathfrak{m}_k^\mathrm{lat}$ is supported on~$Q$.
\end{remark}

\subsection{Proof of Theorem~\ref{thm-main-in-the-introduction}\ref{thm-main-1-in-the-introduction}}

Recall the theorem states that the sequence of measures $(\mathfrak{m}_k)_k$ converges weakly to the measure $\frac{1}{\operatorname{Vol}_r (A_+)}\mathrm{ht}_*\mathrm{Vol}_r$.

\begin{proof}
    As in the proof of Proposition~\ref{prop-conv-measure-lat}, it suffices to show that  
    \begin{equation} \label{eq-thm-main-1}
        \text{for any $z \in [0, \ell(t_\lambda)]$, we have }
        \lim_{k \to \infty} \mathfrak{m}_k ([0,z]) = \frac{1}{\operatorname{Vol}_r(A_+)} \mathrm{ht}_*\mathrm{Vol}_r ([0,z]). 
    \end{equation}
    For any $k = 1,2, \dots $, we have 
    \begin{equation} \label{eq-thm-main-2}
      \begin{aligned}
          \mathfrak{m}_k ([0,z]) & = \frac{1}{k^r}\sum_{i=0}^{ \ell(t_{k\lambda})}\prescript{f}{}{b}_i^{t_{k\lambda}} \delta_{\frac{i}{k}} ([0,z]) \\
          & = \frac{1}{k^r} \sum_{i \in \mathbb{N}, 0 \le i\le kz} \prescript{f}{}{b}_i^{t_{k\lambda}} \\
          & = \frac{1}{k^r} T^{\le kz} \pi^{t_{k \lambda}} (1),
      \end{aligned}
    \end{equation}
    where $T^{\le kz}$ is the truncation and $\pi^{t_{k \lambda}}$ is the Poincar\'e polynomial from Section~\ref{sec-dom-lattice-formula}.
    
    By Corollary~\ref{cor-lattice-formula} and Lemma~\ref{lem-laurent-truncation}\ref{lem-laurent-truncation-1}, we have
    \begin{equation} \label{eq-thm-main-3}
        T^{\le kz} \left( \pi^{k\lambda}_+ (q) \cdot \pi_f(q^{-1}) \right) \le T^{\le kz} \pi^{t_{k \lambda}} (q) \le T^{\le kz} \left( \pi^{k\lambda} (q) \cdot \pi_f(q^{-1}) \right).
    \end{equation}
    Let $l := \lvert \Phi^+ \rvert$ be the length of the longest element in~$W_f$.
    Then the exponents of $q$-powers in the Laurent polynomial $\pi_f(q^{-1})$ concentrate on the interval $[-l, 0]$.
    Therefore, Lemma~\ref{lem-laurent-truncation-3} applies to the first and third terms of \eqref{eq-thm-main-3}.
    We have
    \begin{equation} \label{eq-thm-main-4}
       \left( T^{\le kz}  \pi^{k\lambda}_+ (q) \right) \cdot \pi_f(q^{-1})   \le T^{\le kz} \left( \pi^{k\lambda}_+ (q) \cdot \pi_f(q^{-1}) \right) \le T^{\le kz} \pi^{t_{k \lambda}} (q), 
    \end{equation}
    where the first inequality is obtained by applying Lemma~\ref{lem-laurent-truncation-3} to the first term of \eqref{eq-thm-main-3}, and the second one is exactly the first inequality of \eqref{eq-thm-main-3}.
    Evaluating the first and the third term of \eqref{eq-thm-main-4} at $q=1$, we obtain
    \begin{equation} \label{eq-thm-main1-1}
      \lvert W_f \rvert \cdot T^{\le kz} \pi^{k\lambda}_+ (1) \le  T^{\le kz} \pi^{t_{k \lambda}} (1).      
    \end{equation}
    By Equation  \eqref{eq-thm-main-2}, the right hand side of \eqref{eq-thm-main1-1} equals $k^r \mathfrak{m}_k([0,z])$.
    Therefore, by the definitions of $\mathfrak{m}_{k,+}^\mathrm{lat}$ and $\pi_+^{k\lambda} (q)$, we have
    \begin{equation} \label{eq-thm-main-5}
      \mathfrak{m}_{k,+}^\mathrm{lat} ([0,z]) = \frac{\lvert W_f \rvert}{k^r} \sum_{i \in \mathbb{N}, 0 \le i\le kz} b_{i,+}^{k \lambda} = \frac{\lvert W_f \rvert}{k^r} \cdot T^{\le kz} \pi^{k\lambda}_+ (1) \le \mathfrak{m}_k([0,z]).
    \end{equation}
    
    On the other hand, applying Lemma~\ref{lem-laurent-truncation-3} to the third term of \eqref{eq-thm-main-3}, we also have 
    \begin{equation}\label{eq-thm-main-6}
      T^{\le kz} \pi^{t_{k \lambda}} (q) \le T^{\le kz} \left( \pi^{k\lambda} (q) \cdot \pi_f(q^{-1}) \right) \le \left( T^{\le kz + l} \pi^{k\lambda} (q) \right) \cdot \pi_f(q^{-1}),
    \end{equation}
    where the first inequality comes from \eqref{eq-thm-main-3}.
    Evaluating the first and the third term of \eqref{eq-thm-main-6} at $q=1$ yields
    \begin{equation*}
      T^{\le kz} \pi^{t_{k \lambda}} (1) \le \lvert W_f \rvert \cdot  T^{\le kz + l} \pi^{k\lambda} (1).
    \end{equation*}
    Therefore, by Equation  \eqref{eq-thm-main-2} and the definitions of $\pi^{k\lambda} (q)$ and $\mathfrak{m}_k^\mathrm{lat}$, we have
    \begin{equation}\label{eq-thm-main-7}
      \begin{aligned}
          \mathfrak{m}_k([0,z]) & \le \frac{\lvert W_f \rvert}{k^r} \cdot T^{\le kz + l} \pi^{k\lambda} (1) \\ 
          & = \frac{\lvert W_f \rvert}{k^r} \sum_{0 \le i \le kz+l} b_{i}^{k \lambda} \\
          & = \mathfrak{m}_{k}^\mathrm{lat} \Bigl([0,z+\frac{l}{k}]\Bigr) \\
          & = \mathfrak{m}_{k}^\mathrm{lat} ([0,z]) + \mathfrak{m}_{k}^\mathrm{lat} \Bigl((z,z+\frac{l}{k}]\Bigr).
      \end{aligned}
    \end{equation}
    
    The inequalities \eqref{eq-thm-main-5} and  \eqref{eq-thm-main-7} tell us
    \begin{equation}\label{eq-thm-main-8}
        \mathfrak{m}_{k,+}^\mathrm{lat} ([0,z]) \le \mathfrak{m}_k([0,z]) \le \mathfrak{m}_{k}^\mathrm{lat} ([0,z]) + \mathfrak{m}_{k}^\mathrm{lat} \Bigl((z,z+\frac{l}{k}]\Bigr).
    \end{equation}
    By the proof of Proposition~\ref{prop-conv-measure-lat}, we have 
    \begin{equation} \label{eq-thm-main-9}
        \lim_{k \to \infty} \mathfrak{m}_{k,+}^\mathrm{lat} ([0,z]) = \lim_{k \to \infty} \mathfrak{m}_{k}^\mathrm{lat} ([0,z]) = \frac{1}{\operatorname{Vol}_r(A_+)} \mathrm{ht}_*\mathrm{Vol}_r ([0,z]).
    \end{equation}
    So it remains to show that the ``error term'' $\mathfrak{m}_{k}^\mathrm{lat} \left((z,z+\frac{l}{k}]\right)$ tends to zero.
    For any $\delta \in \mathbb{R}_{>0}$, we have $\frac{l}{k} < \delta$ for $k$ large enough.
    Thus, 
    \begin{equation*}
        \lim_{k \to \infty} \mathfrak{m}_{k}^\mathrm{lat} \Bigl((z,z+\frac{l}{k}]\Bigr) \le \lim_{k \to \infty} \mathfrak{m}_{k}^\mathrm{lat} \bigl((z, z+\delta]\bigr) = \frac{1}{\operatorname{Vol}_r(A_+)}\operatorname{Vol}_r (P^\lambda_{[z,z+\delta]}).
    \end{equation*}
    But $\delta \in \mathbb{R}_{> 0}$ can be taken to be arbitrarily small, and $P^\lambda$ is a bounded polytope.
    Therefore $\operatorname{Vol}_r (P^\lambda_{[z,z+\delta]})$ can be  arbitrarily small.
    This proves 
    \begin{equation} \label{eq-thm-main-10}
        \lim_{k \to \infty} \mathfrak{m}_{k}^\mathrm{lat} \Bigl((z,z+\frac{l}{k}]\Bigr) = 0.
    \end{equation}
    By \eqref{eq-thm-main-8}, \eqref{eq-thm-main-9}, and \eqref{eq-thm-main-10}, we obtain \eqref{eq-thm-main-1}, as desired.
\end{proof}

\section{The uniform convergence of \texorpdfstring{$(S_1, S_2,\ldots)$}{S1, S2, ...}} \label{sec-main-unifm-convg}
Let $\lambda \in \mathbb{Z} \Phi\spcheck \cap \overline{C_+}$ be a fixed dominant coroot lattice point as before. 

For any positive integer $k$, we define the step function $S_k \colon [0, \ell(t_\lambda)] \to \mathbb{R}$ as follows.
For any $z \in [0, \ell(t_\lambda)]$, there exists a unique $i \in \{0, 1, \dots, k\ell(t_\lambda)\}$ such that $z \in [\frac{i}{k}, \frac{i+1}{k})$.
We define 
\[S_k(z) := \frac{1}{k^{r-1}} \prescript{f}{}{b}_{i}^{t_{k\lambda}}.\]

\begin{remark} \label{rem-density}
The function $S_k(z)$ can be interpreted as an approximation of the density function (which in this case does not exist) of the discrete measure $\mathfrak{m}_k$ in the following sense:
    for any $z \in [0, \ell(t_\lambda)]$, there exists a unique $i \in \{0, 1, \dots, k\ell(t_\lambda)\}$ such that $z \in [\frac{i}{k}, \frac{i+1}{k})$, and we have
    \begin{equation} \label{eq-density}
        \mathfrak{m}_k([0, z]) = \frac{1}{k^r} \sum_{0 \le j\le i} \prescript{f}{}{b}_j^{t_{k\lambda}} = \int_0^{\frac{i+1}{k}} S_k(x) \, \mathrm{d}x.
    \end{equation}
\end{remark}

Let us recall the statement of Theorem~\ref{thm-main-in-the-introduction}\ref{thm-main-2-in-the-introduction}:
    The sequence of step functions $(S_k(z))_k$ converges uniformly to the function
    \[z \mapsto \frac{1}{\operatorname{Vol}_r(A_+) \cdot \lVert 2 \rho \rVert} \operatorname{Vol}_{r-1} (\operatorname{ht}^{-1} (z)).\]  

Theorem~\ref{thm-main-in-the-introduction}\ref{thm-main-2-in-the-introduction} has the following corollary:
\begin{corollary} \label{cor-layer} For every $0\leq i\leq \ell(t_\lambda)$, we have
\begin{equation*}
    \lim_{k \to \infty} \frac{1}{k^{r-1}} \prescript{f}{}{b}_{ki}^{t_{k\lambda}} = \frac{1}{\operatorname{Vol}_r(A_+) \cdot \lVert 2 \rho \rVert} \operatorname{Vol}_{r-1} (\operatorname{ht}^{-1} (i)).
\end{equation*}
\end{corollary}

\begin{proof}
    This is because $\lim_k S_k(i) = g(i) / \operatorname{Vol}_{r} (A_+)$ by Theorem~\ref{thm-main-in-the-introduction}\ref{thm-main-2-in-the-introduction}.
\end{proof}

The proof of Theorem~\ref{thm-main-in-the-introduction}\ref{thm-main-2-in-the-introduction}, which occupies Sections~\ref{subsec-Eud-geom} to~\ref{subsec-pf-eq}, is more technical than that of Theorem~\ref{thm-main-in-the-introduction}\ref{thm-main-1-in-the-introduction}. 
Therefore, before the full proof, we outline the main ideas.

\subsection{Outline of the proof of Theorem~\ref{thm-main-in-the-introduction}\ref{thm-main-2-in-the-introduction}} \label{subsec-outline}

To prove the uniform convergence, 
we need to estimate the value $\prescript{f}{}{b}_{i}^{t_{k \lambda}}/ k^{r-1}$ of the step function $S_k$ at $z \in [\frac{i}{k}, \frac{i+1}{k})$.
Using Corollary~\ref{cor-lattice-formula}, we switch this estimation to the estimation of the numbers
\begin{equation} \label{eq-card-lat}
    b^{k \lambda}_{i+j} = \operatorname{Card} \left\{P^\lambda \cap \frac{1}{k} \mathbb{Z} \Phi\spcheck \cap H_{2\rho, \frac{i+j}{k}} \right\}
\end{equation}
and their companions $b^{k \lambda}_{i+j, +}$, where $0 \le j \le \lvert \Phi^+ \rvert$ (see Inequality \eqref{eq-main-1}).
Note that $b^{k \lambda}_{i+j} \ne 0$ only if $i+j$ is even, and in this case $(\frac{1}{k} \mathbb{Z} \Phi\spcheck) \cap H_{2\rho, \frac{i+j}{k}}$ is a lattice of rank $(r-1)$ in the affine hyperplane $H_{2\rho, \frac{i+j}{k}}$.

We choose a fundamental domain~$B_k$ of the lattice $(\frac{1}{k} \mathbb{Z}\Phi\spcheck) \cap H_{2 \rho, 0}$ containing the origin.
If we join all the translations of $B_k$ by points in~$P^\lambda \cap (\frac{1}{k} \mathbb{Z}\Phi\spcheck) \cap H_{2\rho, \frac{i+j}{k}}$, we obtain the region
\begin{equation*} 
    \mathcal{R} = \mathcal{R}_{k,i,j} := \bigsqcup \left\{ p + B_k \; \middle| \; p\in P^\lambda \cap \frac{1}{k} \mathbb{Z}\Phi\spcheck \cap H_{2\rho, \frac{i+j}{k}} \right\} \subset H_{2\rho, \frac{i+j}{k}}.
\end{equation*}
Since we can compute the volume of $B_k$ directly from $\Phi$, estimating the value of \eqref{eq-card-lat} is equivalent to estimating $\operatorname{Vol}_{r-1}(\mathcal{R})$.
The proof of the convergence is then achieved by comparing $\operatorname{Vol}_{r-1}(\mathcal{R})$ with $\operatorname{Vol}_{r-1}(P^\lambda\cap H_{2\rho, z})$.
This, as well as the uniformity, requires carefully estimating the volume of some open neighborhood of the boundary of  $P^\lambda\cap H_{2\rho, z}$ (see Figures~\ref{fig:P-pmdelta} and~\ref{fig:P-pmdelta-N} for an example of such a neighborhood).
When $k$ is large enough, for any $z$, the boundary of $\mathcal{R}$ is ``contained'' in such a neighborhood. 
Because the volume of such a neighborhood can be sufficiently small, $\operatorname{Vol}_{r-1}(\mathcal{R})$ can be sufficiently close to $\operatorname{Vol}_{r-1}(P^\lambda\cap H_{2\rho, z})$. This leads to the proof of the uniform convergence.

\subsection{Volume estimations} \label{subsec-Eud-geom}


Only in this subsection, let $E = \mathbb{R}^n$ be an arbitrary Euclidean space of dimension $n$ with inner product $(-|-)\colon E \times E \to \mathbb{R}$.
In the proof of Theorem~\ref{thm-main-in-the-introduction}\ref{thm-main-2-in-the-introduction}, we will use the results in this subsection mainly in the case $n = r-1$.

For two points $x, y \in E$, we denote by $\operatorname{d} (x, y)$ the \emph{distance} between $x$ and $y$, that is, $\operatorname{d} (x,y) = \sqrt{(x-y|x-y)}$.
For a nonempty subset $Y \subset E$ the \emph{distance} between $x$ and $Y$ is defined by
\[\operatorname{d}(x,Y) := \inf_{y \in Y} \operatorname{d}(x,y)\]
and the \emph{diameter} of $Y$ is defined by
\[\operatorname{diam}(Y) := \sup_{y,y' \in Y} \operatorname{d} (y,y').\]
For an $n$-dimensional polytope $P$ with $n > 0$, its \emph{boundary} $\partial P$ is the union of all its faces of dimension strictly less than $n$. Both $P$ and $\partial P$ are bounded closed sets. 
For any real number $\delta > 0$, we define
\begin{equation*}
    P^{+\delta} := \{x \in E \mid \operatorname{d}(x,P) < \delta\}, \qquad
    P^{-\delta} := \{x \in P \mid \operatorname{d}(x, \partial P) > \delta\}.
\end{equation*}
For example, in Figure~\ref{fig:P-pmdelta}, $P$ is the solid triangle in a Euclidean plane, $P^{+\delta}$ is the interior of the dotted area, and $P^{-\delta}$ is the interior of the dashed triangle.
\begin{figure} [ht]
    \centering
    \begin{tikzpicture}
        \coordinate (A) at (0,0);
        \coordinate (B) at (4,0);
        \coordinate (C) at (60:4);

        \coordinate (Ain) at (30:1);
        \coordinate (Bin) at ($(B) + (150:1)$);
        \coordinate (Cin) at ($(C) + (270:1)$);

        \coordinate (A1) at ($(A) + (270:0.5)$);
        \coordinate (A2) at ($(A) + (150:0.5)$);
        \coordinate (B1) at ($(B) + (30:0.5)$);
        \coordinate (B2) at ($(B) + (270:0.5)$);
        \coordinate (C1) at ($(C) + (150:0.5)$);
        \coordinate (C2) at ($(C) + (30:0.5)$);

        \draw [thick] (A) -- (B) -- (C) -- (A);
        \draw [dashed, thick] (Ain) -- (Bin) -- (Cin) -- (Ain);
        \draw [dotted, thick] (A2) -- (C1); 
        \draw [dotted, thick] (C2) -- (B1); 
        \draw [dotted, thick] (B2) -- (A1);
        \draw [dotted, thick] (A2) arc (150:270:0.5);
        \draw [dotted, thick] (B2) arc (-90:30:0.5);
        \draw [dotted, thick] (C2) arc (30:150:0.5);
        \draw [->, thick] (2, 0) -- (2, 0.5);
        \draw [->, thick] (2, 0.5) -- (2, 0);
        \draw [->, thick] (2, 0) -- (2, -0.5);
        \draw [->, thick] (2, -0.5) -- (2, 0);

        \coordinate [label=left:$\delta$] (d1) at (2,0.25);
        \coordinate [label=left:$\delta$] (d2) at (2,-0.25);
    \end{tikzpicture}
    \caption{A triangle $P$ and the $P^{+\delta}$, $P^{-\delta}$.}
    \label{fig:P-pmdelta}
\end{figure}
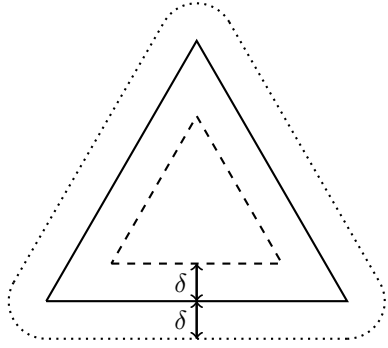

A moment's thought shows that the difference set $P^{+\delta} \setminus \overline{P^{-\delta}}$, where $\overline{P^{-\delta}}$ is the closure of $P^{-\delta}$, is the open neighbourhood
\[\mathcal{N}(\partial P, \delta) := \{x \in E \mid \operatorname{d}(x, \partial P) < \delta\}\]
of $\partial P$.
We want to estimate its volume, $\operatorname{Vol}_n (\mathcal{N}(\partial P, \delta))$.
Before this, we need to introduce some notations.

Suppose $F^\circ$ is an open face of $P$ and $\dim F^\circ \le n - 1$.
Then $F^\circ \subseteq \partial P$.
Let $x \in F^\circ$ be an arbitrary point, and $H^\perp_{x}$ be the affine subspace transversal to $\langle F^\circ \rangle_\text{aff}$ at $x$, that is, the ($n - \dim F^\circ$)-dimensional affine subspace passing through $x$ and perpendicular to $\langle F^\circ \rangle_\text{aff}$.
We denote by
\[B_{x, \delta}^\perp := \{y \in H^\perp_x \mid \operatorname{d}(x,y) < \delta\}\]
the open ball in~$H^\perp_x$ with center $x$ and radius $\delta$, and by 
\[C_{F^\circ, \delta} := \bigsqcup_{x \in F^\circ} B_{x, \delta}^\perp\]
the disjoint union of the balls.
Geometrically, $C_{F^\circ, \delta}$ is isometric to the product of $F^\circ$ and an ($n - \dim F^\circ$)-ball of radius $\delta$, which looks like a ``cylinder''.
For instance, in the case $n = 3$, Figure~\ref{fig:cylinder} illustrates examples of $C_{F^\circ, \delta}$, where $F^\circ$ is an open triangle, an open segment, and a point, respectively. 

\begin{figure} [ht]
    \centering
    \begin{tikzpicture}
        \fill [fill=lightgray] (0,0) -- (3,0) -- (2,0.8) -- (0,0);

        \draw [dotted, thick] (2,0.3) -- (2,1.3);
        \draw [dotted, thick] (3, -0.5) -- (2,0.3) -- (0,-0.5);
        \draw [dashed, thick] (0, 0.5) -- (0,-0.5) -- (3,-0.5) -- (3,0.5) -- (0,0.5) -- (2,1.3) -- (3,0.5);
        \draw [thick] (0,0) -- (3,0) --  (2,0.8) -- (0,0);
        
        \draw [thick, ->] (-0.2, 0) -- (-0.2,0.5);
        \draw [thick, ->] (-0.2, 0) -- (-0.2,-0.5);
        \draw [thick, ->] (-0.2, 0.5) -- (-0.2,0);
        \draw [thick, ->] (-0.2, -0.5) -- (-0.2,0);

        \coordinate [label=left:$\delta$] (d1) at (-0.2,0.25);
        \coordinate [label=left:$\delta$] (d2) at (-0.2,-0.25);

        \draw [thick] (5,0) -- (7,0);
        \draw [dashed, thick] (5,0) ellipse (0.2 and 0.5);
        \draw [dashed, thick] (5, 0.5) -- (7,0.5);
        \draw [dashed, thick] (5, -0.5) -- (7,-0.5);
        \draw [dashed, thick] (7,-0.5) arc (-90:90:0.2 and 0.5);
        \draw [dotted, thick] (7,0.5) arc (90:270:0.2 and 0.5);

        \draw [->, thick] (6,0) -- (6, 0.5);
        \draw [->, thick] (6,0.5) -- (6, 0);
        \coordinate [label=left:$\delta$] (d3) at (6,0.25);

        \fill (9,0) circle (1pt);
        \draw [thick, dashed] (9,0) circle (0.5);
        \draw [thick, dashed] (8.5,0) arc (180:360:0.5 and 0.2);
        \draw [thick, dotted] (9.5,0) arc (0:180:0.5 and 0.2);

        \draw [->, thick] (9,0) -- (9, 0.5);
        \draw [->, thick] (9,0.5) -- (9, 0);
        \coordinate [label=left:$\delta$] (d4) at (9,0.66);
    \end{tikzpicture}
    \caption{Examples of $C_{F^\circ, \delta}$ in a $3$-dimensional Euclidean space.}
    \label{fig:cylinder}
\end{figure}
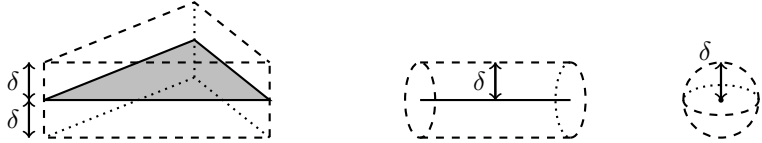

We can compute the volume of $C_{F^\circ, \delta}$.

\begin{lemma} \label{lem-vol-cylinder}
    Let $B_\delta^{k}$ be a $k$-dimensional ball of radius $\delta$.
    Then for an $i$-dimensional open face $F^\circ$ of an $n$-dimensional polytope $P$, we have
    \begin{equation*}
        \operatorname{Vol}_n (C_{F^\circ, \delta})  = \operatorname{Vol}_{i} (F^\circ) \times \operatorname{Vol}_{n - i} (B_\delta^{n - i}) 
         = \operatorname{Vol}_{i} (F^\circ) \times c_{n - i} \delta^{n - i},
    \end{equation*}
    where $c_{n - i} = \uppi^{\frac{n - i}{2}} / \Gamma( \frac{n - i}{2} + 1)$ and $\Gamma$ is the Euler's Gamma function.
\end{lemma}

\begin{proof}
    Clear since  $\operatorname{Vol}_k (B_\delta^{k}) = c_k \delta^k$.
\end{proof}

Note that $c_{n-i}$ is a constant that only depends on~$\dim F^\circ = i$, not on the shape or the volume of $F^\circ$.

\begin{lemma} \label{lem-N-union-cylinder}
    Let $\mathcal{F}_i$ be the set of $i$-dimensional open faces of an $n$-dimensional polytope $P$ with $n > 0$. 
    Then, for any $\delta > 0$,
    \begin{equation*}
        \mathcal{N}(\partial P, \delta) = \bigcup_{0 \le i \le n - 1} \bigcup_{F^\circ \in \mathcal{F}_i} C_{F^\circ, \delta}.
    \end{equation*}
\end{lemma}

\begin{proof}
    Suppose $y \in C_{F^\circ, \delta}$ for some open face $F^\circ$ of dimension strictly less than $n$.
    Then $y \in B^\perp_{x, \delta}$ for some $x \in F^\circ$.
    Thus $\operatorname{d} (y, \partial P) < \delta$ and $y \in \mathcal{N}(\partial P, \delta)$.
    This proves ``$\supseteq$''.

    Conversely, suppose $y \in E$ satisfies $\operatorname{d} (y, \partial P) < \delta$. We can assume $y \notin \partial P$.
    Since $\partial P$ is a bounded closed subset and hence a compact subset in~$E$, there exists $x \in \partial P$ (hence $x \ne y$) such that $\operatorname{d} (y, x) = \operatorname{d} (y, \partial P)$.
    Suppose $F^\circ$ is the open face containing $x$.
    We claim that the affine line $\langle \{x,y\} \rangle_\text{aff}$ is perpendicular to the affine subspace $\langle F^\circ \rangle_{\text{aff}}$.
    Otherwise, since $F^\circ$ is open in~$\langle F^\circ \rangle_\text{aff}$, we can find some $x'\in F^\circ$ such that $\operatorname{d}(y, x') < \operatorname{d} (y,x)$, which contradicts $\operatorname{d} (y, x) = \operatorname{d} (y, \partial P)$.
    In particular, our claim implies $y \in B^\perp_{x, \delta} \subseteq C_{F^\circ, \delta}$.
    This proves ``$\subseteq$''.
\end{proof}

For the example in Figure~\ref{fig:P-pmdelta}, the set $\mathcal{N}(\partial P, \delta)$ for a small $\delta$ is a union of three disks and three rectangles, see the gray area in Figure~\ref{fig:P-pmdelta-N}.

\begin{figure} [ht]
    \centering
    \begin{tikzpicture}
        \coordinate (A) at (0,0);
        \coordinate (B) at (4,0);
        \coordinate (C) at (60:4);

        \coordinate (Ain) at (30:1);
        \coordinate (Bin) at ($(B) + (150:1)$);
        \coordinate (Cin) at ($(C) + (270:1)$);

        \coordinate (A1) at ($(A) + (270:0.5)$);
        \coordinate (A2) at ($(A) + (150:0.5)$);
        \coordinate (B1) at ($(B) + (30:0.5)$);
        \coordinate (B2) at ($(B) + (270:0.5)$);
        \coordinate (C1) at ($(C) + (150:0.5)$);
        \coordinate (C2) at ($(C) + (30:0.5)$);

        \fill [fill=lightgray] (A2) -- (C1) -- ($(C1) + (330:1)$) -- ($(A2) + (330:1)$) -- (A2); 
        \fill [fill=lightgray] (C2) -- (B1) -- ($(B1) + (210:1)$) -- ($(C2) + (210:1)$) -- (C2); 
        \fill [fill=lightgray] (B2) -- (A1) -- ($(A1) + (90:1)$) -- ($(B2) + (90:1)$) -- (B2);
        \fill [fill=lightgray] (A2) arc (150:510:0.5);
        \fill [fill=lightgray] (B2) arc (-90:270:0.5);
        \fill [fill=lightgray] (C2) arc (30:390:0.5);

        \draw [dotted, thick] (A2) -- (C1) -- ($(C1) + (330:1)$) -- ($(A2) + (330:1)$) -- (A2); 
        \draw [dotted, thick] (C2) -- (B1) -- ($(B1) + (210:1)$) -- ($(C2) + (210:1)$) -- (C2); 
        \draw [dotted, thick] (B2) -- (A1) -- ($(A1) + (90:1)$) -- ($(B2) + (90:1)$) -- (B2);
        \draw [dotted, thick] (A2) arc (150:510:0.5);
        \draw [dotted, thick] (B2) arc (-90:270:0.5);
        \draw [dotted, thick] (C2) arc (30:390:0.5);
        
        \draw [thick] (A) -- (B) -- (C) -- (A);
    \end{tikzpicture}
    \caption{The set $\mathcal{N}(\partial P, \delta)$ equals the union of the dotted areas.}
    \label{fig:P-pmdelta-N}
\end{figure}
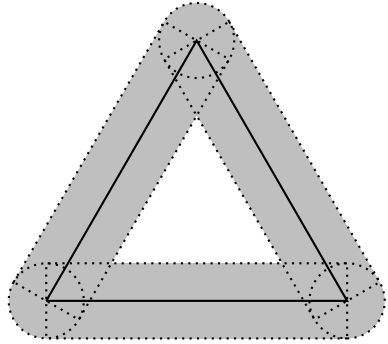

As a consequence, we have the following estimate of $\operatorname{Vol}_n (\mathcal{N}(\partial P, \delta))$.

\begin{proposition} \label{prop-neighbour-volume}
    Let $c_{k}$ and $\mathcal{F}_i$ be as in Lemmas~\ref{lem-vol-cylinder} and~\ref{lem-N-union-cylinder}.
    Then 
    \[\operatorname{Vol}_n (\mathcal{N}(\partial P, \delta)) \le \sum_{0 \le i \le n - 1} \Bigl(\sum_{F^\circ \in \mathcal{F}_i} \operatorname{Vol}_i (F^\circ) \Bigr)  c_{n - i} \delta^{n-i}.\]
\end{proposition}

\begin{proof}
     Follows immediately from Lemmas~\ref{lem-vol-cylinder} and~\ref{lem-N-union-cylinder}.
\end{proof}

Now, suppose $\{e_1, \dots, e_n\}$ is a basis of $E$.
Then, for any $k \in \mathbb{Z}_{> 0}$, 
\[L_k := \left\{\frac{m_1}{k} e_1 + \dots + \frac{m_n}{k} e_n \xmiddle| m_1, \dots, m_n \in \mathbb{Z}\right\}\] 
is a lattice in~$E$.
Only in this subsection, we define 
\[B := \left\{c_1 e_1 + \dots + c_n e_n \in E \xmiddle| c_1, \dots, c_n \in [0,1)\right\}\]
and 
\[B_k := \frac{1}{k} B = \left\{\frac{c_1}{k} e_1 + \dots + \frac{c_n}{k} e_n \xmiddle| c_1, \dots, c_n \in [0,1)\right\}\]
for any $k \in \mathbb{Z}_{> 0}$.
In subsequent subsections, we will use the notations $B$ and $B_k$ for the case $n = r-1$ and $e_i = \alpha_{i+1}\spcheck - \alpha_1 \spcheck$.
Note that for any $k$, $B_k$ is a fundamental domain of the translation action of $L_k$ on~$E$, and we have
\[\operatorname{diam}(B_k) = \frac{1}{k} \operatorname{diam}(B).\]
We have the following lemma.

\begin{lemma} \label{lem-union-box-nbhd}
    Let $\{e_1, \dots, e_n\}$ be a basis of $E$, and notations $L_k$ and $B_k$ be as above.
    Then, for any positive $\delta > 0$, 
    any integer $k > \operatorname{diam}(B) / \delta$, 
    any $n$-dimensional polytope $P$ in~$E$, and any set of lattice points $Z$ satisfying $(P \setminus \partial P) \cap L_k \subseteq Z \subseteq P \cap L_k$, we have
    \[P^{-\delta} \subseteq \bigsqcup_{x \in Z} (x + B_k) \subseteq P^{+\delta}.\]
\end{lemma}

\begin{proof}
    Let integer $k > \operatorname{diam}(B) / \delta$ be arbitrary.
    Then $\operatorname{diam} (B_k) < \delta$.
    Suppose $x \in P \cap L_k$.
    Then for any $y \in x + B_k$, we have
    \[\operatorname{d} (y, P) \le \operatorname{d} (y, x) \le \operatorname{diam} (B_k) < \delta.\]
    Therefore, $y \in P^{+\delta}$.
    Since $x \in P \cap L_k$ and $y \in x + B_k$ are arbitrary, we have
    \[\bigsqcup_{x \in P \cap L_k} (x+B_k) \subseteq P^{+\delta}.\]
    
    On the other hand, suppose $y \in P^{-\delta}$ (if $P^{-\delta} \ne \emptyset$).
    Then there exists a unique $x \in L_k$ such that $y \in x + B_k$, since $\{x + B_k \mid x \in L_k\}$ forms a partition of the ambient space $E$.
    Then 
    \begin{equation} \label{eq-lattice-neighbourhood}
        \operatorname{d} (y, x) \le \operatorname{diam} (B_k) < \delta.
    \end{equation}
    Next, we prove $x \in P \setminus \partial P$ by contradiction.
    Suppose $x \notin P$.
    Then 
    \[\operatorname{d} (y,x) > \operatorname{d} (y, \partial P) > \delta,\]
    since $y \in P^{-\delta} \subseteq P \setminus \partial P$ (the interior of $P$).
    But this contradicts \eqref{eq-lattice-neighbourhood}.
    Thus, $x \in P$.
    If $x \in \partial P$, then we have 
    \[\operatorname{d} (y, x) \ge \operatorname{d} (y, \partial P) > \delta,\]
    which is also a contradiction.
    Therefore, $x \in P \setminus \partial P$.
    Since $y \in P^{-\delta}$ is taken arbitrarily, we have 
    \[P^{-\delta} \subseteq \bigsqcup_{x \in (P \setminus \partial P) \cap L_k} (x + B_k).\]
    
    At last, we have the chain of containment relations
    \[P^{-\delta} \subseteq \bigsqcup_{x \in (P \setminus \partial P) \cap L_k} (x + B_k) 
    \subseteq \bigsqcup_{x \in Z} (x + B_k) 
    \subseteq \bigsqcup_{x \in P \cap L_k} (x+B_k) \subseteq P^{+\delta},\]
    because $(P \setminus \partial P) \cap L_k \subseteq Z \subseteq P \cap L_k$. The proof is complete.
\end{proof}

\subsection{Proof of Theorem~\ref{thm-main-in-the-introduction}\ref{thm-main-2-in-the-introduction}}
Return to the setting of Theorem~\ref{thm-main-in-the-introduction}\ref{thm-main-2-in-the-introduction}.
Note that in the case $r=1$, that is, $W$ is of affine type $A_1$, Theorem~\ref{thm-main-in-the-introduction}\ref{thm-main-2-in-the-introduction} holds trivially since $S_k (z) \equiv 1$. 
Therefore, we assume $r \ge 2$.

For $z \in [0, \ell(t_\lambda)]$, we define 
\[P^\lambda_z : = \operatorname{ht}^{-1} (z) = P^\lambda \cap H_{2\rho, z}.\]
Then $P^\lambda_z$ is a polytope in the $(r-1)$-dimensional Euclidean space $H_{2\rho, z}$ (see Lemma~\ref{lem-intersect-face}\ref{lem-intersect-face-1}).  
Moreover, since $P^\lambda$ is convex of dimension $r$, we have $\dim P^\lambda_z = r - 1$ except for the two extremal cases $z = 0$ and $z = \ell(t_\lambda)$, where $P^\lambda_0$ and $P^\lambda_{\ell(t_\lambda)}$ are both single points.
As in Section~\ref{subsec-Eud-geom}, we define
\begin{equation*}
    P^{\lambda, + \delta}_z := \{x \in H_{2 \rho, z} \mid \operatorname{d}(x, P^\lambda_z) < \delta \}, \quad
    P^{\lambda, - \delta}_z := \{x \in P^\lambda_z \mid \operatorname{d}(x, \partial P^\lambda_z) > \delta \}.
\end{equation*}
Then the set $P^{\lambda, + \delta}_z \setminus \overline{P^{\lambda, - \delta}_z}$ is the open neighbourhood
\[\mathcal{N}(\partial P^\lambda_z, \delta) := \{x \in H_{2 \rho, z} \mid \operatorname{d} (x, \partial P^\lambda_z) < \delta\}\]
of $\partial P^\lambda_z$ in~$H_{2 \rho, z}$.
Remember that $\mathcal{N}(\partial P^\lambda_z, \delta)$ is an open subset in the ($r-1$)-dimensional space $H_{2\rho, z}$.

\begin{lemma} \label{lem-Vol-P+-delta}
    For any real number $\varepsilon > 0$, there exists $\delta > 0$ small enough such that for any $z \in [0, \ell(t_\lambda)]$, we have
    \[\operatorname{Vol}_{r-1} (\mathcal{N}(\partial P^\lambda_z, \delta)) < \varepsilon.\]
    In particular, since $P^{\lambda, - \delta}_z \subseteq P^\lambda_z \subseteq P^{\lambda, + \delta}_z$, we have
    \[ \operatorname{Vol}_{r-1} (P^{\lambda, +\delta}_z) - \operatorname{Vol}_{r-1} (P^\lambda_z) < \varepsilon,\] 
    \[\operatorname{Vol}_{r-1} (P^\lambda_z) - \operatorname{Vol}_{r-1} (P^{\lambda,  -\delta}_z) < \varepsilon.\]
\end{lemma}

\begin{proof}
    For each $i = 0, 1, \dots, r-2$, let $\mathcal{F}_{z,i}$ be the set of $i$-dimensional open faces of $P^\lambda_z$, and $\mathcal{F}^\lambda$ be the set of all open faces of $P^\lambda$.
    By Lemma~\ref{lem-intersect-face}, each open face of $P^\lambda_z$ is an intersection of some open face in~$\mathcal{F}^\lambda$ with $H_{2\rho, z}$.
    Therefore, it holds that
    \[\sum_{0 \le i \le r-2} \lvert \mathcal{F}_{z,i} \rvert \le \lvert \mathcal{F}^\lambda \rvert, \text{ for any } z \in [0, \ell(t_\lambda)].\]
    Moreover, since $P^\lambda$ is a bounded area, there exists a uniform bound $M \in \mathbb{R}_+$ (independent of $z$) such that
    \[\operatorname{Vol}_{i} (F^\circ) \le M \text{ for any $z$ and any } F^\circ \in \mathcal{F}_{z,i}.\]
    Then by Proposition~\ref{prop-neighbour-volume}, for any $z \in [0, \ell(t_\lambda)]$, 
    \begin{align*}
        \operatorname{Vol}_{r-1} (\mathcal{N}(\partial P^\lambda_z, \delta)) & \le \sum_{0 \le i \le r-2}  \Bigl( \sum_{F^\circ \in \mathcal{F}_{z,i}} \operatorname{Vol}_i (F^\circ) \Bigr) c_{r-1-i} \delta^{r-1-i} \\
        & \le M \lvert \mathcal{F}^\lambda \rvert  \sum_{0 \le i \le r-2} c_{r-1-i}  \delta^{r-1-i} \\
        & = M \lvert \mathcal{F}^\lambda \rvert  (c_1 \delta + c_2 \delta^2 + \dots + c_{r-1} \delta^{r-1}),
    \end{align*}
    where $c_{1}, \dots, c_{r-1}$ are constants independent of $z$ (see Lemma~\ref{lem-vol-cylinder}).
    The existence of the desired $\delta$ is now clear.
\end{proof}

\begin{lemma} \label{lem-Vol-Px}
    For any real number $\varepsilon > 0$, there exists $\delta > 0$ small enough such that for any $z, z' \in [0, \ell(t_\lambda)]$ with $\lvert z- z' \rvert < \delta$, we have
    \[\lvert \operatorname{Vol}_{r-1} (P^\lambda_z) - \operatorname{Vol}_{r-1} (P^\lambda_{z'}) \rvert < \varepsilon.\]
\end{lemma}

\begin{proof}
    This is because the function $z \mapsto \operatorname{Vol}_{r-1} (P^\lambda_z)$ is a continuous function on~$[0, \ell(t_\lambda)]$, and hence uniformly continuous.
\end{proof}

Recall that $\{\alpha_1\spcheck, \dots, \alpha_r\spcheck\}$ is a basis of $E$, and 
\begin{equation} \label{eq-basis-subspace}
    \{\alpha_i\spcheck - \alpha_1\spcheck \mid i = 2, \dots, r \}
\end{equation}
is a basis of $H_{2\rho, 0}$, as well as a basis for the lattice $\mathbb{Z} \Phi\spcheck \cap H_{2 \rho, 0}$.
For arbitrary $z \in \mathbb{R}$, if we choose and fix an origin $o \in H_{2 \rho, z}$ for the Euclidean space $H_{2 \rho, z}$, then the set of vectors in \eqref{eq-basis-subspace} is also a basis for $H_{2 \rho, z}$.
If moreover $z \in 2 \mathbb{Z}$  and  $o$ is taken from $ \mathbb{Z} \Phi\spcheck \cap  H_{2 \rho, z}$, then \eqref{eq-basis-subspace} is also a basis for the lattice 
\[ \mathbb{Z} \Phi\spcheck \cap H_{2 \rho, z} = \left\{o + \sum_{1 \le i \le r} n_i \alpha_i\spcheck \xmiddle| \sum_{1 \le i \le r} n_i = 0 \right\}.\]
We define 
\begin{equation} \label{eq-def-B}
    B := \left\{ \sum_{2 \le i \le r} c_i (\alpha_i\spcheck - \alpha_1\spcheck) \xmiddle| c_2, \dots, c_r \in [0,1) \right\}
\end{equation}
to be the parallelotope in~$H_{2\rho, 0}$, and 
\[B_k := \frac{1}{k} B = \left\{ \sum_{2 \le i \le r} \frac{c_i}{k} (\alpha_i\spcheck - \alpha_1\spcheck) \xmiddle| c_2, \dots, c_r \in [0,1) \right\}\]
for any $k \in \mathbb{N}_{>0}$.
Then 
\[\operatorname{Vol}_{r-1} (B_k) = \frac{1}{k^{r-1}} \operatorname{Vol}_{r-1} (B).\]

\begin{remark}\label{rem-vol-independent-of-choice}
    There are many possible choices of basis for the lattice $\mathbb{Z} \Phi\spcheck \cap H_{2 \rho, 0}$.
    For example, 
    $\{\alpha_i\spcheck - \alpha_{i+1}\spcheck \mid i = 1, \dots, r-1\}$
    is another choice.
    Although the shape and the diameter of the corresponding parallelotope $B$ both depend on the chosen basis, the volume $\operatorname{Vol}_{r-1}(B)$ is independent of this choice. 
\end{remark}

As observed in Remark~\ref{rmk-b-i-lambda}, the numbers $b_i^{k\lambda}$ and $b_{i,+}^{k\lambda}$ are zero if $i$ is odd.
On the other hand, for $i$ even, we have the following estimation:
\begin{lemma} \label{lem-bVol-le-VolP}
    For any $\delta>0$, any integer $k > \operatorname{diam}(B) / \delta$, and any $i \in 2 \mathbb{Z}$, we have
    \[\operatorname{Vol}_{r-1} (P^{\lambda, -\delta}_{\frac{i}{k}}) \le b_{i,+}^{k \lambda} \cdot \operatorname{Vol}_{r-1} (B_k) \le b_{i}^{k \lambda} \cdot \operatorname{Vol}_{r-1} (B_k) \le \operatorname{Vol}_{r-1} (P^{\lambda, +\delta}_{\frac{i}{k}}).\]
\end{lemma}
\begin{proof}
    Let $i \in 2\mathbb{Z}$ be arbitrary.
    Recall that 
    \begin{align*}
        b_i^{k\lambda} & = \operatorname{Card} ( P^{k \lambda} \cap \mathbb{Z} \Phi\spcheck \cap H_{2 \rho, i}) \\
        & = \operatorname{Card} (P^\lambda \cap \frac{1}{k} \mathbb{Z} \Phi\spcheck \cap H_{2 \rho, \frac{i}{k}}) \\ 
        & = \operatorname{Card} (P^\lambda_{\frac{i}{k}} \cap L_{i,k}),
    \end{align*}
    where $L_{i,k} := (\frac{1}{k} \mathbb{Z} \Phi\spcheck) \cap H_{2 \rho, \frac{i}{k}}$.
    Let $o \in L_{i,k}$ be arbitrary.
    Then 
    \begin{equation*}
        L_{i,k} = \left\{o + \sum_{2 \le j \le r} \frac{n_j}{k} (\alpha_j\spcheck - \alpha_1\spcheck) \xmiddle| n_2, \dots, n_r \in \mathbb{Z} \right\}
    \end{equation*}
    is a lattice in the hyperplane $H_{2 \rho, \frac{i}{k}}$.
    Moreover, $o + B_k$ is a fundamental domain for the translation action of $L_{i,k}$ on~$H_{2 \rho, \frac{i}{k}}$.
    
    Let integer $k > \operatorname{diam}(B) / \delta$ be arbitrary.
    Then, by Lemma~\ref{lem-union-box-nbhd}, we have
    \[  \bigsqcup_{x \in P^\lambda_{\frac{i}{k}} \cap L_{i,k}} (x + B_k)  \subseteq P^{\lambda, + \delta}_{\frac{i}{k}}.\]
    Therefore, 
    \[b_i^{k \lambda} \cdot \operatorname{Vol}_{r-1} (B_k) \le \operatorname{Vol}_{r-1} (P^{\lambda, + \delta}_{\frac{i}{k}}).\]
    
    Similarly, 
    \[b_{i,+}^{k\lambda} = \operatorname{Card} ( P^{\lambda}_+ \cap \frac{1}{k} \mathbb{Z} \Phi\spcheck \cap H_{2 \rho, \frac{i}{k}}) = \operatorname{Card} ( P^{\lambda}_+ \cap L_{i,k}) .\]
    Notice that 
    \[ P^{\lambda}_{\frac{i}{k}} \setminus \partial P^{\lambda}_{\frac{i}{k}} \subseteq P^{\lambda}_+ \cap H_{2 \rho, \frac{i}{k}} \subseteq P^{\lambda}_{\frac{i}{k}}.\]
    By Lemma~\ref{lem-union-box-nbhd} again, 
    \[ P^{\lambda, - \delta}_{\frac{i}{k}} \subseteq \bigsqcup_{x \in P^\lambda_{+} \cap L_{i,k}} (x + B_k),   \]
    and hence 
    \[\operatorname{Vol}_{r-1} (P^{\lambda, - \delta}_{\frac{i}{k}}) \le b_{i,+}^{k \lambda} \cdot \operatorname{Vol}_{r-1} (B_k).\]
    This completes the proof.
\end{proof}

Recall that 
\begin{equation*}
    b_{j,f} = \operatorname{Card} \{w \in W_f \mid \ell(w) = j \}
\end{equation*}
is the coefficient of $q^j$ in~$\pi_f(q)$.
The following lemma is classical:
\begin{lemma} \label{lem-even=odd}
\begin{equation*}
    \sum_{\substack{0 \le j \le \lvert \Phi^+ \rvert\\  j\in 2\mathbb{Z}}} b_{j,f} =     \sum_{\substack{0 \le j \le \lvert \Phi^+ \rvert\\  j\in 2\mathbb{Z}+1}} b_{j,f} = \frac{\lvert W_f \rvert}{2}.
\end{equation*}
\end{lemma}

The following proposition states that the sequence of step functions $(S_k(z))_k$ converges uniformly to the continuous function 
\[z \mapsto \frac{\lvert W_f \rvert}{2 \operatorname{Vol}_{r-1} (B)} \operatorname{Vol}_{r-1}(P^\lambda_z),\]
where $B$ is as in Equation  \eqref{eq-def-B}.

\begin{proposition} \label{prop-unif-conv}
    For any $\varepsilon > 0$, there exists a positive integer $K$, such that for any integer $k > K$ and any $z \in [0, \ell(t_\lambda)]$, we have \[\lvert S_k(z) - c \operatorname{Vol}_{r-1}(P^\lambda_z) \rvert < \varepsilon,\] where
    \[c = \frac{\lvert W_f \rvert}{2 \operatorname{Vol}_{r-1}(B)}\]
    is a constant.
\end{proposition}

\begin{proof}
Let $\varepsilon > 0$ be arbitrary, and $\varepsilon' := \frac{\operatorname{Vol}_{r-1} (B)}{\lvert W_f \rvert} \varepsilon$.
By Lemma~\ref{lem-Vol-P+-delta}, there exists $\delta > 0$ small enough such that 
\begin{equation} \label{eq-main-2}
    \operatorname{Vol}_{r-1} (P^{\lambda, +\delta}_z) < \operatorname{Vol}_{r-1}(P^\lambda_z) + \varepsilon' \quad \text{for any $z \in [0, \ell(t_\lambda)]$,}
\end{equation}
and 
\begin{equation} \label{eq-main-2p}
    \operatorname{Vol}_{r-1} (P^{\lambda, -\delta}_z) > \operatorname{Vol}_{r-1}(P^\lambda_z) - \varepsilon' \quad \text{for any $z \in [0, \ell(t_\lambda)]$.}
\end{equation}
For such $\delta$, by Lemma~\ref{lem-bVol-le-VolP}, there exists a positive integer $K$ such that 
\begin{equation} \label{eq-main-3}
    b_{i}^{k\lambda} \operatorname{Vol}_{r-1} (B_k) \le \operatorname{Vol}_{r-1}(P^{\lambda, +\delta}_{\frac{i}{k}}) \quad \text{for any $k > K$ and any $i \in 2 \mathbb{Z}$,}
\end{equation}
and
\begin{equation} \label{eq-main-3p}
    b_{i,+}^{k\lambda} \operatorname{Vol}_{r-1} (B_k) \ge \operatorname{Vol}_{r-1}(P^{\lambda, -\delta}_{\frac{i}{k}}) \quad \text{for any $k > K$ and any $i \in 2 \mathbb{Z}$.}
\end{equation}
Moreover, by Lemma~\ref{lem-Vol-Px}, there exists $\delta' > 0$ such that for any $z, z' \in [0, \ell(t_\lambda)]$ with $\lvert z - z' \rvert < \delta'$, we have $\lvert \operatorname{Vol}_{r-1} (P^\lambda_z) - \operatorname{Vol}_{r-1} (P^\lambda_{z'}) \rvert < \varepsilon'$.
We can choose the integer $K$ large enough such that $\frac{\lvert \Phi^+ \rvert}{K} < \delta'$.
Then we have
\begin{equation}\label{eq-main-4}
  \begin{gathered}
      \lvert \operatorname{Vol}_{r-1} (P^\lambda_z) - \operatorname{Vol}_{r-1} (P^\lambda_{\frac{i+j}{k}}) \rvert < \varepsilon' \\ \text{for any $k > K$, $i , j \in \mathbb{Z}$ with $z \in [ \tfrac{i}{k}, \tfrac{i+1}{k} )$  and $0 \le j \le \lvert \Phi^+ \rvert$.}
  \end{gathered}
\end{equation}

Now, let $\delta$, $K$ be as above, and $k > K$, $z \in [0, \ell(t_\lambda)]$ be arbitrary.
There exists a unique $i \in \{0, 1, \dots, k \ell(t_\lambda)\}$ such that $z \in [ \frac{i}{k}, \frac{i+1}{k})$.
Consider the coefficient of $q^i$ in Corollary~\ref{cor-lattice-formula}, we have the inequalities
\begin{equation} \label{eq-main-1}
    \sum_{0 \le j \le \lvert \Phi^+ \rvert} b_{j,f} b_{i+j, +}^{k\lambda}  \le \prescript{f}{}{b}_i^{t_{k\lambda}} \le \sum_{0 \le j \le \lvert \Phi^+ \rvert} b_{j,f} b_{i+j}^{k\lambda}.
\end{equation}
Here the numbers $b_{i+j, +}^{k\lambda}$ and $b_{i+j}^{k\lambda}$ are regarded as zero if $i + j > \ell(t_{k\lambda})$.
Notice that $b^{k \lambda}_{i+j} = 0$ whenever $i + j$ is odd.
We have the following inequalities:
\begin{align*}
    &\prescript{f}{}{b}_i^{t_{k\lambda}}  \operatorname{Vol}_{r-1} (B_k)  \\ 
    \text{(by Inequality \eqref{eq-main-1})} \le & \sum_{0 \le j \le \lvert \Phi^+ \rvert} b_{j,f} b_{i+j}^{k\lambda}  \operatorname{Vol}_{r-1} (B_k) \\ 
      =& \sum_{\substack{0 \le j \le \lvert \Phi^+ \rvert \\ i + j \in 2 \mathbb{Z}}} b_{j,f} b_{i+j}^{k\lambda}  \operatorname{Vol}_{r-1} (B_k) \\
    \text{(by Inequality \eqref{eq-main-3})}  \le &  \sum_{\substack{0 \le j \le \lvert \Phi^+ \rvert \\ i + j \in 2 \mathbb{Z}}} b_{j,f} \operatorname{Vol}_{r-1} (P^{\lambda, + \delta}_{\frac{i+j}{k}}) \\
    \text{(by Inequality \eqref{eq-main-2})} < & \sum_{\substack{0 \le j \le \lvert \Phi^+ \rvert \\ i + j \in 2 \mathbb{Z}}} b_{j,f} \Bigl(\operatorname{Vol}_{r-1} (P^{\lambda}_{\frac{i+j}{k}}) + \varepsilon' \Bigr) \\
    \text{(by Inequality \eqref{eq-main-4})} < & \sum_{\substack{0 \le j \le \lvert \Phi^+ \rvert \\ i + j \in 2 \mathbb{Z}}} b_{j,f} \Bigl(\operatorname{Vol}_{r-1} (P^{\lambda}_{z}) + 2\varepsilon' \Bigr) \\
    \text{(by Lemma~\ref{lem-even=odd})} = & \frac{\lvert W_f \rvert}{2} \Bigl( \operatorname{Vol}_{r-1} (P^{\lambda}_{z}) + 2\varepsilon' \Bigr).
\end{align*}
At last, since $\operatorname{Vol}_{r-1}(B_k) = \frac{1}{k^{r-1}} \operatorname{Vol}_{r-1}(B)$, we obtain
\[S_{k}(z) = \frac{1}{k^{r-1}} \prescript{f}{}{b}_i^{t_{k\lambda}} < \frac{\lvert W_f \rvert}{2 \operatorname{Vol}_{r-1}(B)} \Bigl( \operatorname{Vol}_{r-1} (P^{\lambda}_{z}) + 2\varepsilon' \Bigr) = c \operatorname{Vol}_{r-1} (P^{\lambda}_{z}) + \varepsilon.\]

Similarly, by Inequalities \eqref{eq-main-1}, \eqref{eq-main-3p}, \eqref{eq-main-2p}, \eqref{eq-main-4} and Lemma~\ref{lem-even=odd}, we have
\[S_k(z) > c \operatorname{Vol}_{r-1} (P^{\lambda}_{z}) - \varepsilon,\]
and we are done.
\end{proof}

Proposition~\ref{prop-unif-conv} proves Theorem~\ref{thm-main-in-the-introduction}\ref{thm-main-2-in-the-introduction} up to the following equality:
\begin{equation} \label{eq-main-eq}
    \frac{1}{\operatorname{Vol}_r(A_+) \cdot \lVert \rho \rVert} = \frac{\lvert W_f \rvert}{\operatorname{Vol}_{r-1} (B)},
\end{equation}
where $B$ is as in Equation  \eqref{eq-def-B}.
We prove this equality in the next subsection.

\subsection{Proof of Equation  (\ref{eq-main-eq})} \label{subsec-pf-eq}
In this subsection, we provide two independent proofs of Equation \eqref{eq-main-eq}.
The first proof uses the convergence results (Theorem~\ref{thm-main-in-the-introduction}\ref{thm-main-1-in-the-introduction} and Proposition~\ref{prop-unif-conv}), while the second only uses Lie-theoretical information.

\begin{proof}[First proof]
    By our definitions of the discrete measure $\mathfrak{m}_k$ and the step function $S_k(z)$, 
    we have
    \begin{equation} \label{eq-pf-eq-1-1}
    \begin{aligned}
        \mathfrak{m}_k([0, \ell(t_\lambda)]) 
        & = \frac{1}{k^r} \sum_{0 \le i\le \ell(t_{k\lambda})} \prescript{f}{}{b}_i^{t_{k\lambda}} \\
        & = \int_0^{\ell(t_\lambda)} S_k(z) \, \mathrm{d}z + \frac{\prescript{f}{}{b}_{\ell(t_{k\lambda})}^{t_{k\lambda}}}{k^r} \\
        & = \int_0^{\ell(t_\lambda)} S_k(z) \, \mathrm{d}z + \frac{1}{k^r}.
    \end{aligned}
   \end{equation}
    
    By the weak convergence result (Theorem~\ref{thm-main-in-the-introduction}\ref{thm-main-1-in-the-introduction}), we have
    \begin{equation} \label{eq-pf-eq-1-3}
        \lim_{k \to \infty} \mathfrak{m}_k([0,\ell(t_\lambda)]) = \frac{\operatorname{Vol}_r(P^\lambda)}{\operatorname{Vol}_r(A_+)}.
    \end{equation}
    While by the uniform convergence (Proposition~\ref{prop-unif-conv}), 
    we have
    \begin{equation} \label{eq-pf-eq-1-4}
      \begin{aligned}
          \lim_{k \to \infty} \int_0^{\ell(t_\lambda)} S_k(z) \, \mathrm{d}x & = \int_0^{\ell(t_\lambda)} \frac{\lvert W_f \rvert \cdot \lVert 2\rho \rVert}{2\operatorname{Vol}_{r-1} (B)} g(z) \, \mathrm{d}z \\
          & = \frac{\lvert W_f \rvert \cdot \lVert \rho \rVert \cdot \operatorname{Vol}_r(P^\lambda)}{\operatorname{Vol}_{r-1} (B)}.
      \end{aligned}
    \end{equation}
    Because of Equation  \eqref{eq-pf-eq-1-1}, the two limits in \eqref{eq-pf-eq-1-3} and \eqref{eq-pf-eq-1-4} must coincide.
    This gives Equation  \eqref{eq-main-eq}.
\end{proof}

\begin{proof}[Second proof]
    Retain the notation $D$ from Section~\ref{subsec-measure}.
    Then 
    \[\overline{D} = \left\{\sum_{1 \le i \le r} a_i \alpha_i \spcheck \in E \xmiddle| 0 \le a_i \le 1 \text{ for all } i \right\}\]
    is the closed parallelotope spanned by simple coroots.
    Let 
    \[\Delta := \operatorname{Conv}\{0, \alpha_1 \spcheck, \dots, \alpha_r\spcheck\}\]
    be the simplex with vertices $\{0, \alpha_1 \spcheck, \dots, \alpha_r\spcheck\}$.
    It is well known that 
    \begin{equation} \label{eq-pf-eq-2-1}
        \operatorname{Vol}_r (\overline{D}) = r ! \cdot \operatorname{Vol}_r(\Delta).
    \end{equation}
    
    Let $Z := H_{\rho, 1} \cap \overline{D}$.
    Then, $Z$ is the face of $\Delta$ containing $\alpha_1 \spcheck, \dots, \alpha_r \spcheck$ since $(\rho | \alpha_i\spcheck) = 1$.
    The distance from $0$ to $H_{\rho,1}$ equals $\frac{1}{\lVert \rho \rVert}$, and we have
    \begin{equation} \label{eq-pf-eq-2-2}
        \operatorname{Vol}_r(\Delta) = \frac{1}{r} \cdot \frac{1}{\lVert \rho \rVert} \cdot \operatorname{Vol}_{r-1} (Z).
    \end{equation}
    Notice that
    \begin{align*}
        Z & = \left\{ \sum_{1 \le i \le r} a_i \alpha_i\spcheck \xmiddle| 0 \le a_i \le 1 \text{ for all $i$, and } \sum_i a_i = 1 \right\} \\
        & = \left\{ \alpha_1\spcheck + \sum_{2 \le i \le r} a_i (\alpha_i\spcheck - \alpha_1 \spcheck) \xmiddle| 0 \le a_i \le 1 \text{ for all $i$, and } \sum_{2 \le i \le r} a_i \le 1\right\}.
    \end{align*}
    Therefore, $Z$ is a translation of the $(r-1)$-dimensional simplex with vertices 
    \[0, \alpha_2\spcheck - \alpha_1 \spcheck, \dots, \alpha_r\spcheck - \alpha_1 \spcheck.\]
    Note also that $B$ is the parallelotope spanned by $\alpha_2\spcheck - \alpha_1 \spcheck, \dots, \alpha_r\spcheck - \alpha_1 \spcheck$.
    Therefore
    \begin{equation} \label{eq-pf-eq-2-3}
        (r-1)! \cdot\operatorname{Vol}_{r-1} (Z) = \operatorname{Vol}_{r-1}(B).
    \end{equation}
    Combining Equations \eqref{eq-pf-eq-2-1}, \eqref{eq-pf-eq-2-2}, \eqref{eq-pf-eq-2-3}, and Lemma~\ref{lem-vol-D}, we have
    \[\frac{\operatorname{Vol}_{r-1} (B)}{\lVert \rho \rVert} = \operatorname{Vol}_r(\overline{D}) = \operatorname{Vol}_r (D) = \lvert W_f \rvert \cdot \operatorname{Vol}_r(A_+)\]
    which gives Equation \eqref{eq-main-eq}.
\end{proof}

\begin{remark} \label{rmk-deduce}
    The weak convergence result (Theorem~\ref{thm-main-in-the-introduction}\ref{thm-main-1-in-the-introduction}) can also be deduced from the uniform convergence (Proposition~\ref{prop-unif-conv}) and Equations \eqref{eq-density} and \eqref{eq-main-eq}.
\end{remark}

\section{General dominant elements} \label{sec-general}

Theorem~\ref{thm-main-in-the-introduction}\ref{thm-main-2-in-the-introduction} can be extended to general dominant lower intervals.
For $y\in \prescript{f}{}{W}$, we define  $\prescript{f}{}{b}_i^{y}:= \operatorname{Card} \left\{x \in \prescript{f}{}{W} \xmiddle| x \le y, \ell(x) = i \right\}$.
\begin{theorem}
    Let $\lambda \in \mathbb{Z} \Phi \spcheck \cap \overline{C_+}$.
    Suppose $(w_k)_k$ is a sequence of elements in~$\prescript{f}{}{W}$ such that $w_k \in t_{k\lambda}W_f$ for each $k$.
    We define the step functions $S_k$ by
    \begin{equation*}
        S_k(z) := k^{-(r-1)} \prescript{f}{}{b}_i^{w_k},
    \end{equation*}
    whenever $z \in \left[ \frac{i}{k}, \frac{i+1}{k} \right)$.
    Then $(S_k)_k$ converges uniformly to
    $g / \operatorname{Vol}_r(A_+)$.
\end{theorem} 
\begin{proof}
Let $\lambda \in \mathbb{Z} \Phi \spcheck \cap \overline{C_+}$ 
and $J := \{i \mid 1 \le i\le r, (\lambda | \alpha_i) \ne 0\}$.
Then, $\lambda = \sum_{i \in J} a_i \varpi_i\spcheck$ where each $a_i \in \mathbb{R}_{>0}$ and the $\varpi_i\spcheck$'s are the fundamental coweights, that is, $(\varpi_i\spcheck | \alpha_j) = \delta_{ij}$.
Note that each $\varpi_i\spcheck$ is a rational combination of simple coroots.
Therefore, there is a positive integer $n$ such that $\lambda_0 := n\sum_{i \in J} \varpi_i\spcheck \in \mathbb{Z} \Phi\spcheck$.
For this $n$, there is $k_0 \in \mathbb{N}$ such that for any $k \ge k_0$ we have 
\begin{equation*}
k\lambda - \lambda_0 \in \sum_{i \in J} \mathbb{R}_{>0} \varpi_i\spcheck.
\end{equation*}
In particular, $\lambda_0, k\lambda - \lambda_0 \in \overline{C_+}$ and $\prescript{\lambda}{}{W_f} = \prescript{k\lambda}{}{W_f} = \prescript{\lambda_0}{}{W_f}$.

For any $k \ge k_0$, we write $w_k = t_{k\lambda} u_k$ where $u_k \in W_f$.
By Lemma~\ref{lem-lattice-onwall}, the assumption $w_k \in \prescript{f}{}{W}$ forces that $u_k \in \prescript{\lambda}{}{W_f}$.
Moreover, we have $w_k \le t_{k\lambda}$.
On the other hand, by Lemmas~\ref{lem-lattice-onwall} and~\ref{lem-length-trans}, we have
\begin{equation*}
    \ell(w_k)  = \ell(t_{k\lambda}) - \ell(u_k) 
      = \ell(t_{k\lambda - \lambda_0}) + \ell(t_{\lambda_0}) - \ell(u_k) 
      = \ell(t_{k\lambda - \lambda_0}) + \ell(t_{\lambda_0} u_k).
\end{equation*}
Note that $w_k = t_{k\lambda - \lambda_0} t_{\lambda_0} u_k$.
Therefore, $t_{k\lambda - \lambda_0} \le w_k \le t_{k\lambda}$.
Hence, for all $i$ we have
\[\prescript{f}{}{b}_i^{t_{k\lambda-\lambda_0}} \le \prescript{f}{}{b}_i^{w_k} \le \prescript{f}{}{b}_i^{t_{k\lambda}}.\]
Let $S'_k$ and $S''_k$ be the corresponding sequences of step functions associated with the sequences $(\prescript{f}{}{b}_i^{t_{k\lambda-\lambda_0}})_i$ and $(\prescript{f}{}{b}_i^{t_{k\lambda}})_i$, respectively.
By Theorem~\ref{thm-main-in-the-introduction}\ref{thm-main-2-in-the-introduction}, $(S''_k)_k$ converges uniformly to $g / \operatorname{Vol}_r(A_+)$.
While for $(S'_k)_k$, using the same arguments as in Section~\ref{sec-main-unifm-convg}, it can be proved that $(S'_k)_k$ converges uniformly to $g / \operatorname{Vol}_r(A_+)$ as well.
We omit the details.
\end{proof}

\bibliographystyle{amsplain}
\bibliography{template} 

\end{document}